\documentclass[12pt,letterpaper]{scrartcl}
\usepackage{import}
\usepackage{preamble}
\usepackage{standalone}

\title{Non-semisimple CFT/TFT correspondence I: General setup}

\author{
	Aaron Hofer$^a$\qquad
	Ingo Runkel$^b$\\[0.3cm]
	 \\   
	$^a$\normalsize\slshape Max-Planck-Institut für Mathematik,\\
	\normalsize\slshape Vivatsgasse 7, 53111 Bonn, Germany \\
    \normalsize{\texttt{\href{mailto:hofer@mpim-bonn.mpg.de}{hofer@mpim-bonn.mpg.de}}}
    \\
    \\  
	$^b$\normalsize\slshape Fachbereich Mathematik, Universit\"{a}t Hamburg,\\
	\normalsize\slshape Bundesstra{\ss}e 55, 20146 Hamburg, Germany\\
    \normalsize{\texttt{\href{mailto:ingo.runkel@uni-hamburg.de}{ingo.runkel@uni-hamburg.de}}}
	}

\date{}

\begin{document}
\maketitle
\begin{abstract}
\noindent We extend the TFT construction of CFT correlators of \autocite{FFFS2002,FRSI,FjFRS} to so-called finite logarithmic CFTs for which the algebraic input data is no longer semisimple but still finite. More specifically, starting from the data of a chiral CFT given in the form of a not necessarily semisimple modular tensor category $\calc$ we use a three dimensional topological field theory
with surface defects based on the surgery TFT of \autocite{DGGPR19} to construct a full CFT as a braided monoidal oplax natural transformation.
    
We make our construction explicit in the example of the transparent surface defect, resulting in the so-called Cardy case. In particular, we consider topological line defects and their action on bulk fields in these logarithmic CFTs, providing a source of examples for non-invertible and non-semisimple topological symmetries.
\end{abstract}

\thispagestyle{empty}

\newpage

{\small
\tableofcontents
}

\newpage

\section{Introduction}
Even though $2$d CFTs have been studied by mathematicians for decades, a rigorous construction, in the form of a complete set of consistent correlation functions, still remains elusive for a large class of theories.
Instead of tackling the problem of constructing a $2$d CFT directly, the situation becomes more tractable by splitting the problem into a complex-analytic/algebro-geometric and a purely algebraic/topological part. For the class of so-called \emph{rational} CFTs the work of \autocite{FFFS2002,FRSI,FjFRS} completely solved the second part, using three dimensional topological field theories (TFTs), over twenty years ago. Very recently, a great step towards combining the second with the first part, again in the rational setting, was achieved in \autocite{DW25modfuncVOA}.  

In this article, we will present a rigorous solution of the second part for the class of so-called \emph{finite logarithmic} CFTs, which includes all rational theories as well. More precisely, we will extend the TFT construction of rational CFT correlators of \autocite{FFFS2002,FRSI,FjFRS}, to the setting where the algebraic input data is no longer semisimple but still finite. The next section is devoted to review the general ideas behind the TFT construction of CFT correlators and motivate the need to go beyond semisimplicity.  

We also want to mention that a handful of full $2$d CFTs have been constructed directly using probabilistic methods, see e.g.\ \autocite{GKR24LiouRev} and references therein for a review of recent work in this direction. These theories are not covered by our approach as the corresponding algebraic structures contain an infinite number of simple objects.
\subsection{Background and previous results}
The complex-analytic part corresponds to the study of \emph{chiral} CFTs which are defined on complex curves and there are several approaches to study chiral CFTs rigorously. For our purposes the one based on vertex operator algebras (VOAs) will be the most suitable. In this setting a chiral CFT is encoded by a VOA $\calV$, its representation category $\mathrm{Rep}(\calV)$, and a modular functor $\mathrm{Bl}_\calV$, usually constructed from the VOA. A modular functor is, colloquially speaking, a systematic assignment of mapping class group representations to surfaces which is compatible with gluing along boundaries. In the context of chiral CFTs the modular functor encodes the monodromy of the so-called conformal block spaces, which are solution spaces of certain differential equations, which the correlation functions need to obey \autocite{FrieShen87CFT,Segal1988cftmodfunc,MS1989cft,TUY1989CFT,Tillmann1998modfunc,AU2007ModFunc,FBZ2001VOA,DGT19confblocks,DGT19factori,GZ24converge,GZ25anConfBl}. 

In this setting, the surfaces are equipped with a complex structure and the conformal block spaces are expected to form a vector bundle with projectively flat connection over the moduli space of complex curves. The mapping class group action corresponds to the monodromy of this bundle by the Riemann-Hilbert correspondence.

In this article we will not discuss the complex-analytic part of the CFT construction further, see \autocite{FSWY23algCFT} for an overview and further references. Instead, we will focus on the second part, which is concerned with \emph{full} CFT, or more precisely with how to construct a full CFT out of a given chiral one. Full CFTs are defined on conformal manifolds, possibly with boundary. Moreover, one can further consider stratifications of the manifold. In this setting $1$-dimensional strata are used to describe (topological) line defects between possibly different full CFTs living on the $2$-strata. We will call such a $2$-manifold a \emph{world sheet}. Topological defects generalise the notion of ordinary symmetries and can be used to understand further phenomena, such as dualities, on a conceptual level \autocite{FFRS2007defects}.

Given a chiral CFT in the form of a VOA $\calV$, its representation category $\mathrm{Rep}(\calV)$, and its modular functor $\mathrm{Bl}_\calV$, a corresponding full CFT consists of two extra pieces of data, the \emph{field content}, and a \emph{consistent system of correlators}. The field content describes the state space of the full CFT on the interval as well as on the circle. These spaces are equipped with an action of $\calV$ for the interval or $\calV\otimes_\kk\calV$ for the circle, i.e.\ they are objects in $\mathrm{Rep}(\calV)$ or $\mathrm{Rep}(\calV\otimes_\kk\calV)$. A consistent system of correlators is an assignment of a world sheet $\frakS$ to an element in the conformal block space of its complex double $\mathrm{Bl}_\calV(\widehat{\frakS})$ such that certain invariance and gluing properties are satisfied. Crucially, the conformal structure of the world sheet $\frakS$ is already encoded in the modular functor $\mathrm{Bl}_\calV$. Accordingly, only the topology of $\frakS$ will be relevant which leads us to consider only the underlying \emph{topological world sheet}.

Typically, a chiral CFT does not fully determine a full CFT and one needs to specify an extra input datum. One way to obtain this extra input datum is via the connection to $3$d topological field theories (TFTs). Since the 80's it has been known that there is a close relationship between chiral CFTs and $3$d TFTs where the CFT is expected to be a boundary theory of the TFT \autocite{witten1989jones}.   
As mentioned above, this idea has been used to solve the second problem outlined above for the class of so-called rational CFTs in the series of papers \autocite{FFFS2002,FRSI,FRSII,FRSIII,FRSIV,FjFRS,FFS12}. For these CFTs the representation category $\calc := \mathrm{Rep}(\calV)$ of the VOA $\calV$ is a modular fusion category, i.e.\ a certain type of finitely \emph{semisimple} ribbon tensor category.
This type of category is precisely the algebraic input datum used in the construction of the $3$d surgery TFTs in \autocite{RT1991,Tu16}. It has long been conjectured, and recently been shown \autocite{DW25modfuncVOA}, that the topological modular functor constructed from the $3$d TFT reproduces the algebro-geometric one constructed directly from the VOA $\calV$ \autocite{FBZ2001VOA,AU2007ModFunc,DGT19factori,DGT19confblocks}. In particular, the spaces of conformal blocks of the chiral CFT are isomorphic to the state spaces of the $3$d TFT. From this we see that the correlators of the full CFT correspond to special states of the TFT, which means we should be able to encode them completely topologically.

These special states were constructed systematically in \autocite{FFFS2002,FRSI,FRSII,FRSIII,FRSIV,FjFRS,FFS12} by considering certain $3$-manifolds, the so-called \emph{connecting manifolds}. They identified the extra input datum needed with a symmetric special Frobenius algebra in $\calc$ and obtained the field content via the representation theory of this algebra inside $\calc$. Later it was pointed out in \autocite{KS11def}, and further refined in \autocite{FSV11defbicat}, that this algebra corresponds to a surface defect, i.e.\ an embedded $2$-manifold, in the TFT. 

In this article we extend this ``TFT construction of CFT correlators'' to the class of \emph{finite logarithmic} CFTs in the sense of \autocite{CG17logCFT}. In this setting, the category of VOA-modules $\calc$ is expected to still give a modular tensor category, meaning that it is no longer semisimple, but still finite. These categories naturally generalise modular fusion categories and, in particular, are still rigid. Moreover, we will only need the modular tensor category $\calc$ itself without any knowledge of a VOA realisation. 
For these categories a $3$d TFT with embedded ribbon graphs
\begin{equation}
        \rmV_\calc \colon \bord_{3,2}^{\chi}(\calc) \to \vect
\end{equation}
has been constructed in \autocite{DGGPR19} generalising the $3$d TFTs of \autocite{RT1991,Tu16}. In \autocite{HR2024modfunc} we used this TFT to construct a (full) topological modular functor in the form of a symmetric monoidal $2$-functor
\begin{equation}
\mathrm{Bl}_{\calc}\colon \bord^{\oc}_{2+\epsilon,2,1} \to \cat{P}\mathrm{rof}_{\mathbbm{k}}^{\coend \mathrm{ex}}
\end{equation} 
where $\bord^{\oc}_{2+\epsilon,2,1}$ is the open-closed bordism $(2,1)$-category and $\cat{P}\mathrm{rof}_{\mathbbm{k}}^{\coend \mathrm{ex}}$ the $2$-category of left exact profunctors, see \Cref{subsec:modfunc} for details. 

Apart from the TFT construction a purely $2$-dimensional construction of full rational CFTs via string-net models has been achieved in \autocite{FSY2023RCFTstring1}. 

\medskip

Beyond semisimplicity, the situation is less well understood. 
Early algebraic descriptions based on on the idea of reconstructing the bulk theory from the boundary theory in logarithmic CFTs were given in \autocite{GR08boundCFT,GRW10triplet,GRW14centre}
A construction of correlators in the setting of 
ribbon Hopf algebras has been carried out in \autocite{FSS12FrobHopf,FSS13CardyCartan,FSS14Hopfgenus}. 
Some of the idea's developed there have been extended to general non-semisimple modular tensor categories to study full CFTs in a model-independent manner \autocite{FS17consCorrelators,FGSS18cardy,FS21CFTmodulecats}, see also \autocite{FS19FLCFT} for a concise review of some of these aspects.

More recently an operadic approach was used to classify and construct topological modular functors \autocite{BW22modfunc,MSWY23drinmodfunc} as well as full CFTs 
\autocite{Woike24microcosm,woike2025constructioncorrelatorsfiniterigid}
(with one fixed boundary condition and without defects).
We expect the operadic approach to be equivalent to the one based on symmetric monoidal $2$-categories and 3d TFT used here and in \autocite{HR2024modfunc}, but a detailed comparison has not yet been worked out.
    
The approach based on symmetric monoidal $2$-categories is
natural from the perspective of $3$d TFTs and it has the added benefit of implementing boundary conditions as well as topological defects for the full CFTs easily. In particular, it is connected with the \emph{symmetry TFT} framework which has recently gained popularity \autocite{PSV19nonabdualgauge,KLWZZ20catSym,GK20orbigroid,ABGHS21symTFT,FMT22symTFT,KOZ22symTFT,BSN23genchargesymTFT}, see also \autocite{Schafer-Nameki23ictplecture,BBFLGPT,CDR24topdefs} for further references. 
The precise connection between the CFT/TFT correspondence and the symTFT framework is via the so-called folding trick and is explained in \autocite[Sec.\,3.5]{CDR24topdefs} as well as \autocite[Sec.\,1.1.4]{hofer25thesis}.

\subsection{Statement of the main results}\label{subsec:mainres}
Since surface defects play a key role in the TFT construction of CFT correlators we need an extension the $3$d TFT $\rmV_\calc$ which includes surface defects.
To this end let 
\begin{equation}
        \rmZ_\calc \colon \bord_{3,2}^{\chi,\mathrm{ def}}(\DD_\calc) \to \vect
\end{equation}
be a $3$d defect TFT in the sense of \autocite{CMS16deftri,CRS17deforbi}.
extending the $3$d TFT with embedded ribbon graphs
\begin{equation}
        \rmV_\calc \colon \bord_{3,2}^{\chi}(\calc) \to \vect
\end{equation}
of \autocite{DGGPR19}. This roughly means that we want to view $\bord_{3,2}^{\chi}(\calc)$ as a subcategory of $\bord_{3,2}^{\chi,\mathrm{ def}}(\DD_\calc)$ and the restriction of $\rmZ_\calc$ to this subcategory should be naturally isomorphic to $\rmV_\calc$, see \Cref{def:deftftextension} for the precise statement.
We will see that the construction of CFT correlators can be achieved independently of the specifics of the defect TFT $\rmZ_\calc$ as long as certain algebraic conditions are satisfied, see \Cref{sec:tftconstcft}. 

Using the defect data $\DD_\calc$ we define the $(2,1)$-category $\WS(\DD_\calc)$ of \emph{topological world sheets} consisting of compact, stratified and labelled $1$-manifolds as objects,
stratified and labelled bordisms with corners  as 1-morphisms
between these, and isotopy classes of diffeomorphisms as 2-morphisms.
This is done by introducing stratifications to the open-closed bordism $(2,1)$-category $\bord^{\oc}_{2+\epsilon,2,1}$ and comes with an obvious forgetful $2$-functor $\WS(\DD_\calc) \to \bord^{\oc}_{2+\epsilon,2,1}$, see \Cref{subsec:topworld} for details. An example of a $1$-morphism
in $\WS(\DD_\calc)$ is shown in Figure~\ref{fig:intro-wsh-example}.

\begin{figure}[t]
    \centering
        \begin{tikzpicture}[baseline=0cm]
    \fill[white!80!magenta,opacity=0.7]
    {[rounded corners] (-0.75,-2) -- (-0.8,-0.8) -- (-1,-0.5) -- (-2,0.5) -- (-2.2,0.8) -- (-2.25,2)} --
    (-2.25,2) arc (180:0:0.75 and 0.5) --
    (-0.75,2) arc (180:360:0.75 and 1) --
    (0.75,2) arc (180:260:0.75 and 0.5) 
    {[rounded corners] ([xshift=1.5cm,yshift=2cm] 260:0.75 and 0.5) --([xshift=1.5cm,yshift=1.8cm] 260:0.75 and 0.5) -- ([yshift=-1.8cm] 110:0.75 and 1.5) -- ([yshift=-2cm] 110:0.75 and 1.5)} --
    ([yshift=-2cm] 110:0.75 and 1.5) arc (110:180:0.75 and 1.5);
    \fill[white!70!magenta,opacity=0.7]
    ([xshift=1.5cm,yshift=2cm] 70:0.75 and 0.5) arc (70:295.9:0.75 and 0.5);
    \fill[white!80!cyan,opacity=0.8]
    ([xshift=1.5cm,yshift=2cm] 70:0.75 and 0.5) arc (70:-64.1:0.75 and 0.5);
    \fill[white!70!magenta,opacity=0.7]
    (0.75,-2) arc (0:180:0.75 and 1.5) --
    (-0.75,-2) arc (180:0:0.75 and 0.5); 
    \fill[white!50!magenta,opacity=0.7]
    {[rounded corners] (1.3,-0.2) -- (0.8,-0.75) -- (0.75,-2)} --
    (0.75,-2) arc (0:50:0.75 and 1.5)  
    {[rounded corners] ([yshift=-2cm] 50:0.75 and 1.5) --
    ([xshift=0.15cm,yshift=-1.9cm] 50:0.75 and 1.5) -- (0.9,-0.3) -- (1.3,-0.2)} 
    ;
    \fill[white!60!magenta,opacity=0.7]
    {[rounded corners] (-0.75,-2) -- (-0.8,-0.8) -- (-1,-0.5) -- (-2,0.5) -- (-2.2,0.8) -- (-2.25,2)} --
    (-2.25,2) arc (180:360:0.75 and 0.5) --
    (-0.75,2) arc (180:360:0.75 and 1) --
    (0.75,2) arc (180:260:0.75 and 0.5) 
    {[rounded corners] ([xshift=1.5cm,yshift=2cm] 260:0.75 and 0.5) --([xshift=1.5cm,yshift=1.8cm] 260:0.75 and 0.5) -- ([yshift=-1.8cm] 110:0.75 and 1.5) -- ([yshift=-2cm] 110:0.75 and 1.5)} --
    ([yshift=-2cm] 110:0.75 and 1.5) arc (110:180:0.75 and 1.5)
    ;
    \fill[white!60!cyan,opacity=0.7]
    {[rounded corners] (1.4,-0.1) --  (1.4,0.1) -- ([xshift=1.6cm,yshift=0.7cm] 70:0.75 and 0.5) --([xshift=1.5cm,yshift=2cm] 295.9:0.75 and 0.5)} --([xshift=1.5cm,yshift=2cm] 295.9:0.75 and 0.5) arc (295.9:360:0.75 and 0.5) {[rounded corners] (2.25,2) --  (2.2,0.8) -- (2,0.5) -- (1.4,-0.1)}
    ;
    \fill[white!40!cyan,opacity=0.6]
    [rounded corners] ([xshift=1.5cm,yshift=2cm] 295.9:0.75 and 0.5) -- ([xshift=1.6cm,yshift=0.7cm] 70:0.75 and 0.5) --  (1.4,0.1) [sharp corners] --   (1.4,-0.1) -- (1.3,-0.2) [rounded corners] -- (0.9,-0.3) -- ([xshift=0.15cm,yshift=-1.9cm] 50:0.75 and 1.5)  [sharp corners] -- ([yshift=-2cm] 50:0.75 and 1.5) arc (50:110:0.75 and 1.5) [rounded corners] -- ([yshift=-1.8cm] 110:0.75 and 1.5) -- ([xshift=1.5cm,yshift=1.8cm] 260:0.75 and 0.5) [sharp corners] --([xshift=1.5cm,yshift=2cm] 260:0.75 and 0.5) arc (260:295.9:0.75 and 0.5);
    \begin{scope}[very thick,decoration={
                      markings,
                      mark=at position 0.45 with {\arrow{>}}}
                     ] 
    \draw[thick,string-red!70, line cap= round, dashed,postaction={decorate}] {[rounded corners] (1.4,-0.1) -- (1.4,0.1)  -- ([xshift=1.6cm,yshift=0.7cm] 70:0.75 and 0.5) -- ([xshift=1.5cm,yshift=2cm] 70:0.75 and 0.5)};
    \end{scope}
    \draw[line cap=round]{[rounded corners] (0.75,-2) -- (0.8,-0.8) -- (1,-0.5) -- (2,0.5) -- (2.2,0.8) -- (2.25,2)}
    {[rounded corners] (-0.75,-2) -- (-0.8,-0.8) -- (-1,-0.5) -- (-2,0.5) -- (-2.2,0.8) -- (-2.25,2)}
    (0.75,2) arc (0:-180:0.75 and 1);
    \begin{scope}[very thick,decoration={
                      markings,
                      mark=at position 0.5 with {\arrow{>}}}
                     ] 
    \draw[thick,string-green,postaction={decorate}] {[rounded corners]([xshift=1.5cm,yshift=2cm] 260:0.75 and 0.5) -- ([xshift=1.5cm,yshift=1.8cm] 260:0.75 and 0.5)-- ([yshift=-1.8cm] 110:0.75 and 1.5) -- ([yshift=-2cm] 110:0.75 and 1.5) };
    \end{scope}
    \draw[thick,string-green] {[rounded corners] ([yshift=-2cm] 110:0.75 and 1.5) -- ([yshift=-1.8cm] 110:0.75 and 1.5) -- ([xshift=1.5cm,yshift=1.8cm] 260:0.75 and 0.5) --([xshift=1.5cm,yshift=2cm] 260:0.75 and 0.5)};
    \draw[thick,string-red,line cap = round] {[rounded corners] ([yshift=-2cm] 50:0.75 and 1.5)  --([xshift=0.15cm,yshift=-1.9cm] 50:0.75 and 1.5) -- (0.9,-0.3) -- (1.3,-0.2)};
    \draw[thick, string-red, line cap= round] ([xshift=1.5cm,yshift=2cm] 70:0.75 and 0.5) --  ([xshift=1.5cm,yshift=2cm] 295.9:0.75 and 0.5);
    \draw[thick,white!30!magenta] (-1.5,2) ellipse (0.75 and 0.5)
    (0.75,-2) arc (0:180:0.75 and 0.5)
    ([xshift=1.5cm,yshift=2cm] 70:0.75 and 0.5) arc (70:260:0.75 and 0.5);
    \draw[thick,cyan]([xshift=1.5cm,yshift=2cm] 70:0.75 and 0.5) arc (70:-100:0.75 and 0.5);
    \begin{scope}[very thick,decoration={
                      markings,
                      mark=at position 0.502 with {\arrow{>}}}
                     ] 
    \draw[thick,line cap=round,string-violet!40!magenta,postaction={decorate}] (0.75,-2) arc (0:180:0.75 and 1.5);
    \draw[thick,line cap=round,string-blue!50!string-green,postaction={decorate}] ([yshift=-2cm] 110:0.75 and 1.5) arc (110:180:0.75 and 1.5);
    \end{scope}
    \begin{scope}[very thick,decoration={
                      markings,
                      mark=at position 0.7 with {\arrow{>}}}
                     ] 
    \draw[thick,line cap=round,string-violet,postaction={decorate}] ([yshift=-2cm] 0:0.75 and 1.5) arc (0:50:0.75 and 1.5);
    \end{scope}
    \fill[string-violet]
    (0.75,-2) circle (0.04);
   \fill[string-blue!50!string-green] (-0.75,-2) circle (0.04);
    \fill[black]([yshift=-2cm] 110:0.75 and 1.5) circle (0.04)
    ([yshift=-2cm] 50:0.75 and 1.5) circle (0.04);
    \fill[string-red]([xshift=1.5cm,yshift=2cm] 70:0.75 and 0.5) circle (0.03);
    \fill[string-green]([xshift=1.5cm,yshift=2cm] 260:0.75 and 0.5) circle (0.03);
    \draw[white!30!magenta,thick](3.2,-0.5) -- (3.2,0.5);
    \fill[]
    (3.2,-0.5) circle (0.04);
    \fill[string-violet] (3.2,0.5) circle (0.04);
    \draw[white!30!magenta,thick] (5.5,0) circle (0.5)
    (7,0.5) arc (90:270:0.5);
    \draw[cyan,thick] (7,0.5) arc (90:-90:0.5);
    \fill[string-green] (7,-0.5) circle (0.03);
    \fill[string-red] (7,0.5) circle (0.03);
    \draw[thick,->,line cap=round] (3.7,0) -- (4.7,0);
    \fill[black] (2.7,-0.07) circle (0.025)
    (2.7,0.07) circle (0.025);
\end{tikzpicture}
    \caption{A $1$-morphism from an interval to the disjoint union of two defect circles.    
    Different colours represent different labels for the strata. The three interval section on the boundary for which an orientation is indicated correspond to free boundary components with different boundary conditions. The other boundaries are gluing boundaries and carry the same label as the adjacent $2$-stratum.}
    \label{fig:intro-wsh-example}
\end{figure}
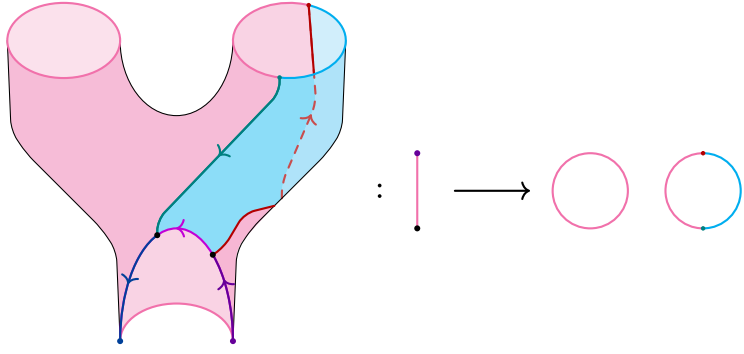

By pulling back the full modular functor constructed from $\rmV_\calc$ in \autocite{HR2024modfunc} we obtain a symmetric monoidal $2$-functor
\begin{equation}
\mathrm{Bl}_{\calc}\colon \WS(\DD_\calc) \to \cat{P}\mathrm{rof}_{\mathbbm{k}}^{\coend \mathrm{ex}}.
\end{equation}
A full conformal field theory is now defined as a braided monoidal oplax natural transformation
\begin{align}
    \begin{tikzcd}[ampersand replacement=\&]
	\WS(\DD_\calc) \&\& {\cat{P}\mathrm{rof}_{\mathbbm{k}}^{\coend \mathrm{ex}}}
	\arrow[""{name=0, anchor=center, inner sep=0}, "{\Delta_{\kk}}", curve={height=-24pt}, from=1-1, to=1-3]
	\arrow[""{name=1, anchor=center, inner sep=0}, "{\mathrm{Bl_\calc}}"', curve={height=24pt}, from=1-1, to=1-3]
	\arrow["\Cor", shorten <=6pt, shorten >=6pt, Rightarrow, from=0, to=1]
    \end{tikzcd}.
\end{align}
where $\Delta_\kk \colon \WS \to \cat{P}\mathrm{rof}_{\mathbbm{k}}^{\coend \mathrm{ex}}$ is the constant symmetric monoidal $2$-functor sending every object to $\vect$.
This definition captures the algebraic structure of a full CFT with the $1$-morphism components of $\Cor$ determining the field content and the $2$-morphism components being the actual correlators as elements in vector spaces of conformal blocks. Moreover, the axioms of an oplax natural transformation correspond precisely to the diffeomorphism covariance and compatibility under gluing of surfaces. This formulation is closely related to the notions of twisted or relative field theories considered in \autocite{ST11twist,FT12relQFT,JFS17relTFT}.

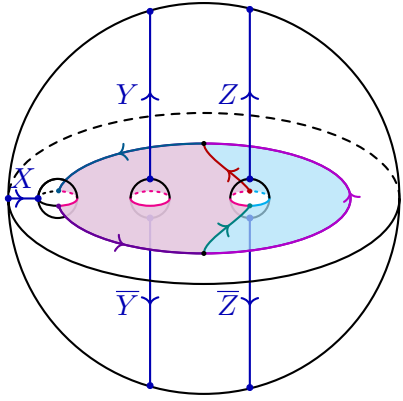
\begin{figure}[t]
    \centering
            \begin{tikzpicture}[scale=0.65]
        \draw[black,thick] (-1.1,0) circle (0.4)
        (0.94,0) circle (0.4)
        (-2.99,0) circle (0.4);
        \filldraw[string-blue]
        (0.94,-0.4) circle (0.06)
        (-1.1,-0.4) circle (0.06);
        \begin{scope}[very thick,decoration={
                      markings,
                      mark=at position 0.52 with {\arrow{>}}}
                     ] 
        \draw[string-blue,thick,postaction={decorate}] (0.94,-0.4) -- (0.94,-3.86);  
        \draw[string-blue,thick,postaction={decorate}] (-1.1,-0.4) -- (-1.1,-3.83); 
        \end{scope}
        \fill[cyan!30,opacity=0.7] (0,0) ellipse (3 and 1.125);
        \fill[magenta!30,opacity=0.7] {[rounded corners] (0,-1.125) -- (0,-1) -- (1,-0.125) -- (1,0) -- (1,0.125) -- (0,1) -- (0,1.125)}
        (0,1.125) arc (90:270:3 and 1.125);
        \draw[draw=string-violet,thick] (0,0) ellipse (3 and 1.125);
        \begin{scope}[very thick,decoration={
                      markings,
                      mark=at position 0.52 with {\arrow{>}}}
                     ] 
        \draw[thick,string-green,postaction={decorate}](0,-1.125) .. controls (0,-0.8) and (1,-0.325) .. (1,0);
        \draw[thick,string-red,postaction={decorate}] (1,0) .. controls (1,0.325) and (0,0.8) .. (0,1.125);
        \draw[thick,string-blue!30!string-green,postaction={decorate}] (0,1.125) arc (90:180:3 and 1.125);
        \draw[thick,string-violet,postaction={decorate}] (-3,0) arc (180:270:3 and 1.125);
        \draw[thick,string-violet!50!magenta,postaction={decorate}] (0,-1.125) arc (-90:90:3 and 1.125);
        \end{scope}
        \filldraw[white](-3,0) circle (0.15);
        \fill[white] (-2.99,-0.15) arc (-90:90:0.4 and 0.15);
        \draw[magenta,thick] (-2.99,-0.15) arc (-90:0:0.4 and 0.15);
        \draw[magenta,thick,dotted,line cap=round] (-2.59,0) arc (0:90:0.4 and 0.15);
        \filldraw[white](0.94,0) ellipse (0.4 and 0.15);
        \draw[draw=cyan,thick] (0.94,-0.15) arc (-90:0:0.4 and 0.15);
        \draw[draw=cyan,thick,dotted,line cap=round] (1.34,0) arc (0:90:0.4 and 0.15);
        \draw[magenta,thick,dotted,line cap=round] (0.94,0.15) arc (90:180:0.4 and 0.15);
        \draw[magenta,thick] (0.54,0) arc (180:270:0.4 and 0.15);
        \fill[white] (-1.1,0) ellipse (0.4 and 0.15);
        \draw[magenta,thick] (-1.5,0) arc (180:360:0.4 and 0.15);
        \draw[magenta,thick,dotted,line cap=round] (-1.5,0) arc (180:0:0.4 and 0.15);
        \draw[black,thick] (-0.7,0) arc (0:180:0.4)
        (-2.59,0) arc (0:180:0.4)
        (1.34,0) arc (0:180:0.4)
        (-3.39,0) arc (180:270:0.4 and 0.15);
        \draw[thick,dotted,line cap=round](-3.39,0) arc (180:90:0.4 and 0.15);
        \filldraw[thick] (0,-1.125) circle (0.03);
        \filldraw[thick](0,1.125) circle (0.03);
        \filldraw[string-violet,thick](188:3 and 1.125) circle (0.03);
        \filldraw[string-blue!30!string-green,thick](172:3 and 1.125) circle (0.03);
        \filldraw[string-green,thick](0.94,-0.15) circle (0.03);
        \filldraw[string-red,thick](0.94,0.15) circle (0.03);
        \draw[thick] (0,0) circle (4)
        (4,0) arc (0:-180: 4 and 1.75);
        \draw[thick,dashed,line cap=round] (4,0) arc (0:180: 4 and 1.75);
        \filldraw[string-blue] (-4,0) circle (0.06)
        (0.94,3.87) circle (0.06)
        (-1.1,3.83) circle (0.06)
        (0.94,-3.87) circle (0.06)
        (-1.1,-3.83) circle (0.06)
        (-3.39,0) circle (0.06)
        (0.94,0.4) circle (0.06)
        (-1.1,0.4) circle (0.06); 
        \begin{scope}[very thick,decoration={
                      markings,
                      mark=at position 0.52 with {\arrow{>}}}
                     ] 
        \draw[string-blue,thick,postaction={decorate}] (0.94,0.4) -- (0.94,3.86);  
        \draw[string-blue,thick,postaction={decorate}] (-1.1,0.4) -- (-1.1,3.83); 
        \draw[string-blue,thick,postaction={decorate}] (-4,0) -- (-3.39,0); 
        \end{scope}
        \node at (-3.7,0) [above,string-blue] {\small $X$};
        \node at (-1.1,2.13) [left,string-blue] {\small $Y$};
        \node at (-1.1,-2.13) [left,string-blue] {\small $\overline{Y}$};
        \node at (0.94,2.13) [left,string-blue] {\small $Z$};
        \node at (0.94,-2.13) [left,string-blue] {\small $\overline{Z}$};
    \end{tikzpicture}
    \caption{Connecting manifold for the $1$-morphism in Figure~\ref{fig:intro-wsh-example}. The blue $1$-strata are endowed with a canonical choice of framing and labelled with objects in $\calc$.}
    \label{fig:intro-connmf-example}
\end{figure}

As in the original construction of \autocite{FFFS2002,FRSI,FjFRS}, the correlators of the full CFT are obtained by evaluating the defect TFT on certain $3$-dimensional bordisms, the so-called \emph{connecting manifolds}, see Figure~\ref{fig:intro-connmf-example}.
Moreover, by considering the connecting manifolds also in one dimension lower, i.e.\ for objects in $\WS_\calc$ we obtain the field content as well. Under some technical assumptions on the defect TFT we prove our main result:
\begin{thm}[{\Cref{thm:fullcft}}]
Evaluating the $3$d defect TFT $\rmZ_\calc$ on the connecting manifolds induces a braided monoidal oplax natural transformation
\begin{align}
    \begin{tikzcd}[ampersand replacement=\&]
	\WS(\DD_\calc) \&\& {\cat{P}\mathrm{rof}_{\mathbbm{k}}^{\coend \mathrm{ex}}}
	\arrow[""{name=0, anchor=center, inner sep=0}, "{\Delta_{\kk}}", curve={height=-24pt}, from=1-1, to=1-3]
	\arrow[""{name=1, anchor=center, inner sep=0}, "{\mathrm{Bl_\calc}}"', curve={height=24pt}, from=1-1, to=1-3]
	\arrow["\Cor", shorten <=6pt, shorten >=6pt, Rightarrow, from=0, to=1]
    \end{tikzcd},
\end{align}
where $\Delta_\kk \colon \WS \to \cat{P}\mathrm{rof}_{\mathbbm{k}}^{\coend \mathrm{ex}}$ is the constant symmetric monoidal $2$-functor sending every object to $\vect$.
\end{thm}

Finally, we will show that the TFT 
\begin{equation*}
        {\rmV}_\calc \colon \bord_{3,2}^{\chi}(\calc) \to \vect
\end{equation*}
of \autocite{DGGPR19} satisfies the technical conditions mentioned above. The resulting CFT corresponds to the so-called \emph{Cardy} or \emph{diagonal} CFT and we show that our construction reproduces expectations from the literature \autocite{FGSS18cardy} for boundary states and partition functions.

The tensor category of topological line defects in this CFT is again given by $\mathcal{C}$, providing to our knowledge the first family of examples of non-semisimple and non-invertible topological line defects in a conformal field theory. In the setting of lattice models, such defects were studied in \autocite{DHY24nonsesisym, KLLMSV25bialgMPO}. We show:

\begin{prop}[{\Cref{prop:algofdefs}}]
    The algebra generated by the defect operators acting on the object of bulk fields is isomorphic (as an algebra) to the linearised Grothendieck ring 
    $\mathrm{Gr}_{\mathbb{C}}(\calc)$ 
    of $\calc$.
\end{prop}

We note that $\mathrm{Gr}_{\mathbb{C}}(\calc)$ is semisimple (as a $\mathbb{C}$-algebra) if and only if $\calc$ is semisimple \autocite{GR17projmod}. Thus, conversely, in a finitely non-semisimple charge-conjugate CFT, the fusion algebra formed by topological defect operators is itself non-semisimple. In particular, there must be defect operators that act non-diagonalisable on the space of bulk fields.

\bigskip

The rest of this paper is organised as follows. In \Cref{sec:modtens} we recall the algebraic ingredients used throughout this paper, including modular tensor categories, and in particular a certain Hopf algebra and its (co)modules internal to such a category. In \Cref{sec:deftft} we recall and introduce the topological and combinatorial setup of the bordism categories needed for the considerations in the rest of the paper. Afterwards, in \Cref{sec:algcft} we recall the definition of topological modular functors as well as the construction of  \autocite{HR2024modfunc}. There we also define and discuss the notion of full CFT based on such a modular functor. In \Cref{sec:tftconstcft} we will construct a full CFT from a given $3$d TFT culminating in our main result \Cref{thm:fullcft}. Finally in \Cref{sec:logcardy} we compute various quantities of physical interest
in the simplest non-trivial example to which our construction applies.

\subsection*{Acknowledgements}
We thank Francesco Costantino and Terry Gannon for useful discussions as well as helpful comments. 
The majority of this work was carried out as part of AH’s doctoral research \autocite{hofer25thesis} at the University of Hamburg.
AH and IR acknowledge support by the Deutsche Forschungsgemeinschaft (DFG, German Research Foundation) under Germany`s Excellence Strategy - EXC 2121 ``Quantum Universe" - 390833306, and by the Collaborative Research Centre CRC 1624 ``Higher structures, moduli spaces and integrability'' - 506632645. AH gratefully acknowledges the Max Planck Institute for Mathematics in Bonn for its hospitality and financial support.

\subsection*{Conventions}
Throughout this article, $\calc$ will be a modular tensor category.
Moreover, we will appeal to the standard coherence results and assume $\calc$ to be strictly monoidal and strictly pivotal. We will always work over an algebraically closed field $\mathbbm{k}$ of characteristic zero. Unless otherwise noted, functors between abelian (and linear) categories will always be assumed to be additive (and linear). For details see \Cref{sec:modtens}.

By ``manifold'' we will always mean a compact smooth manifold. However since we are exclusively working in dimensions less then four we will sometimes work in the topological category instead. Every manifold we will consider will be oriented, and every diffeomorphism will be orientation preserving unless explicitly stated otherwise. For any manifold $M$ we will denote the manifold with reversed orientation by $-M$. The interval $[0,1]$ will be denoted by $I$ and the unit circle by $S^1$. Finally, by a closed manifold we mean a compact manifold without boundary.

 In this article a $2$-category will always mean a weak $2$-category otherwise known as a bicategory. Analogously a $2$-functor will always mean a weak $2$-functor with coherence isomorphisms otherwise known as a pseudofunctor.

\section{Algebraic preliminaries}\label{sec:modtens}
In this section, we collect definitions and results related to ribbon categories with a particular emphasis on modular tensor categories $\calc$. 
Moreover, we collect a number of results on relations between $\calc$, its Drinfeld centre $\calZ(\calc)$, and the Deligne product $\calc\boxtimes \overline{\calc}$ in the modular setting. 

\subsection{Finite ribbon categories}\label{subsec:ribboncats}
A linear category is called \emph{finite} if it is equivalent, as a linear category, to the category $A$-$\Mod$ of finite dimensional modules of some finite dimensional algebra $A$. In particular, a finite linear category is abelian, every object has a projective cover, and its Hom sets are finite dimensional vector spaces. For the complete intrinsic definition of finite linear categories see \autocite[Sec.\,1.8]{EGNO}. 
By a \emph{finite tensor category} we mean a finite linear category which is in addition a rigid monoidal category such that the monoidal product $\otimes$ is bilinear and the monoidal unit $\unit$ is simple. 

A \emph{finite ribbon category} $\calc$ is a finite tensor category which is also ribbon.  We will employ the following conventions for structure morphisms in $\calc$. Every object $X$ in $\calc$ has a two-sided dual $X^{*}$, with duality morphisms denoted by:
\begin{equation}
   \begin{aligned}
    \mathrm{ev}_X &\colon X^{*} \otimes X \to \unit, \qquad \mathrm{coev}_X \colon \unit \to X \otimes X^{*}, 
    \\
    \widetilde{\mathrm{ev}_X} &\colon X \otimes X^{*} \to \unit, \qquad \widetilde{\mathrm{coev}_X} \colon \unit \to X^{*} \otimes X. 
\end{aligned} 
\end{equation}
The components of the braiding and twist isomorphisms will be denoted by
\begin{align}
    \beta_{X,Y} \colon X \otimes Y \to Y \otimes X, \qquad \vartheta_X \colon X \to X.
\end{align}
The twist $\vartheta$ satisfies
\begin{align}
    \vartheta_{X\otimes Y} = \beta_{Y,X} \circ \beta_{X,Y} \circ (\vartheta_X \otimes \vartheta_Y), \quad \mathrm{and} \quad 
    (\vartheta_X)^* = \vartheta_{X^*}
\end{align}
for all $X,Y \in \calc$. A direct computation shows that the twist $\vartheta$ endows $\calc$ with a pivotal structure, which is even spherical \autocite[Sec.\,8.10]{EGNO}. Moreover, we will appeal to the standard coherence results and assume $\calc$ to be strictly monoidal and strictly pivotal.
In diagrammatic notation the structural morphisms of $\calc$ will be represented as
{\allowdisplaybreaks\begin{align}
        \lev_X &= 
    \begin{tikzpicture}[baseline]
        \draw[string-blue,thick,line cap=round] (0.5,0) arc (0:180:0.5cm);
        \draw[string-blue,thick,line cap=round,->] (0.5,0) arc (0:92:0.5cm);
        \node at (-0.5,0) [string-blue,thick,below] {\scriptsize $X^*$};
        \node at (0.5,0) [string-blue,thick,below] {\scriptsize $X$};
    \end{tikzpicture} 
    & \lcoev_X &= 
    \begin{tikzpicture}[baseline]
        \draw[string-blue,thick,line cap=round] (0.5,0) arc (0:-180:0.5cm);
        \draw[string-blue,thick,line cap=round,->] (0.5,0) arc (0:-92:0.5cm);
        \node at (0.5,0) [string-blue,thick,above] {\scriptsize $X^*$};
        \node at (-0.5,0) [string-blue,thick,above] {\scriptsize $X$};
    \end{tikzpicture} 
    & \beta_{X,Y} &= 
    \begin{tikzpicture}[baseline]
        \draw[string-blue,thick,line cap=round] (0.5,-0.5) node[string-blue,below] {\scriptsize $Y$} -- (0.15,-0.15) -- (-0.5,0.5) node[string-blue,above] {\scriptsize $Y$};
        \draw[white, line width=1.5mm] (-0.5,-0.5)-- (0.5,0.5);
        \draw[string-blue, thick] (-0.5,-0.5) node[string-blue,below] {\scriptsize $X$} -- (0.5,0.5) node[string-blue,above] {\scriptsize $X$};
\end{tikzpicture}
        \nonumber \\
        \rev_X &= 
        \begin{tikzpicture}[baseline]
            \draw[string-blue,thick,line cap=round] (0.5,0) arc (0:180:0.5cm);
            \draw[string-blue,thick,line cap=round,->] (-0.5,0) arc (180:87:0.5cm);
            \node at (-0.5,0) [string-blue,thick,below] {\scriptsize $X$};
            \node at (0.5,0) [string-blue,thick,below] {\scriptsize $X^{*}$};
        \end{tikzpicture}
        & \rcoev_X &= 
       \begin{tikzpicture}[baseline]
            \draw[string-blue,thick,line cap=round] (0.5,0) arc (0:-180:0.5cm);
            \draw[string-blue,thick,line cap=round,->] (-0.5,0) arc (-180:-87:0.5cm);
            \node at (0.5,0) [string-blue,thick,above] {\scriptsize $X$};
            \node at (-0.5,0) [string-blue,thick,above] {\scriptsize $X^{*}$};
        \end{tikzpicture} 
        & \theta_{X} &= 
        \begin{tikzpicture}[baseline]
        \draw[string-blue,thick,line cap=round] (0,-0.2) -- (0,0.-0.5) node[string-blue,below] {\scriptsize $X$} ;
         \draw[string-blue,thick,line cap=round] (0,0.2) --(0,0.5) node[string-blue,above] {\scriptsize $X$};
        \draw[string-blue,thick,line cap=round] (0,-0.2) arc (180:90:0.15 and 0.3);
        \draw[string-blue,thick,line cap=round] (0,0.2) arc (180:210:0.15 and 0.3);
        \draw[string-blue,thick,line cap=round] (0.15,0.1) arc (90:-90:0.1 and 0.1);
        \draw[string-blue,thick,line cap=round] (0.15,-0.1) arc (-90:-120:0.15 and 0.3);
\end{tikzpicture}
\end{align}}
Note that we read such diagrams from the bottom to the top.
We will use the same conventions as \autocite{TV} and denote with $\overline{\calc}$ the same underlying pivotal
tensor category, but equipped with the inverse braiding and twist, i.e.\
\begin{equation}
    \begin{aligned}
        \overline{\beta}_{X,Y} := \beta_{Y,X}^{-1} \colon X \otimes Y \to Y\otimes X \qquad \qquad \overline{\vartheta}_{X} := \vartheta_{X}^{-1} \colon X \to X
    \end{aligned}
\end{equation}
see \autocite[Sec.\,1.2.2 \& Ex.\ 3.1.7]{TV}. We will call $\overline{\calc}$ the \emph{mirrored} category of $\calc$.

Next recall that $\calc$ admits an inner Hom functor $\calc^{\mathrm{op}} \times \calc \to \calc$ which sends $(X,Y) \in \calc^{\mathrm{op}} \times \calc$ to $X^{*} \otimes Y$. Due to our finiteness assumptions the coend of this functor exists 
and can be explicitly described as a cokernel, see \autocite[Ch.\,5]{KL2001nsstqft} for details. We will denote this coend as
\begin{align}
    \coend := \int^{X \in \calc}\,X^{*} \otimes X
\end{align}
with universal dinatural transformation
\begin{align}
    \iota_X \colon X^* \otimes X \to \coend.
\end{align}
The object $\coend$ is referred to as the \emph{canonical coend} and naturally carries the structure of a Hopf algebra with Hopf pairing in the braided monoidal category $\calc$. For Hopf algebras in braided monoidal categories we will follow the convention of \autocite[Ch.\,6]{TV}.
The structural morphisms of $\coend$ are induced from the universal property of the coend as shown in \Cref{fig:coend-Hopf-def}. If we do not assume $\calc$ to be strictly pivotal the canonical isomorphisms $(X\otimes Y)^* \cong Y^* \otimes X^*$ for the product, $\unit^* \otimes \unit \cong \unit$ for the unit, and $X \cong X^{**}$ for the antipode are needed, see \autocite[Sec.\,6.4\,\&\,6.5]{TV} for details. 
Exchanging the over and under braidings in the definition of $\omega$ gives a pairing $\overline{\omega} \colon \coend \otimes \coend \to \unit $ which satisfies 
\begin{equation}\label{eq:hopfpairingmirror}
    \omega \circ (S\otimes \id_\coend)= \overline{\omega} = \omega \circ (\id_\coend \otimes S)
\end{equation}
see \autocite[Eq.\,5.2.8]{KL2001nsstqft}.

\begin{figure}[tb]
\begin{equation*}
    \begin{aligned}
      \mu : \begin{tikzpicture}[scale=0.85,baseline=0cm]
       \begin{scope}[decoration={
                      markings,
                      mark=at position 0.55 with {\arrow{>}}}
                     ] 
            \draw[string-red,thick,postaction={decorate},line cap=round] (-0.75,-0.1) -- (-0.75,-1);
            \draw[string-red,thick,postaction={decorate},line cap=round] (-0.25,-1) -- (-0.25,-0.1);
            \draw[string-red,thick,postaction={decorate},line cap=round] (0.75,-1) -- (0.75,-0.1);
            \draw[string-red,thick,postaction={decorate},line cap=round] (0.25,-0.1) -- (0.25,-1);
        \end{scope}
        \draw[string-blue,thick,line cap=round]
        (0.5,0) -- (0.5,-0.1)
        (-0.5,0) -- (-0.5,-0.1)
        (0.5,0) arc (0:180:0.5 and 0.5)
        (0,0.5) -- (0,1)
        (0.1,-0.1) -- (0.9,-0.1)
        (-0.1,-0.1) -- (-0.9,-0.1);
       \node at (-0.8,-0.1) [string-blue,left] {\scriptsize $\iota$\tiny${}_X$};
       \node at (0.8,-0.1) [string-blue,right] {\scriptsize $\iota$\tiny${}_Y$};
       \node at (-0.68,-0.875) [string-red,below] {\scriptsize $X^*$};
       \node at (-0.23,-0.9) [string-red,below] {\scriptsize $X$};
        \node at (0.32,-0.875) [string-red,below] {\scriptsize $Y^*$};
       \node at (0.77,-0.9) [string-red,below] {\scriptsize $Y$};
       \node at (0,0.915) [string-blue,above] {\scriptsize $\coend$};
   \end{tikzpicture}  
   &= \begin{tikzpicture}[scale=0.85,baseline=0.05cm]
        
        \begin{scope}[decoration={
                      markings,
                      mark=at position 0.85 with {\arrow{>}}}
                     ] 
        \draw[string-red,thick,postaction={decorate},line cap=round] (-0.75,0.3) .. controls (-0.75,-0.3) and (0.25,-0.4) .. (0.25,-1);
        \end{scope}
       \begin{scope}[decoration={
                      markings,
                      mark=at position 0.55 with {\arrow{>}}}
                     ] 
            \draw[string-red,thick,postaction={decorate},line cap=round] (0.75,-1) -- (0.75,0.3);
            \fill[white] (-0.23,-0.55) -- (-0.13,-0.4) -- (0.05,-0.45) -- (-0.05,-0.6) -- cycle;
            \begin{scope}[xshift=-0.3cm,yshift=0.3cm]
                \fill[white] (-0.23,-0.55) -- (-0.13,-0.4) -- (0.05,-0.45) -- (-0.05,-0.6) -- cycle;
            \end{scope}
            \draw[string-red,thick,postaction={decorate},line cap=round] (-0.25,0.3) .. controls (-0.25,-0.3) and (-0.75,-0.4) .. (-0.75,-1);
            \draw[string-red,thick,postaction={decorate},line cap=round] (-0.25,-1) .. controls (-0.25,-0.4) and (0.25,-0.3) .. (0.25,0.3);
        \end{scope}
        \draw[string-blue,thick,line cap=round]
        (0,0.3) -- (0,1)
        (0.9,0.3) -- (-0.9,0.3);
       \node at (0.8,0.3) [string-blue,right] {\scriptsize $\iota$\tiny${}_{X\otimes Y}$};
       \node at (-0.68,-0.875) [string-red,below] {\scriptsize $X^*$};
       \node at (-0.23,-0.9) [string-red,below] {\scriptsize $X$};
        \node at (0.32,-0.875) [string-red,below] {\scriptsize $Y^*$};
       \node at (0.77,-0.9) [string-red,below] {\scriptsize $Y$};
       \node at (0,0.915) [string-blue,above] {\scriptsize $\coend$};
   \end{tikzpicture}   &
   \hfill \eta : \begin{tikzpicture}[baseline]
        \draw[string-blue,thick,line cap=round] (0.25,-0.5) -- (0.25,0.5);
        \filldraw[draw=string-blue, thick,fill=white] (0.25,-0.5) circle (0.07);
        \node at (0.2475,0.475) [string-blue,above] {\scriptsize $\coend$};
   \end{tikzpicture}  
   &= \begin{tikzpicture}[baseline]
        \draw[string-blue,thick,line cap=round] (0.25,-0.5) -- (0.25,0.5);
        \draw[string-blue,thick,line cap=round] (-0.1,-0.5) -- (0.6,-0.5);
        \node at (0.2475,0.475) [string-blue,above] {\scriptsize $\coend$};
        \node at (0.5,-0.5) [string-blue,right] {\scriptsize $\iota$\tiny${}_{\unit}$};
   \end{tikzpicture}   \\
   \Delta : \begin{tikzpicture}[scale=0.85,baseline=0cm]
       \begin{scope}[decoration={
                      markings,
                      mark=at position 0.55 with {\arrow{>}}}
                     ] 
            \draw[string-red,thick,postaction={decorate},line cap=round] (-1.8,-0.8) -- (-1.8,0);
            \draw[string-red,thick,postaction={decorate},line cap=round] (-2.4,0) -- (-2.4,-0.8);
        \end{scope}
        \draw[string-blue,thick,line cap=round]
        (-1.6,1.2) -- (-1.6,1)
        (-2.6,1.2) -- (-2.6,1)
        (-1.6,1) arc (0:-180:0.5 and 0.5)
        (-2.1,0.5) -- (-2.1,0)
        (-1.6,0) -- (-2.6,0);
       \node at (-1.7,0) [string-blue,right] {\scriptsize $\iota$\tiny${}_X$};
       \node at (-1.78,-0.7) [string-red,below] {\scriptsize $X$};
       \node at (-2.3,-0.675) [string-red,below] {\scriptsize $X^*$};
       \node at (-1.575,1.15) [string-blue,above] {\scriptsize $\coend$};
       \node at (-2.575,1.15) [string-blue,above] {\scriptsize $\coend$};
   \end{tikzpicture}   
   &= \begin{tikzpicture}[scale=0.85,baseline=0.25cm]
        \begin{scope}[decoration={
                      markings,
                      mark=at position 0.52 with {\arrow{>}}}
                     ]   
            \draw[string-red,thick,postaction={decorate},line cap=round] (-1.9,0.7) arc (0:-180: 0.3 and 0.35); 
            \draw[string-red,thick,postaction={decorate},line cap=round] (-3.2,0.5) .. controls (-3.2,0.2) and (-2.55,0.1) .. (-2.55,-0.2);
            \draw[string-red,thick,postaction={decorate},line cap=round] (-1.85,-0.2) .. controls (-1.85,0.1) and (-1.2,0.2) .. (-1.2,0.5);
            \draw[string-red,thick,line cap=round] (-1.2,0.7) -- (-1.2,0.9)
            (-1.9,0.7) -- (-1.9,0.9)
            (-1.2,0.5) -- (-1.2,0.9)
            (-3.2,0.5) -- (-3.2,0.9)
            (-2.5,0.7) -- (-2.5,0.9)
            (-2.55,-0.2) -- (-2.55,-0.5)
            (-1.85,-0.2) -- (-1.85,-0.5);
            \draw[string-blue,thick,line cap=round] (-1.05,0.9) -- (-2.05,0.9)
            (-3.35,0.9) -- (-2.35,0.9);
            \draw[string-blue,thick,line cap=round] (-1.55,0.9) -- (-1.55,1.5)
            (-2.85,0.9) -- (-2.85,1.5);
            
        \end{scope}
        \node at (-1.5,1.475) [string-blue,above] {\scriptsize $\coend$};
        \node at (-2.85,1.475) [string-blue,above] {\scriptsize $\coend$};
        \node at (-1.8,-0.45) [string-red,below] {\scriptsize $X$};
        \node at (-2.45,-0.4125) [string-red,below] {\scriptsize $X^*$};
   \end{tikzpicture} 
   & \hfill \epsilon : \begin{tikzpicture}[baseline]
       \draw[string-red,thick,line cap=round] (0,-0.5) -- (0,0);
       \draw[string-red,thick,->,line cap=round] (0,0) -- (0,-0.3);
       \draw[string-red,thick,line cap=round] (0.5,-0.5) -- (0.5,0);
       \draw[string-red,thick,->,line cap=round] (0.5,-0.5) -- (0.5,-0.2);
        \draw[string-blue, thick,line cap=round] (-0.1,0) -- (0.6,0);
        \draw[string-blue,thick,line cap=round] (0.25,0) -- (0.25,0.5);
        \filldraw[draw=string-blue, thick,fill=white] (0.25,0.5) circle (0.07);
        \node at (0.5,0) [string-blue,right] {\scriptsize $\iota$\tiny${}_{X}$};
        \node at (0.5,-0.5) [string-red,below] {\scriptsize $X$};
        \node at (0,-0.5) [string-red,below] {\scriptsize $X^*$};
   \end{tikzpicture}   
   &= \begin{tikzpicture}[baseline]
        \draw[string-red,thick,line cap=round] (0.5,0.25) arc (0:180:0.25cm);
        \draw[string-red,thick,->,line cap=round] (0.5,0.25) arc (0:95:0.25cm);
        \draw[string-red,thick,line cap=round] (0.5,0.25) -- (0.5,-0.5);
        \draw[string-red,thick,line cap=round] (0,0.25) -- (0,-0.5);
        \node at (0,-0.5) [string-red,thick,below] {\scriptsize $X^*$};
        \node at (0.5,-0.5) [string-red,thick,below] {\scriptsize $X$};
    \end{tikzpicture}
   \\
   S :\begin{tikzpicture}[scale=0.85,baseline=0cm]
        
       \begin{scope}[decoration={
                      markings,
                      mark=at position 0.55 with {\arrow{>}}}
                     ] 
            \draw[string-red,thick,postaction={decorate},line cap=round] (-1.8,-0.8) -- (-1.8,0);
            \draw[string-red,thick,postaction={decorate},line cap=round] (-2.4,0) -- (-2.4,-0.8);
        \end{scope}
        \draw[string-blue,thick,line cap=round]
        (-2.1,0.6) circle (0.2)
        (-2.1,1.2) -- (-2.1,0)
        (-1.9,0.6) -- (-2.3,0.6)
        (-1.6,0) -- (-2.6,0);
       \node at (-1.7,0) [string-blue,right] {\scriptsize $\iota$\tiny${}_X$};
       \node at (-1.78,-0.7) [string-red,below] {\scriptsize $X$};
       \node at (-2.3,-0.675) [string-red,below] {\scriptsize $X^*$};
       \node at (-2.085,1.15) [string-blue,above] {\scriptsize $\coend$};
   \end{tikzpicture} 
   &= \begin{tikzpicture}[scale=0.85,baseline=0cm]
        
       \begin{scope}[decoration={
                      markings,
                      mark=at position 0.33 with {\arrow{<}}}
                     ]
            \draw[string-red,thick,line cap=round] (-1.54,-0.4) arc (0:-115:0.13 and 0.13);
            \draw[string-red,thick,line cap=round] (-1.8,-0.8) -- (-1.8,-0.2);
            \fill[white] (-1.85,-0.4) rectangle (-1.75,-0.25);
            \draw[string-red,thick,line cap=round] (-1.8,-0.4) arc (180:0:0.13 and 0.13);
            \draw[string-red,thick,line cap=round] (-2.4,-0.2) -- (-2.4,-0.8);
            \draw[string-red,thick,postaction={decorate},line cap=round] (-2.4,0.6) .. controls (-2.4,0.3) and (-1.8,0.1).. (-1.8,-0.2);
            \begin{scope}[xshift=-2cm,yshift=0.7cm]
                \fill[white] (-0.25,-0.55) -- (-0.13,-0.4) -- (0.05,-0.45) -- (-0.07,-0.6) -- cycle;
            \end{scope} 
            \draw[string-red,thick,postaction={decorate},line cap=round] (-2.4,-0.2) .. controls (-2.4,0.1) and (-1.8,0.3) .. (-1.8,0.6);
        \end{scope}
        \draw[string-blue,thick,line cap=round]
        (-2.1,1.2) -- (-2.1,0.6)
        (-1.6,0.6) -- (-2.6,0.6);
       \node at (-1.7,0.6) [string-blue,right] {\scriptsize $\iota$\tiny${}_{X^*}$};
       \node at (-1.78,-0.7) [string-red,below] {\scriptsize $X$};
       \node at (-2.3,-0.675) [string-red,below] {\scriptsize $X^*$};
       \node at (-2.085,1.15) [string-blue,above] {\scriptsize $\coend$};
   \end{tikzpicture}  
   & \hfill \omega : \begin{tikzpicture}[scale=0.85,baseline=0cm]
       \begin{scope}[decoration={
                      markings,
                      mark=at position 0.55 with {\arrow{>}}}
                     ] 
            \draw[string-red,thick,postaction={decorate},line cap=round] (-0.75,-0.1) -- (-0.75,-1);
            \draw[string-red,thick,postaction={decorate},line cap=round] (-0.25,-1) -- (-0.25,-0.1);
            \draw[string-red,thick,postaction={decorate},line cap=round] (0.75,-1) -- (0.75,-0.1);
            \draw[string-red,thick,postaction={decorate},line cap=round] (0.25,-0.1) -- (0.25,-1);
        \end{scope}
        \filldraw[fill=white,draw=string-blue,thick,line cap=round] (-0.5,0.6) rectangle (0.5,1);
        \draw[string-blue,thick,line cap=round]
        (0.5,0) -- (0.5,-0.1)
        (-0.5,0) -- (-0.5,-0.1)
        (0.1,-0.1) -- (0.9,-0.1)
        (-0.1,-0.1) -- (-0.9,-0.1)
        (0.5,0) .. controls (0.5,0.2) and (0.3,0.4) .. (0.3,0.6)
        (-0.5,0) .. controls (-0.5,0.2) and (-0.3,0.4) .. (-0.3,0.6);
       \node at (-0.8,-0.1) [string-blue,left] {\scriptsize $\iota$\tiny${}_X$};
       \node at (0.8,-0.1) [string-blue,right] {\scriptsize $\iota$\tiny${}_Y$};
       \node at (-0.68,-0.875) [string-red,below] {\scriptsize $X^*$};
       \node at (-0.23,-0.9) [string-red,below] {\scriptsize $X$};
        \node at (0.32,-0.875) [string-red,below] {\scriptsize $Y^*$};
       \node at (0.77,-0.9) [string-red,below] {\scriptsize $Y$};
       \node at (0,0.8) [string-blue] {\scriptsize $\omega$};
   \end{tikzpicture} 
   &= \begin{tikzpicture}[scale=0.85,baseline=0cm]
       \begin{scope}[decoration={
                      markings,
                      mark=at position 0.52 with {\arrow{>}}}
                     ] 
            \draw[string-red,thick,line cap=round] (-0.75,0.5) -- (-0.75,-1);
            \draw[string-red,thick,line cap=round] (-0.25,-1) -- (-0.25,-0.75);
            \draw[string-red,thick,line cap=round] (0.75,-1) -- (0.75,0.5);
            \draw[string-red,thick,line cap=round] (0.25,-0.75) -- (0.25,-1);
            \draw[string-red,thick,postaction={decorate},line cap=round] (0.75,0.5) arc (0:180:0.25 and 0.22);
            \draw[string-red,thick,postaction={decorate},line cap=round] (-0.25,0.5) arc (0:180:0.25 and 0.22);
            \draw[string-red,thick,line cap=round] (0.25,-0.75) .. controls (0.25,-0.55) and (-0.25,-0.325) .. (-0.25,-0.125);
            \begin{scope}[xshift=0.1cm,yshift=0.05cm]
                \fill[white] (-0.2,-0.5) -- (-0.1,-0.4) -- (0,-0.5) -- (-0.1,-0.6) -- cycle;
            \end{scope}
            \draw[string-red,thick,line cap=round] (-0.25,-0.75) .. controls (-0.25,-0.55) and (0.25,-0.325) .. (0.25,-0.125);
            \draw[string-red,thick] (-0.25,0.5) .. controls (-0.25,0.3) and (0.25,0.075) .. (0.25,-0.125);
            \begin{scope}[xshift=0.1cm,yshift=0.685cm]
                \fill[white] (-0.2,-0.5) -- (-0.1,-0.4) -- (0,-0.5) -- (-0.1,-0.6) -- cycle;
            \end{scope}
            \draw[string-red,thick,line cap=round] (0.25,0.5) .. controls (0.25,0.3) and (-0.25,0.075) .. (-0.25,-0.125);
        \end{scope}
       \node at (-0.68,-0.875) [string-red,below] {\scriptsize $X^*$};
       \node at (-0.23,-0.9) [string-red,below] {\scriptsize $X$};
        \node at (0.32,-0.875) [string-red,below] {\scriptsize $Y^*$};
       \node at (0.77,-0.9) [string-red,below] {\scriptsize $Y$};
   \end{tikzpicture} 
    \end{aligned}
\end{equation*}
\caption{The dinatural transformations defining the Hopf algebra structure morphisms and the Hopf pairing on $\coend$. We use red for the ribbons labelled with $X;Y \in \calc$ here to highlight that the structural morphism of $\coend$ do not depend on the specific labels.}\label{fig:coend-Hopf-def}
\end{figure}
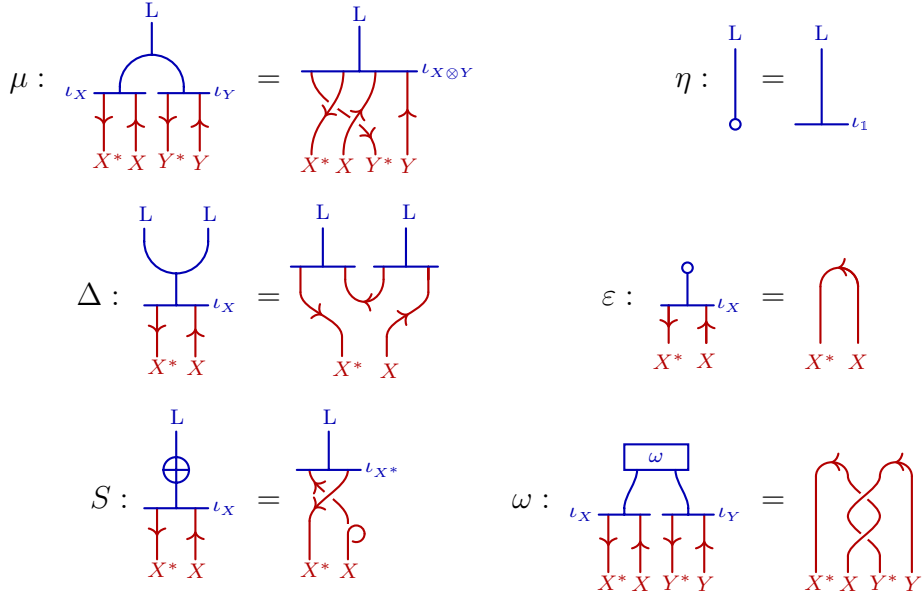

\begin{definition}\label{def:mtc}
    A finite ribbon category $\calc$ is called \emph{modular} if the canonical Hopf pairing $\omega$ of the coend $\coend$ is non-degenerate.
\end{definition}
There are other equivalent definitions of modularity, one of them will be discussed below, see \autocite[Thm.\,1.1]{SHIMIZU2019nondeg} for the full list. From now on we will always assume $\calc$ to be modular. 
It can be shown that $\calc$ is unimodular and that the Hopf algebra $\coend$ admits 
a unique-up-to-scalar two-sided integral $\Lambda \colon \unit \to \coend$ and cointegral $\lambda\colon \coend \to \unit$, see \autocite[Sec.\,2]{DGGPR19} and references therein for more details. 

We will normalise the integral and cointegral in terms of the \textsl{modularity parameter} $\zeta \in \mathbbm{k}^{\times}$ as
\begin{align}\label{eq:integralconv}
    \lambda \circ \Lambda = \id_{\unit} 
    \quad , \quad
    \omega \circ (\id_{\coend} \otimes \Lambda) = \zeta \lambda \ .
\end{align}
That the second condition is possible is in fact equivalent to non-degeneracy of $\omega$, see \autocite[Thm.\,5]{Kerler95genealogy} or \autocite[Lem.\,6.2]{TV}.
The following statement is analogous to the case of classical Hopf algebras \autocite[Cor.\,4.2.13]{KL2001nsstqft}. 
\begin{proposition}\label{prop:frobpairingcoend}
    The cointegral $\lambda$ induces a non-degenerate pairing 
    $\kappa := \lambda \circ \mu \colon \coend \otimes \coend \to \unit$ 
    which equips $\coend$ with the structure of a Frobenius algebra in $\calc$. The copairing is given by $ (S \otimes \id_{\coend})\circ\Delta \circ \Lambda\colon \unit \to \coend\otimes \coend$. 
\end{proposition}
The pairing $\kappa$ is called the \emph{Radford pairing}. Note that this convention for $\kappa$ keeps the algebra structure of $\coend$ fixed while the one we used in \autocite{HR2024modfunc} keeps the coalgebra structure. This difference in convention makes some of the computations in \cref{sec:logcardy} a bit more straightforward.
Since the Hopf pairing is also non-degenerate, the composition 
\begin{align}\label{eq:def-S-endo-of-L}
\calS := (\omega \otimes \id_{\coend}) \circ (\id_{\coend} \otimes ( (S \otimes \id_{\coend})\circ\Delta\circ \Lambda)) \colon \coend \to \coend
\end{align}
is invertible and satisfies
\begin{align}\label{eq:S-property-pairing}
    \kappa \circ (\calS \otimes \id_\coend) = \omega = \kappa \circ (\id_\coend \otimes \calS) 
\end{align}
by definition. We can define another endomorphism of $\coend$ by the universal property via $\cat{T} \circ \iota_X = \iota_X \circ (\id_{X^*} \otimes \theta_X)$. 
This morphism is invertible with inverse $\cat{T}^{-1} \circ \iota_X = \iota_X \circ (\id_{X^*} \otimes \theta_X^{-1})$. 
Moreover the constants $\Delta^{\pm}$ defined via 
\begin{align}\label{eq:defDeltapm}
    \epsilon \circ \cat{T}^{\pm1} \circ \Lambda = \Delta^{\pm} \, \id_{\unit}
\end{align} are non-zero 
and satisfy $\zeta = \Delta^+ \Delta^-$, see \autocite[Prop.\,2.6 \& Cor.\,4.6]{DGGPR19}. Using this and the normalisation of $\Lambda$ and $\lambda$ a direct computation shows that 
\begin{align} \label{eq:S^2}
    \calS^2 = \zeta S^{-1}
\end{align}
where $S^{-1}$ is the inverse of the antipode of $\coend$. The name \emph{modular} tensor category is justified by the following lemma:
\begin{lemma}[{\autocite[Thm.\,2.1.9]{Lyu95invproj}}]
    The morphisms $\calT$ and $\calS$ induce a projective $\mathrm{SL}(2,\Z)$-action on the morphism spaces $\Hom_{\calc}(\coend,\unit)$ and $\Hom_{\calc}(\unit,\coend)$ via pre- and postcompostion, respectively.
\end{lemma}
Accordingly, we will call $\calS\colon \coend \to \coend$ the modular $S$-transformation of $\calc$. The modular group action on $\Hom_\calc(\coend,\unit)$ is actually a part of the chiral modular functor constructed in \autocite{HR2024modfunc}.

\subsection{(Co)modules of the canonical Hopf algebra}\label{subsec:reptheory}
Let us now discuss how the categories of $\coend$-modules and comodules are related to more well-known categories. 

To this end first recall the \emph{Drinfeld centre} of a monoidal category $\calA$ as the braided monoidal category $\calZ(\calA)$ with objects pairs $(X,\gamma)$, where $X \in \calA$, and $\gamma \colon X \otimes - \Rightarrow - \otimes X$ is a natural isomorphism, called a \emph{half braiding}, satisfying a hexagon type axiom \autocite[Def.\,7.13.1]{EGNO}.
A key property of the Drinfeld centre is that for suitable $\calA$ it is modular:
\begin{proposition}[{\autocite[Thm.\,5.11]{Shimizu17ribbondrinf}}]\label{prop:drinmod}
Let $\calA$ be finite tensor category satisfying the sphericality condition of \autocite[Def.\,3.5.2]{DSS13tensorcats}.
 Then $\calZ(\calA)$ is a modular tensor category.
\end{proposition}
In particular a modular tensor category $\calc$ satisfies this sphericality condition which implies that $\calZ(\calc)$ is a modular tensor category as well. To see this, combine \autocite[Thm.\,1.3]{SS21modtracenaka} and \autocite[Cor.\,4.7]{GR17projmod}.  
A generalisation of this result to pivotal tensor categories can be found in \autocite[Cor.\,2.13]{MW22drinGV}.

The forgetful functor $U \colon \calZ(\calc) \to \calc$ which sends a pair $(X,\gamma)$ to $X$ has a left adjoint $F \colon \calc \to \calZ(\calc)$. In the modular setting $F$ is actually also a right adjoint as we will see below. The corresponding monad $U \circ F$ is naturally isomorphic to the central monad $Z$ on $\calc$
\begin{equation}
        (U \circ F)(-) \cong Z (-) := \int^{X \in \calc} X^* \otimes - \otimes X
\end{equation}
and the corresponding Eilenberg-Moore category is precisely $\calZ(\calc)$ \autocite[Sec.\,5.6]{BV12monadcenter} (see also \autocite[Ch.\,9]{TV} for a textbook account). See also \autocite[Prop.\,5.9]{HR2024modfunc} for a topological proof of this statement using modular functors. 

From the braiding in $\calc$ we get an isomorphism of monads $Z(-)\cong -\otimes \coend$, see e.g.\ \autocite[Lem.\,3.7]{SHIMIZU2019nondeg}, which gives rise to the following:
\begin{lemma}\label{lem:center-modules}
    The functor 
    \begin{equation}
        \begin{aligned}
            \calZ(\calc) &\to \calc_\coend\\
            (X,\gamma) &\mapsto (X,\rho^\gamma)
        \end{aligned}
    \end{equation}
    where the right $\coend$-action $\rho^\gamma$ is obtained from the half-braiding $\gamma$ via
    \begin{equation}\label{eq:canonicalactiondef}
    \begin{tikzpicture}[baseline=1cm]
            
            \begin{scope}[very thick,decoration={
                      markings,
                      mark=at position 0.52 with {\arrow{>}}}
                     ]
            \draw[string-red,thick,postaction={decorate},line cap =round] (0.4,0.6) -- (0.4,0);
            \draw[string-red,thick,postaction={decorate},line cap =round] (0.8,0) -- (0.8,0.6);
            \end{scope}  
            \draw[string-blue,thick, line cap = round] (0,0) -- (0,2)
            (0.6,0.6)--(0.6,1)
            (0.3,0.6) -- (0.9,0.6);
            \filldraw[fill=white,draw=string-blue,thick] (-0.2,1) rectangle (0.8,1.5);
            \node at (0,0) [below,string-blue] {\scriptsize $X$};
            \node at (0,2) [above,string-blue] {\scriptsize $X$};
            \node at (0.5,0) [below,string-red] {\scriptsize $Y^*$};
            \node at (0.85,0) [below,string-red] {\scriptsize $Y$};
            \node at (0.8,0.6) [right,string-blue] {\scriptsize $\iota$\tiny${}_Y$};
            \node at (0.4,1.25) [string-blue] {\scriptsize $\rho$\tiny${}^\gamma$};
    \end{tikzpicture}
        =
    \begin{tikzpicture}[baseline=1cm]
            \draw[string-blue,thick, line cap = round] (0.8,1.5) -- (0.8,2);
            \draw[string-red,thick, line cap = round](0,1.4) -- (0,1.6);
            \draw[string-red,thick,-<,line cap = round] (0.4,0) -- (0.4,0.4);
            \draw[string-red,thick,line cap = round] (0.4,0.3) .. controls (0.4,0.6) and (-0.6,0.8) .. (-0.5,1.6);
            \draw[draw=white,double=string-blue,very thick, line cap = round] (0,0) -- (0,1);
            \draw[string-blue,thick, line cap = round]
            (0,0) -- (0,1);
            \begin{scope}[very thick,decoration={
                      markings,
                      mark=at position 0.52 with {\arrow{>}}}
                     ]
            \draw[string-red,thick,postaction={decorate},line cap =round] (0.8,0) -- (0.8,1);
            \draw[string-red,thick, line cap = round,postaction={decorate}] (0,1.6) arc (0:180:0.25 and 0.25);
            \end{scope}    
            \filldraw[fill=white,draw=string-blue,thick] (-0.2,1) rectangle (1,1.5);
            \node at (0,0) [below,string-blue] {\scriptsize $X$};
            \node at (0.8,2) [above,string-blue] {\scriptsize $X$};
            \node at (0.5,0) [below,string-red] {\scriptsize $Y^*$};
            \node at (0.85,0) [below,string-red] {\scriptsize $Y$};
            \node at (0.4,1.25) [string-blue] {\scriptsize $\gamma$\tiny${}_Y$};
    \end{tikzpicture}   
\end{equation}
is an equivalence categories.

\end{lemma}
This lemma is a special case of \autocite[Thm.\,3.2]{Maj92braidcat}. 

Under the equivalence $\calZ(\calc) \simeq \calc_\coend$ the forgetful functor $U \colon \calZ(\calc) \to \calc$ simply corresponds to the functor forgetting the $\coend$-module structure. In particular, its left adjoint is given by the free module functor $-\otimes \coend \colon \calc \to \calc_\coend$. Moreover, since $\coend$ is a Frobenius algebra this is automatically also a right adjoint for $U$, as claimed above.

There is also a different functor $\calc \to \calZ(\calc)$ given by sending an object $X \in \calc$ to the object $(X,\beta_{X,-})\in \calZ(\calc)$ with half braiding given by the braiding $\beta_{X,-}$. In particular, we can equip every object in $\calc$ with a $\coend$-action by combining this functor with the above equivalence. We will call the resulting action the \emph{canonical} $\coend$-\emph{action on} $X$ and denote it with $\rho_X \colon X \otimes \coend \to X$. By taking the (left) partial trace $\lptr(\rho_X) \colon \coend \to \unit$ we get an element in $\chi_X \in \Hom_\calc(\coend,\unit)$ which is called the \emph{internal character of} $X$ \autocite[Sec.\,3.5]{Shimizu15character}.\footnote{
In \autocite{Shimizu15character} internal characters are defined
using the canonical end $\eend = \int_{X \in \calc} X \otimes X^*$ which can be identified with the dual of the coend $\coend$ via the Hopf pairing $\omega$ and under this identification the two versions of internal characters coincide.}

Next, recall the Deligne tensor product of two finite linear categories $\calA$ and $\calB$ is the finite linear category $\calA \boxtimes \calB$ together with a functor $\boxtimes \colon \calA \times \calB \to \calA \boxtimes \calB$ which is left exact in each argument and satisfies the following universal property: Let $\calD$ be another finite linear category and let $F \colon \calA \times \calB \to \calD$ be a functor left exact in each argument, then there is a unique (up to unique natural isomorphism)
left exact functor $\widehat{F}  \colon \calA \boxtimes \calB \to \calD$ together with an equivalence $F \cong \widehat{F} \circ \boxtimes$. See \autocite[Sec.\,1.11]{EGNO} for the existence of the Deligne product as well as more properties.\footnote{Usually one defines the Deligne tensor product for right exact functors, however by the equivalence of left exact and right exact functors for finite linear categories \autocite[Thm.\,3.2]{FSS19eilenbergwatts} this does not make a difference, see also the discussion in \autocite[Sec.\,2.4]{BW22modfunc}.} For $\calc$ and $\calD$ modular tensor categories, also $\calc\boxtimes \calD$ is a modular tensor category. This can for example be seen by noting that the isomorphism $\coend_{\calc \boxtimes \calD} \cong \coend_{\calc}\boxtimes \coend_{\calD}$ \autocite[Cor.\,3.12]{FSS19eilenbergwatts} where $\coend_\calc$ and $\coend_\calD$ are the canonical Hopf algebra of $\calc$ and $\calD$, respectively, is an isomorphism of Hopf algebras. 

Let us now turn to the question of comodules of $\coend$. First note that we can equip any object $X \in \calc$ with the structure of a right $\coend$-comodule with coaction $\delta_X$ given by
\begin{equation}
    \delta_X := (\id_X \otimes \iota_X) \circ (\mathrm{coev}_X \otimes \id_X) \colon X \to X \otimes \coend.
\end{equation}
We will call this the \emph{canonical} $\coend$-\emph{coaction on} $X$ and the (left) partial trace $\check{\chi}_X := \lptr(\delta_X) = \iota_X \circ \rcoev_X \colon \unit \to \coend$ the \emph{internal cocharacter of} $X$.

\begin{lemma}\label{lem:double-comodules}
    The functor
    \begin{equation}
        \begin{aligned}
            \calc \boxtimes \overline{\calc} &\to \calc^{\coend} \\
            X \boxtimes Y &\mapsto (X\otimes Y, \id_X \otimes \delta_Y)
        \end{aligned}
    \end{equation}
    is an equivalence of categories.
\end{lemma}
This is a special case of \autocite[Cor.\,2.7.2]{Lyu99squared}, see also \autocite[Prop.\,7.18.4]{EGNO} or \autocite[Rem.\,3.6]{Shimizu16pivotalcenter} for a proof based on the theory of exact module categories. 

Using the Hopf pairing $\omega\colon \coend \otimes \coend \to \unit$ we can turn comodules into modules yielding a functor $\calc^\coend \to \calc_\coend$. Under the equivalences of \Cref{lem:center-modules} and \ref{lem:double-comodules} this is precisely the ribbon functor $\calc\boxtimes\overline{\calc} \to \calZ(\calc)$ sending $X\boxtimes Y \in \calc\boxtimes\overline{\calc}$ to $(X\otimes Y,((\beta_{X,-}\otimes \id_Y) \circ(\id_X\otimes \overline{\beta}_{Y,-}))$. Moreover, since $\calc$ is assumed to be modular, i.e.\ $\omega$ is non-degenerate, this functor is an equivalence. 
The converse also holds:

\begin{proposition}[{\autocite[Thm.\,1.1]{SHIMIZU2019nondeg}}]\label{prop:modularchar}
    Let $\calB$ be a finite ribbon tensor category. Then $\calB$ is modular if and only if the canonical ribbon functor
    \begin{align}
        \calB \boxtimes \overline{\calB} \to \calZ(\calB)
    \end{align}
    is an equivalence.
\end{proposition}
Note that the functor $\calc\boxtimes\overline{\calc} \to \calc$ induced by the monoidal product $\otimes$ on $\calc$ corresponds to the forgetful functor $\calZ(\calc) \to \calc$ and thus also has a two-sided adjoint $\calc \to \calc\boxtimes\overline{\calc}$ which in this form is given by $F(X)= \int^{Y \in \calc} (Y^*\otimes X) \boxtimes Y \in \calc \boxtimes \overline{\calc}$. 

For an object $X \in \calc$ the canonical action $\rho_X$ and coaction $\delta_X$ satisfy $\rho_X = (\id_X \otimes \omega) \circ (\delta_X \otimes \id_\coend)$ which immediately gives $\chi_X = \omega \circ (\check{\chi}_X \otimes \id_\coend)$. 

Finally, using yet another characterisation of modularity\footnote{Namely the triviality of the Müger centre \autocite[Sec.\,4]{SHIMIZU2019nondeg}.} one can show that the endomorphisms obtained by combining the canonical coaction $\delta$ with the cointegral $\lambda$ has the following cutting property:
\begin{lemma}[{\autocite[Lem.\,6.3]{GR17projmod}}]\label{lem:cutting}
    For each $X \in \calc$ there exist $m\geq0$, $a \colon X \to \unit^{\oplus m}$, $b \colon  \unit^{\oplus m}\to X$ such that
    \begin{equation}
    \begin{aligned}
        (\id_X \otimes \lambda) \circ \delta_X &= b \circ a  \\
        &= \sum_{\alpha=1}^m  b_\alpha \circ a_\alpha
    \end{aligned}
    \end{equation}
    where $a_\alpha \colon X \to \unit$, $b_\alpha \colon  \unit \to X$ are the components of $a$, $b$.
\end{lemma}
In particular, for non-semisimple $\calc$, the case $m=0$ appears for example for $X=\unit$ since $(\id_\unit \otimes \lambda)\circ \delta_\unit = \lambda \circ \eta = 0$ by a monadic extension of Maschke's theorem \autocite[Thm.\,6.5]{BV07hopfmonads}.
Using \eqref{eq:integralconv} we get a similar statement for the integral $\Lambda$ as $\rho_X \circ (\id_X \otimes \Lambda)= \zeta (\id_X \otimes \lambda) \circ \delta_X $.
\section{3d defect TFTs and topological world sheets}\label{sec:deftft}

In this section we will discuss the topological and combinatorial setup needed for the considerations in the rest of the paper. In particular we will discuss ($2$-)categories of bordisms with labelled stratifications and the corresponding TFTs. More specifically we will first briefly discuss the category of $3$-dimensional defect bordisms and the corresponding defect TFTs introduced in \autocite{CMS16deftri}. Afterwards we will define a $(2,1)$-category of ``topological world sheets" in analogy to the world sheets defined in \autocite[Sec.\,2.2]{FSY2023RCFTstring1} by introducing defects to the $(2,1)$-category of $2$-dimensional open-closed bordisms.

\subsection{Defect TFTs}\label{sec:3ddeftft}
In this section we will review the notion of \emph{defect TFTs} of \autocite{CMS16deftri,CRS17defRT}. A $3$-\emph{dimensional defect TFT} is a symmetric monoidal functor
\begin{equation}
    \rmZ \colon \bord_{3,2}^{\mathrm{def}}(\DD) \to \vect,
\end{equation}
where the source category consists of stratified and labelled bordisms.

To make this precise, first recall that a $n$-dimensional \emph{stratified} manifold $M$ consists of an $n$-dimensional manifold together with a filtration $M = M_n \supset M_{n-1} \supset \dots \supset M_0 \supset M_{-1} = \varnothing$ satisfying a number of technical conditions including that $M_i/M_{i-1}$ is an $i$-dimensional submanifold of $M$ for $0 \leq i \leq n$. The connected components of $M_i/M_{i-1}$ will be called the $i$-\emph{strata} of $M$. We will not make the other conditions explicit here as we will focus on a specific class of stratified manifolds and instead refer to \autocite[Sec.\,2.1]{CMS16deftri} for the full definition. The class of stratified manifolds we will work with is the one discussed in \autocite[Sec.\,2.2.1]{CRS17deforbi}, called defect manifolds there, for which local neighbourhoods are defined inductively. Instead of repeating the full definition here we will instead roughly go through it up to dimension $n \leq 3$ as this will cover all cases of interest for us:

For $n=1$ the induction starts and we have three possible local neighbourhoods: The open oriented interval $(-1,1)$, as well as the interval $(-1,1)$ with a single $0$-stratum at $0$ which is either positively or negatively oriented:
    \begin{equation}
        \begin{tikzpicture}
            \begin{scope}[decoration={markings, mark=at position 0.55 with {\arrow{>}}}] 
            \draw[line cap=round,thick,postaction={decorate},string-yellow] (0,0) -- (1.5,0);
            \end{scope}
        \end{tikzpicture}\, ,\quad
        \begin{tikzpicture}
            \begin{scope}[decoration={markings, mark=at position 0.55 with {\arrow{>}}}] 
            \draw[line cap=round,thick,postaction={decorate},string-yellow] (0,0) -- (0.75,0);
             \draw[line cap=round,thick,postaction={decorate},string-yellow] (0.75,0) -- (1.5,0);
            \end{scope}
            \fill[magenta] (0.75,0) circle (0.06);
            \node at (0.75,0) [above,magenta] {\scriptsize $+$};
        \end{tikzpicture}\, ,\quad
        \begin{tikzpicture}
            \begin{scope}[decoration={markings, mark=at position 0.55 with {\arrow{>}}}] 
            \draw[line cap=round,thick,postaction={decorate},string-yellow] (0,0) -- (0.75,0);
             \draw[line cap=round,thick,postaction={decorate},string-yellow] (0.75,0) -- (1.5,0);
            \end{scope}
            \fill[magenta] (0.75,0) circle (0.06);
            \node at (0.75,0) [above,magenta] {\scriptsize $-$};
        \end{tikzpicture}\, ;
    \end{equation}

 For $n=2$ we have the first induction step leading us to consider two types of local neighbourhoods: The first type is given by cylinders over the one dimensional local neighbourhoods with the induced orientation:
\begin{equation}\label{eq:2dlocngbh1}
        \begin{tikzpicture}[baseline=0.7cm]
            \fill[white!70!string-yellow] (0,0) rectangle (1.5,1.5);
        \end{tikzpicture}\, ,\quad
        \begin{tikzpicture}[baseline=0.7cm]
            \fill[white!70!string-yellow] (0,0) rectangle (1.5,1.5);
            \begin{scope}[decoration={markings, mark=at position 0.55 with {\arrow{>}}}] 
            \draw[line cap=round,thick,postaction={decorate},magenta] (0.75,0) -- (0.75,1.5);
            \end{scope}
        \end{tikzpicture}\, ,\quad
        \begin{tikzpicture}[baseline=0.7cm]
           \fill[white!70!string-yellow] (0,0) rectangle (1.5,1.5);
            \begin{scope}[decoration={markings, mark=at position 0.55 with {\arrow{>}}}] 
            \draw[line cap=round,thick,postaction={decorate},magenta] (0.75,1.5) -- (0.75,0);
            \end{scope}
        \end{tikzpicture}\, ;
    \end{equation}
Note that the last two are equivalent since there is an orientation preserving diffeomorphism between them. The second type is given by cones over $1$-dimensional stratified circles such as 
\begin{equation}
    \begin{tikzpicture}[baseline=0cm]
        \draw[string-yellow,very thick] (0,0) circle (1);
        \filldraw[magenta] (90:1) circle (0.06);
        \filldraw[magenta] (210:1) circle (0.06);
        \filldraw[magenta] (330:1) circle (0.06);
        \node at  (90:1) [above,magenta] {\scriptsize $-$};
        \node at  (210:1) [left,magenta] {\scriptsize $+$};
        \node at  (330:1) [right,magenta] {\scriptsize $-$};
    \end{tikzpicture}
\end{equation}
(with an arbitrary number of $0$-strata) with induced orientation on the $1$- and $2$-strata and two possible orientations on the $0$-stratum given by the cone point: 
\begin{equation}\label{eq:2dlocngbh2}
    \begin{tikzpicture}[baseline=0cm]
        \fill[white!70!string-yellow] (0,0) circle (1);
        \fill[white!70!string-yellow] (210:1) -- (0,0) -- (330:1)
        (210:1) arc (210:330:1);
        \fill[white!70!string-yellow] (90:1) -- (0,0) -- (210:1)
        (90:1) arc (90:210:1);
        \begin{scope}[decoration={
                      markings,
                      mark=at position 0.52 with {\arrow{>}}}
                     ] 
        \draw[magenta,thick,postaction={decorate}] (0,0) -- (90:1) ;
        \draw[magenta,thick,postaction={decorate}] (210:1)  -- (0,0);
        \draw[magenta,thick,postaction={decorate}] (0,0)  -- (330:1);
        \end{scope}
        \filldraw[string-blue] (0,0) circle (0.06);
        \node at (0,0) [below,string-blue] {\scriptsize $+$};
    \end{tikzpicture}\, ,\quad 
    \begin{tikzpicture}[baseline=0cm]
        \fill[white!70!string-yellow] (0,0) circle (1);
        \fill[white!70!string-yellow] (210:1) -- (0,0) -- (330:1)
        (210:1) arc (210:330:1);
        \fill[white!70!string-yellow] (90:1) -- (0,0) -- (210:1)
        (90:1) arc (90:210:1);
        \begin{scope}[decoration={
                      markings,
                      mark=at position 0.52 with {\arrow{>}}}
                     ] 
        \draw[magenta,thick,postaction={decorate}] (0,0) -- (90:1) ;
        \draw[magenta,thick,postaction={decorate}] (210:1)  -- (0,0);
        \draw[magenta,thick,postaction={decorate}] (0,0)  -- (330:1);
        \end{scope}
        \filldraw[string-blue] (0,0) circle (0.06);
        \node at (0,0) [below,string-blue] {\scriptsize $-$};
    \end{tikzpicture}\, ;
\end{equation}
    
Finally, for $n=3$ we repeat the previous step and again consider cylinders over the two dimensional local neighbourhoods as well as cones over stratified $2$-spheres, with the induced orientation. The cylinders can further be distinguished as cylinders over cylinders over the one dimensional local neighbourhoods:
    \begin{equation}
        \begin{tikzpicture}[baseline=0.75cm,scale=1.3]
            \draw[dotted, line cap=round,thin, white!70!string-yellow] (0,0) -- (1,0.75)
            (2.5,0.25) -- (1,0.75)
            (1,2.25) -- (1,0.75);
            \fill[white!70!string-yellow,opacity=0.3] (0,0) -- (1.5,-0.5) -- (1.5,1) -- (0,1.5) -- cycle;
            \fill[white!70!string-yellow,opacity=0.3] (1.5,1) -- (2.5,1.75) -- (2.5,0.25) --  (1.5,-0.5);
            \fill[white!70!string-yellow,opacity=0.3] (0,1.5) -- (1,2.25) -- (2.5,1.75) -- (1.5,1) -- cycle;
            \draw[thin,white!70!string-yellow] (0,0) -- (1.5,-0.5) -- (1.5,1) -- (0,1.5) -- cycle;
            \draw[thin,white!70!string-yellow] (1.5,1) -- (2.5,1.75) -- (2.5,0.25) --  (1.5,-0.5);
            \draw[thin,white!70!string-yellow] (0,1.5) -- (1,2.25) -- (2.5,1.75);
        \end{tikzpicture}\, , \quad
        \begin{tikzpicture}[baseline=0.75cm,scale=1.3]
            \draw[dotted, line cap=round,thin, white!70!string-yellow] (0,0) -- (1,0.75)
            (2.5,0.25) -- (1,0.75)
            (1,2.25) -- (1,0.75);
            \fill[magenta, opacity=0.2] (1.75,2) -- (1.75,0.5) -- (0.75,-0.25) -- (0.75,1.25) -- cycle;
            \draw[thin,white!40!magenta,line cap=round,dotted] (1.75,2) -- (1.75,0.5);
            \draw[thin,white!40!magenta,line cap=round,dotted] (1.75,0.5) -- (0.75,-0.25);
            \fill[white!70!string-yellow,opacity=0.3] (0,0) -- (1.5,-0.5) -- (1.5,1) -- (0,1.5) -- cycle;
            \fill[white!70!string-yellow,opacity=0.3] (1.5,1) -- (2.5,1.75) -- (2.5,0.25) --  (1.5,-0.5);
            \fill[white!70!string-yellow,opacity=0.3] (0,1.5) -- (1,2.25) -- (2.5,1.75) -- (1.5,1) -- cycle;
            \draw[thin,white!70!string-yellow,line cap=round] (0,0) -- (1.5,-0.5) -- (1.5,1) -- (0,1.5) -- cycle;
            \draw[thin,white!70!string-yellow,line cap=round] (1.5,1) -- (2.5,1.75) -- (2.5,0.25) --  (1.5,-0.5);
            \draw[thin,white!70!string-yellow,line cap=round] (0,1.5) -- (1,2.25) -- (2.5,1.75);
            \draw[thin,white!40!magenta,line cap=round] (0.75,1.25) -- (1.75,2);
            \draw[thin,white!40!magenta,line cap=round] (0.75,1.25) -- (0.75,-0.25);
            \draw[ultra thin,-{>[length=0.2mm]},line cap=round,magenta] (1.6,1.7) -- (1.65,1.735);
            \draw[ultra thin,-{>[length=0.2mm]},line cap=round,magenta] (1.6,1.7) -- (1.6,1.79);   
        \end{tikzpicture}\, , \quad
        \begin{tikzpicture}[baseline=0.75cm,scale=1.3]
            \draw[dotted, line cap=round,thin, white!70!string-yellow] (0,0) -- (1,0.75)
            (2.5,0.25) -- (1,0.75)
            (1,2.25) -- (1,0.75);
            \fill[magenta, opacity=0.2] (1.75,2) -- (1.75,0.5) -- (0.75,-0.25) -- (0.75,1.25) -- cycle;
            \draw[thin,white!40!magenta,line cap=round,dotted] (1.75,2) -- (1.75,0.5);
            \draw[thin,white!40!magenta,line cap=round,dotted] (1.75,0.5) -- (0.75,-0.25);
            \fill[white!70!string-yellow,opacity=0.3] (0,0) -- (1.5,-0.5) -- (1.5,1) -- (0,1.5) -- cycle;
            \fill[white!70!string-yellow,opacity=0.3] (1.5,1) -- (2.5,1.75) -- (2.5,0.25) --  (1.5,-0.5);
            \fill[white!70!string-yellow,opacity=0.3] (0,1.5) -- (1,2.25) -- (2.5,1.75) -- (1.5,1) -- cycle;
            \draw[thin,white!70!string-yellow,line cap=round] (0,0) -- (1.5,-0.5) -- (1.5,1) -- (0,1.5) -- cycle;
            \draw[thin,white!70!string-yellow,line cap=round] (1.5,1) -- (2.5,1.75) -- (2.5,0.25) --  (1.5,-0.5);
            \draw[thin,white!70!string-yellow,line cap=round] (0,1.5) -- (1,2.25) -- (2.5,1.75);
            \draw[thin,white!40!magenta,line cap=round] (0.75,1.25) -- (1.75,2);
            \draw[thin,white!40!magenta,line cap=round] (0.75,1.25) -- (0.75,-0.25);
            \draw[ultra thin,-{>[length=0.2mm]},line cap=round,magenta] (1.6,1.7) -- (1.65,1.735);
            \draw[ultra thin,-{>[length=0.2mm]},line cap=round,magenta] (1.6,1.7) -- (1.6,1.62);   
        \end{tikzpicture}\,,
    \end{equation}
    where the little coordinate systems in the second and third example illustrate the differing orientation on the $2$-stratum.\footnote{To be a bit more precise we consider the orientations on the $2$-strata such that the vertically drawn arrows correspond to the second tangent direction.} As well as cylinders over the two dimensional cones:
    \begin{equation}
        \begin{tikzpicture}[baseline=1cm,scale=1.5]
            \fill[magenta, opacity=0.2] ([yshift=2cm]20:0.75 and 0.4) -- (20:0.75 and 0.4) -- (0,0) -- (0,2) -- cycle;
            \fill[magenta, opacity=0.2] ([yshift=2cm]240:0.75 and 0.4) -- (240:0.75 and 0.4) -- (0,0) -- (0,2) -- cycle;
            \fill[magenta, opacity=0.2] ([yshift=2cm]140:0.75 and 0.4) -- (140:0.75 and 0.4) -- (0,0) -- (0,2) -- cycle;
            \draw[dotted, line cap=round,thin, white!70!string-yellow] (0.75,0) arc (0:180:0.75 and 0.4);
             \draw[thin,white!40!magenta,line cap=round,dotted] ([yshift=2cm]140:0.75 and 0.4) -- (140:0.75 and 0.4)
             ([yshift=2cm]20:0.75 and 0.4) -- (20:0.75 and 0.4)
             (0,0) -- (140:0.75 and 0.4)
            (0,0) -- (240:0.75 and 0.4)
            (0,0) -- (20:0.75 and 0.4);
            \fill[white!70!string-yellow,opacity=0.3] (0,0) ellipse (0.75 and 0.4);
            \fill[white!70!string-yellow,opacity=0.3] (0,2) ellipse (0.75 and 0.4);
            \fill[white!70!string-yellow,opacity=0.3] (-0.75,0) -- (-0.75,2) arc (180:360:0.75 and 0.4) -- (0.75,0) arc (0:180:0.75 and 0.4);
            \draw[thin,white!70!string-yellow,line cap=round] (0,2) ellipse (0.75 and 0.4)
            (-0.75,0) -- (-0.75,2) 
            (0.75,0) -- (0.75,2)
            (0.75,0) arc (0:-180:0.75 and 0.4);
            \draw[thin,white!40!magenta,line cap=round] (0,2) -- ([yshift=2cm]140:0.75 and 0.4)
            (0,2) -- ([yshift=2cm]240:0.75 and 0.4)
            (0,2) -- ([yshift=2cm]20:0.75 and 0.4)
            ([yshift=2cm]240:0.75 and 0.4) --(240:0.75 and 0.4);
            \begin{scope}[decoration={
                      markings,
                      mark=at position 0.52 with {\arrow{>[length=0.6mm]}}}
                     ] 
        \draw[string-blue,postaction={decorate},line cap=round,thick] (0,0) -- (0,2);
        \draw[thin,white!70!string-yellow,line cap=round] ([yshift=2cm]250:0.75 and 0.4) arc (250:280:0.75 and 0.4);
        \end{scope}
        \end{tikzpicture}\, ,\quad
        \begin{tikzpicture}[baseline=1cm,scale=1.5]
            \fill[magenta, opacity=0.2] ([yshift=2cm]20:0.75 and 0.4) -- (20:0.75 and 0.4) -- (0,0) -- (0,2) -- cycle;
            \fill[magenta, opacity=0.2] ([yshift=2cm]240:0.75 and 0.4) -- (240:0.75 and 0.4) -- (0,0) -- (0,2) -- cycle;
            \fill[magenta, opacity=0.2] ([yshift=2cm]140:0.75 and 0.4) -- (140:0.75 and 0.4) -- (0,0) -- (0,2) -- cycle;
            \draw[dotted, line cap=round,thin, white!70!string-yellow] (0.75,0) arc (0:180:0.75 and 0.4);
             \draw[thin,white!40!magenta,line cap=round,dotted] ([yshift=2cm]140:0.75 and 0.4) -- (140:0.75 and 0.4)
             ([yshift=2cm]20:0.75 and 0.4) -- (20:0.75 and 0.4)
             (0,0) -- (140:0.75 and 0.4)
            (0,0) -- (240:0.75 and 0.4)
            (0,0) -- (20:0.75 and 0.4);
            \fill[white!70!string-yellow,opacity=0.3] (0,0) ellipse (0.75 and 0.4);
            \fill[white!70!string-yellow,opacity=0.3] (0,2) ellipse (0.75 and 0.4);
            \fill[white!70!string-yellow,opacity=0.3] (-0.75,0) -- (-0.75,2) arc (180:360:0.75 and 0.4) -- (0.75,0) arc (0:180:0.75 and 0.4);
            \draw[thin,white!70!string-yellow,line cap=round] (0,2) ellipse (0.75 and 0.4)
            (-0.75,0) -- (-0.75,2) 
            (0.75,0) -- (0.75,2)
            (0.75,0) arc (0:-180:0.75 and 0.4);
            \draw[thin,white!40!magenta,line cap=round] (0,2) -- ([yshift=2cm]140:0.75 and 0.4)
            (0,2) -- ([yshift=2cm]240:0.75 and 0.4)
            (0,2) -- ([yshift=2cm]20:0.75 and 0.4)
            ([yshift=2cm]240:0.75 and 0.4) --(240:0.75 and 0.4);
            \begin{scope}[decoration={
                      markings,
                      mark=at position 0.52 with {\arrow{>[length=0.6mm]}}}
                     ] 
        \draw[string-blue,postaction={decorate},line cap=round,thick] (0,2) -- (0,0);
        \end{scope}
        \draw[thin,white!70!string-yellow,line cap=round] ([yshift=2cm]250:0.75 and 0.4) arc (250:280:0.75 and 0.4);
        \end{tikzpicture};
    \end{equation}
     Finally, the last type of local neighbourhoods are cones over stratified $2$-spheres with the induced orientation on the $1$-, $2$-, and $3$-strata and two possible orientations on the $0$-stratum given by the cone point:
     \begin{equation}
         \begin{tikzpicture}[baseline=0cm,scale=1.3]
            \fill[magenta, opacity=0.2] (0,0) ellipse (1.4 and 0.5);
            \draw[thin,white!40!magenta,line cap=round,dotted] (1.4,0) arc (0:180:1.4 and 0.5);
            \fill[magenta, opacity=0.2](0,0) ellipse (0.8 and 1.4);
            \fill[magenta, opacity=0.2] (0,0) -- (0,1.4) arc (90:-15:1 and 1.4) -- cycle;
            \draw[thin,white!40!magenta,line cap=round,dotted]
            (0,1.4) arc (90:-90:0.8 and 1.4);
            \fill[white!70!string-yellow,opacity=0.3] (0,0) circle (1.4);
            \draw[thin,white!70!string-yellow,line cap=round] (0,0) circle (1.4);
            \draw[thin,white!40!magenta,line cap=round] (1.4,0) arc (0:-180:1.4 and 0.5)
            (0,1.4) arc (90:270:0.8 and 1.4)
            (0,1.4) arc (90:-15:1 and 1.4);
            \begin{scope}[decoration={
                      markings,
                      mark=at position 0.52 with {\arrow{>[length=0.6mm]}}}
                     ] 
            \draw[string-blue,postaction={decorate},line cap=round,thick] (0,0) -- (0,1.4);
            \draw[string-blue,postaction={decorate},line cap=round,thick] (0,0) -- (-15:1 and 1.4);
            \draw[string-blue,postaction={decorate},line cap=round,thick] (237:1.4 and 0.5) -- (0,0);
            \draw[string-blue,postaction={decorate},line cap=round,thick] (57:1.4 and 0.5) -- (0,0);
            \end{scope}
            \fill[black] (0,0) circle (0.04);
            \node at (0,0) [below,black] {\scriptsize $+$};
        \end{tikzpicture}\, , \quad
        \begin{tikzpicture}[baseline=0cm,scale=1.3]
            \fill[magenta, opacity=0.2] (0,0) ellipse (1.4 and 0.5);
            \draw[thin,white!40!magenta,line cap=round,dotted] (1.4,0) arc (0:180:1.4 and 0.5);
            \fill[magenta, opacity=0.2](0,0) ellipse (0.8 and 1.4);
            \fill[magenta, opacity=0.2] (0,0) -- (0,1.4) arc (90:-15:1 and 1.4) -- cycle;
            \draw[thin,white!40!magenta,line cap=round,dotted]
            (0,1.4) arc (90:-90:0.8 and 1.4);
            \fill[white!70!string-yellow,opacity=0.3] (0,0) circle (1.4);
            \draw[thin,white!70!string-yellow,line cap=round] (0,0) circle (1.4);
            \draw[thin,white!40!magenta,line cap=round] (1.4,0) arc (0:-180:1.4 and 0.5)
            (0,1.4) arc (90:270:0.8 and 1.4)
            (0,1.4) arc (90:-15:1 and 1.4);
            \begin{scope}[decoration={
                      markings,
                      mark=at position 0.52 with {\arrow{>[length=0.6mm]}}}
                     ] 
            \draw[string-blue,postaction={decorate},line cap=round,thick] (0,0) -- (0,1.4);
            \draw[string-blue,postaction={decorate},line cap=round,thick] (0,0) -- (-15:1 and 1.4);
            \draw[string-blue,postaction={decorate},line cap=round,thick] (237:1.4 and 0.5) -- (0,0);
            \draw[string-blue,postaction={decorate},line cap=round,thick] (57:1.4 and 0.5) -- (0,0);
            \end{scope}
            \fill[black] (0,0) circle (0.04);
            \node at (0,0) [black,below] {\scriptsize $-$};
        \end{tikzpicture}\, ;
     \end{equation}

With the underlying topology out of the way let us now come to the labelling data:
The labels for the strata are encoded in the so-called \emph{defect data}
\begin{equation}
    \DD = (D_3,D_2,D_1;s,t,j).
\end{equation}
Here, the $D_i$, for $i \in \{1,2,3\}$, are sets whose elements will label the $i$-dimensional strata of the $3$-dimensional stratified manifolds while the source, target, and junction map $s,t\colon D_2 \times \{\pm\}\to D_3$ and $j\colon D_1 \times \{\pm\} \to (\textrm{ordered lists of elements of } D_2)$ encode the adjacency conditions for the labels including possible orientations, see \autocite{CMS16deftri,CRS17deforbi} for the detailed definition. Note that in contrast to \autocite{CMS16deftri,CRS17deforbi} we do not use a \emph{cyclic} order for the junction map $j$ but a linear order instead. This extra choice of order will be reflected in the labelling below.

To account for $0$-strata we would need a fourth set $D_0$ together with an `adjacency map' out of it to continue, see \autocite[Def.\,2.4]{CRS17deforbi}. However, it turns out that that for a given defect TFT, there is a canonical way to add this extra data by considering certain invariant vectors in the state spaces of defect spheres \autocite[Sect.\,2.4]{CRS17deforbi}. This process is referred to as ``$D_0$-completion'' and from now on we will always assume that our defect data $\DD$ has been $D_0$-completed. Due to this we will not discuss the combinatorial description of labelling $0$-strata in a $3$-manifold.

Since we want to construct a bordism category we need to label stratified surfaces as well as stratified $3$-manifolds with boundary. In particular, as explained above, the only local neighbourhoods we have to label are cylinders over $2$-dimensional local neighbourhoods which is the same as labelling the $2$-dimensional local neighbourhoods up to an index shift.

For a stratified surface we label every $i$-stratum for $i \in \{0,1,2\}$ with an element in the set $D_{i+1}$, such that the labels of the strata in its neighbourhood are compatible with the source, target, and junction map.
Let us explain how this works in the example of the first $2$-dimensional cone above:
\begin{equation}
    \begin{tikzpicture}[baseline=0cm]
        \fill[white!70!string-yellow] (0,0) circle (1.7);
        \fill[white!70!string-yellow] (210:1.7) -- (0,0) -- (330:1.7)
        (210:1.7) arc (210:330:1.7);
        \fill[white!70!string-yellow] (90:1.7) -- (0,0) -- (210:1.7)
        (90:1.7) arc (90:210:1.7);
        \begin{scope}[decoration={
                      markings,
                      mark=at position 0.52 with {\arrow{>}}}
                     ] 
        \draw[magenta,very thick,postaction={decorate}] (0,0) -- (90:1.7) ;
        \draw[magenta,very thick,postaction={decorate}] (210:1.7)  -- (0,0);
        \draw[magenta,very thick,postaction={decorate}] (0,0)  -- (330:1.7);
        \end{scope}
        \filldraw (240:1.7) circle (0.06);
        \filldraw[string-blue] (0,0) circle (0.06);
        \node at (330:2) [magenta] {$A$};
        \node at (90:2) [magenta] {$B$};
        \node at (210:2) [magenta] {$C$};
        \node at (270:0.5) [string-blue] {$X$};
        \node at (270:2) [string-yellow] {$\beta$};
        \node at (30:2) [string-yellow] {$\alpha$};
        \node at (150:2) [string-yellow] {$\gamma$};
        \node at (237:2) [black] {\scriptsize $-1$};
    \end{tikzpicture}
\end{equation}
The $2$-strata inherit the orientation of the whole underlying manifold, which in this case we take to be the standard orientation of the unit disc in $\mathbbm{R}^2$. 
The $0$-stratum is positively oriented and we label it with $X \in D_1$, then the labels $A,B,C \in D_2$ of the three $1$-strata need to satisfy $j(X,+) = ((C,-),(B,+),(A,+))$, where the signs are chosen in a way to indicate if the corresponding $1$-stratum is oriented towards or away from the $0$-stratum. The linear order of the list $((C,-),(B,+),(A,+))$ is obtained by going in clockwise direction along the stratified circle used to obtain the cone starting from the image of the south pole $-1 \in S^1$, illustrated as a black dot above. The labels need also be compatible with $2$-strata in the form of $s(A,+) = t(C,-) = \beta$, $t(A,+) = s(B,+) = \alpha$, and $t(B,+) = s(C,-) = \gamma$. We want to note here that which side of the $1$-stratum is the source and which one is the target is a convention, see \autocite[Sect.\,2.3]{CMS16deftri} for details.

To label bordisms between stratified and labelled surfaces we require the labels of the $i$-strata in the interior to match the $(i-1)$-strata on the boundary. The symmetric monoidal category $\bord_{3,2}^{\mathrm{def}}(\DD)$ of $3$-\emph{dimensional defect bordisms} with defect data $\DD$ now consists of stratified and labelled surfaces and equivalence classes of stratified and labelled bordisms between them, see \autocite[Def.\,2.4]{CRS17deforbi} for details. As in \autocite[Sec.\,4.3]{DGGPR19} it will be necessary to extend this bordism category to account for a possible gluing anomaly by adding Lagrangian subspaces and signature defects. We will denote the so obtained category by $\bord_{3,2}^{\chi,\mathrm{ def}}(\DD)$.

\begin{definition}
    A $3$-dimensional defect TFT with defect data $\DD$ is a symmetric monoidal functor 
    \begin{equation}
        \rmZ \colon \bord_{3,2}^{\chi,\mathrm{ def}}(\DD) \to \vect.
    \end{equation}
\end{definition}
For a $3$d defect TFT $\rmZ$ we will sometimes refer to the labelled $i$-dimensional strata for $i \in \{0,1,2\}$ as the \emph{point, line}, and \emph{surface defects} of $\rmZ$ and the labelled $3$-strata as the bulk phases of $\rmZ$.

\begin{remark}\label{rem:def3cat}
To any $3$d defect TFT $\rmZ$ with defect data $\DD$ one can associate a certain type of $3$-category $\calT_\rmZ$. The objects of $\calT_\rmZ$ are given by the elements of $D_3$, the $1$-morphisms by (lists of) elements of $D_2$, the $2$-morphisms by (lists of) elements of $D_1$, and the $3$-morphisms from state spaces of defect spheres as in the $D_0$-completion mentioned above. 
The rest of the data is obtained from the source, target, and junction maps. For more details see \autocite[Sect.\,3]{CMS16deftri}. We will refer to $\calT_\rmZ$ as the \emph{defect} $3$-category of $\rmZ$
\end{remark}

Let us now briefly describe how the $3$d TFT with embedded ribbon graphs of \autocite{DGGPR19} can be understood as a defect TFT in analogy to \autocite[Rem.\,5.9\,(ii)]{CRS17defRT} or \autocite[Sec.\,4.2]{CMS16deftri}. More precisely, let
\begin{equation}
        \widehat{\rmV}_\calc \colon \widehat{\bord}_{3,2}^{\chi}(\calc) \to \vect
\end{equation}
be the $3$d TFT constructed from the modular tensor category $\calc$ with admissibility condition on the outgoing boundary components as in \autocite[Sec.\,3]{HR2024modfunc}. The defect data is given by $D_3=\{*\}$,  $D_2=\{\unit\}$, $D_1 = \mathrm{ob}(\calc)$, and $D_0 = \mathrm{Mor}(\calc)$. The source target and junction map are trivial in this case. However, the map one obtains from $D_0$-completion is encoded via the structural morphisms of $\calc$. A bordism with embedded ribbon graph is turned into a bordism with defects by viewing an $X\in \calc$ labelled ribbon as a $\unit$-labelled surface defect bounded by an $X$-labelled line defect on one side and by a $\unit$-labelled line defect on the other side:
\begin{align}
    \begin{tikzpicture}[baseline=-0.1cm]
    \begin{scope}[decoration={markings, mark=at position 0.55 with {\arrow{>}}}] 
        \draw[ultra thick,string-blue,line cap = round,postaction={decorate}]
        (0,-0.7) -- (0,0.7);
    \end{scope}
    \node at (-0.1,0) [string-blue,left] {\scriptsize $X$};
    \end{tikzpicture} \;
    \longleftrightarrow
    \begin{tikzpicture}[baseline=-0.1cm]
    \fill[magenta, opacity=0.1] (0,-0.71) rectangle (0.5,0.71);
    \begin{scope}[decoration={markings, mark=at position 0.55 with {\arrow{>}}}] 
        \draw[thick,string-blue,line cap = round,postaction={decorate}]
        (0,-0.7) -- (0,0.7);
        \draw[magenta,line cap = round]
        (0.5,-0.7) -- (0.5,0.7);
    \end{scope}
    \node at (-0.1,0) [string-blue,left] {\scriptsize $X$};
    \node at (0.25,0.2) [magenta] {\scriptsize $\unit$};
    \node at (0.5,0) [magenta,right] {\scriptsize $\unit$};
    \end{tikzpicture} 
\end{align}

The coupons can be included via point defects in a similar way. From now on we will often view $\widehat{\rmV}_\calc$ as a defect TFT in this way without further comment. The corresponding $3$-category $\calT_{\widehat{\rmV}_\calc}$ is $B^2\calc$, i.e.\ the $3$-category with one object, only the identity $1$-morphism, and $\calc$ as $2$-endomorphism category \autocite[Prop.\,4.4]{CMS16deftri}.

\subsection{Topological world sheets}\label{subsec:topworld}
In this section we will define the $2$-category of \emph{topological world sheets}. The basic idea is to enhance the $2$-category of open-closed bordism to a $2$-category of open-closed \emph{defect} bordisms as we did in the last section with the ordinary bordism category. In this article a $2$-category will always mean a weak $2$-category otherwise known as a bicategory. Analogously a $2$-functor will always mean a weak $2$-functor with coherence isomorphisms often called pseudofunctor, see for example \autocite{JY20twodcategories}. For more background on symmetric monoidal $2$-categories and $2$-functors we refer to \autocite[Ch.\,2]{SchommerPries2011thesis}, a concise review is given in \autocite[App.\,D]{DeRenzi2022nonssext1}.

Let us first recall the symmetric monoidal $2$-category $\bord_{2+\epsilon,2,1}^{\oc}$ of two dimensional oriented open-closed bordisms as a categorification of the category of open-closed bordisms defined in \autocite[Sect.\,3]{LP2008octqfts}, see \autocite{BCR2006ocbordcat} for a detailed account. As objects we consider compact one dimensional manifolds with boundaries, consequentially bordisms between them need to have an underlying manifold with corners, more precisely a $\langle2\rangle$-manifold. A two dimensional $\langle2\rangle$-manifold $\Sigma$ is a $2$-dimensional compact manifold with corners together with a decomposition $\partial \Sigma = \partial^{\mathrm{g}} \Sigma \cup \partial^{\mathrm{f}} \Sigma$ of its boundary into a \emph{gluing} boundary $\partial^{\mathrm{g}} \Sigma$ and a \emph{free} boundary $\partial^{\mathrm{f}} \Sigma$ such that $\partial^{\mathrm{g}} \Sigma \cap \partial^{\mathrm{f}} \Sigma$ consists of the corner points of $\Sigma$. The open-closed bordism $2$-category $\bord_{2+\epsilon,2,1}^{\oc}$ is now defined as follows:
    \begin{itemize}
    \item objects: compact one-dimensional manifolds;
    \item $1$-morphisms: for objects $\Gamma$ and $\Gamma'$, a $1$-morphism $\Sigma$ from $\Gamma$ to $\Gamma'$, denoted as $\Sigma \colon \Gamma \to \Gamma'$,
    is a two-dimensional open-closed bordism $\Sigma \colon \Gamma \to \Gamma'$, i.e.\ a $\langle2\rangle$-manifold together with a parametrisation of its gluing boundary $\partial^{\mathrm{g}}\Sigma \cong -\Gamma \sqcup \Gamma'$;
    \item $2$-morphisms: for $1$-morphisms $\Sigma \colon \Gamma \to \Gamma'$ and $\Sigma' \colon \Gamma \to \Gamma'$ a $2$-morphism $[f] \colon \Sigma \Rightarrow \Sigma'$ is the isotopy class of a diffeomorphism $f \colon \Sigma \to \Sigma'$ of $\langle2\rangle$-manifolds which is compatible with the boundary parametrisations;
    \item horizontal composition: for $1$-morphisms $\Sigma \colon \Gamma \to \Gamma'$ and $\Sigma' \colon \Gamma' \to \Gamma''$ the horizontal composition $\Sigma'\diamond \Sigma \colon \Gamma \to \Gamma''$ is given by the $\langle 2\rangle$-manifold
    \begin{equation}
    \Sigma' \sqcup_{\Gamma'} \Sigma
    \end{equation}
    obtained by gluing along $\Gamma'$ using the parametrisation of the gluing boundaries of $\Sigma$ and $\Sigma'$, respectively.
    \item identity $1$-morphism: for an object $\Gamma$ the identity $1$-morphism $\id_\Gamma \colon \Gamma \to \Gamma$ is 
    \begin{equation}
    \Gamma\times I;
    \end{equation}    
    \item vertical composition: 
    for $2$-morphisms $[f] \colon (\Sigma,\lambda) \Rightarrow (\Sigma',\lambda')$ and
    \\
    \noindent
    $[g] \colon (\Sigma',\lambda')$ $\Rightarrow$ $(\Sigma'',\lambda'')$,     
    the vertical composition is given by the $2$-morphism 
    \begin{equation}
    [g\circ f];
    \end{equation}
    \item identity $2$-morphism: for a $1$-morphism $\Sigma \colon \Gamma \to \Gamma'$ the identity $2$-morphism $\id_{\Sigma} \colon \Sigma \Rightarrow \Sigma$ is  
    \begin{equation}
    [\id];
    \end{equation} 
    \item disjoint union as symmetric monoidal structure;
\end{itemize}
There is more coherence data which needs to be specified such as the associativity $2$-morphisms for horizontal composition, however we will skip these details, see for example \autocite[App.\,B \& C]{DeRenzi2022nonssext1} as well as \autocite[Rem.\,2.2]{FSY2023RCFTstring1} for more details.

As in \Cref{sec:3ddeftft} we will fix a class of stratified, open-closed surfaces by describing the allowed local neighbourhoods. We will call these \emph{unlabelled topological world sheets}, or often also just \emph{topological world sheets}.

In the interior, of a topological world sheet, we allow for the $2$-dimensional local neighbourhoods \eqref{eq:2dlocngbh1} and \eqref{eq:2dlocngbh2} discussed in \Cref{sec:3ddeftft}. Near the boundary components, we will need different local models, depending on the type of boundary. 

Firstly, near the gluing boundary we allow for:
\begin{equation}
    \begin{tikzpicture}[baseline=0.7cm]
        \fill[white!70!magenta] (0,0) rectangle (1.5,1.5);
        \begin{scope}[decoration={markings, mark=at position 0.55 with {\arrow{>}}}] 
            \draw[line cap=round,thick,magenta] (0.01,0) -- (1.49,0);
            \end{scope}
    \end{tikzpicture}\,,\quad
    \begin{tikzpicture}[baseline=0.7cm]
        \fill[white!70!magenta] (0,0) rectangle (1.5,1.5);
        \begin{scope}[decoration={markings, mark=at position 0.55 with {\arrow{>}}}]
            \draw[line cap=round,thick,magenta] (0.01,0) -- (1.49,0);
            \draw[line cap=round,thick,postaction={decorate},string-blue] (0.75,0) -- (0.75,1.49);
        \end{scope}
        \fill[string-blue] (0.75,0) circle (0.03);
    \end{tikzpicture}\,,\quad
    \begin{tikzpicture}[baseline=0.7cm]
        \fill[white!70!magenta] (0,0) rectangle (1.5,1.5);
        \begin{scope}[decoration={markings, mark=at position 0.55 with {\arrow{>}}}]
            \draw[line cap=round,thick,magenta] (0.01,0) -- (1.49,0);
            \draw[line cap=round,thick,postaction={decorate},string-blue] (0.75,1.49) -- (0.75,0);
        \end{scope}
        \fill[string-blue] (0.75,0) circle (0.03);
    \end{tikzpicture}\,,
\end{equation}
where the boundary $1$-strata, indicated with a darker colouring, inherit the orientation of the neighbouring $2$-strata in the interior. The $0$-stratum also inherits the orientation of the interior $1$-stratum ending on it. Moreover, a $0$-stratum, on a gluing boundary, is only allowed as the endpoint of exactly one $1$-stratum in the interior.

For free boundaries we remove the second restriction and allow for an arbitrary (finite) number of $1$-strata ending on a $0$-stratum:
\begin{equation}\label{eq:freeboundngbh}
    \begin{tikzpicture}[baseline=0.7cm]
        \fill[white!70!magenta] (0,0) rectangle (1.5,1.5);
        \begin{scope}[decoration={markings, mark=at position 0.55 with {\arrow{>}}}] 
            \draw[line cap=round,thick,postaction={decorate},string-violet] (0,1.49) -- (0,0.01);
            \end{scope}
    \end{tikzpicture}\, ,\quad
    \begin{tikzpicture}[baseline=0.7cm]
        \fill[white!70!magenta] (0,0) rectangle (1.5,1.5);
        \begin{scope}[decoration={markings, mark=at position 0.55 with {\arrow{>}}}] 
            \draw[line cap=round,thick,postaction={decorate},string-violet] (1.5,0.01) -- (1.5,0.75);
            \draw[line cap=round,thick,postaction={decorate},string-violet] (1.5,0.75) -- (1.5,1.49);
        \end{scope}
        \fill[black] (1.5,0.75) circle (0.06);
        \node at (1.5,0.75) [right,black] {\scriptsize $\pm$};
    \end{tikzpicture}\, ,\quad
    \begin{tikzpicture}[baseline=0.7cm]
        \fill[white!70!magenta] (0,0) rectangle (1.5,1.5);
        \begin{scope}[decoration={markings, mark=at position 0.55 with {\arrow{>}}}] 
            \draw[line cap=round,thick,postaction={decorate},string-violet] (0,0.75) -- (0,0.01);
            \draw[line cap=round,thick,postaction={decorate},string-violet] (0,1.49) -- (0,0.75);
            \draw[line cap=round,thick,postaction={decorate},string-blue] (1.49,1.25) -- (0,0.75);
            \draw[line cap=round,thick,postaction={decorate},string-blue] (0,0.75) -- (1.49,0.25);
        \end{scope}
        \fill[black] (0,0.75) circle (0.06);
        \node at (0,0.75) [left,black] {\scriptsize $+$};
    \end{tikzpicture}\,;
\end{equation}
It is useful to think of these neighbourhoods as a special case of the allowed $2$-dimensional neighbourhoods in the interior, where one $2$-stratum is "trivial" and the adjacent $1$-strata are part of the free boundary.\footnote{In principle we could also allow for an independent orientation on the free boundary $1$-strata, making this analogy even more apparent. However, in this case the the labelling, discussed below, becomes a bit more involved.}

Finally, near the corner point we allow for the following two possibilities:
\begin{equation}
    \begin{tikzpicture}[baseline=0.7cm]
    \fill[white!70!magenta] (0,0) rectangle (1.5,1.5);
        \begin{scope}[decoration={markings, mark=at position 0.55 with {\arrow{>}}}] 
            \draw[line cap=round,thick,postaction={decorate},string-violet] (0,1.49) -- (0,0.01);
            \draw[line cap=round,thick,magenta] (0.01,0) -- (1.49,0);
            \end{scope}
            \fill[string-violet] (0,0) circle (0.04);
    \end{tikzpicture}\, ,\quad
    \begin{tikzpicture}[baseline=0.7cm]
        \fill[white!70!magenta] (0,0) rectangle (1.5,1.5);
        \begin{scope}[decoration={markings, mark=at position 0.55 with {\arrow{>}}}] 
            \draw[line cap=round,thick,postaction={decorate},string-violet] (1.5,0.01) -- (1.5,1.49);
            \draw[line cap=round,thick,magenta] (0.01,0) -- (1.49,0);
            \end{scope}
            \fill[string-violet] (1.5,0) circle (0.04);
    \end{tikzpicture}\,, 
\end{equation}
where the $2$-stratum has the standard orientation of $\mathbbm{R}^2$. 

A \emph{morphism between (unlabelled) topological world sheets} is a continuous map between open-closed surfaces which sends $j$-strata to $j$-strata in a smooth and orientation preserving manner.

\begin{remark}
    This definition of unlabelled world sheet is closely related, but differs slightly from the one given in \autocite[Def.\,2.4]{FSY2023RCFTstring1}. They only consider stratifications with at least one $0$-stratum for every boundary component. This extra $0$-stratum is necessary for their construction to make the object of bulk fields, see \Cref{subsec:oplaxnattrafo}, unique \autocite[Sec.\,5.2]{yang2022thesis}.
\end{remark}

Next, let us discuss the labelling data. As in \Cref{sec:3ddeftft}, this will be encoded in what we will call $2$-\emph{dimensional defect data}
\begin{equation}
\DD^{2\mathrm{d}} = (D_2^{2\mathrm{d}},D_1^{2\mathrm{d}},D_0^{2\mathrm{d}};s^{2\mathrm{d}},t^{2\mathrm{d}},j^{2\mathrm{d}}).
\end{equation}
As before, the elements of the sets $D_i^{2\mathrm{d}}$, for $i\in \{0,1,2\}$, will be used to label the $i$-strata while the source, target, and junction map $s^{2\mathrm{d}},t^{2\mathrm{d}}\colon D_1^{2\mathrm{d}} \times \{\pm\}\to D_2^{2\mathrm{d}}$ and $j^{2\mathrm{d}}\colon D_0 \times \{\pm\} \to (\textrm{ordered lists of elements of } D_1^{2\mathrm{d}})$ encode the adjacency conditions, see \autocite[Sec.\,2.3]{DKR11defbicat} or \autocite[Ex.\,2.5]{CRS17deforbi}. In order to include the labels for the free boundaries, i.e.\ the boundary conditions, we assume that the set $D_2$ contains a special element $T$, which we think of as the ``trivial phase'', and label free boundaries with the elements in $D_1^{2\mathrm{d}}$ which have $T$ as their source. We will call $T$ the \emph{transparent} element in $D_2^{2\mathrm{d}}$.
\begin{remarks}
\begin{enumerate}
    \item Alternatively, we also could have introduced an extra set of boundary conditions $D_{1,\,\partial}^{2\mathrm{d}}$ and boundary point defects $D_{0,\,\partial}^{2\mathrm{d}}$ together with a target $t_\partial^{2\mathrm{d}}$ and junction map $j_\partial^{2\mathrm{d}}$. To see how this extra data is already encoded in the above, note that we can simply set $D_{1,\,\partial}^{2\mathrm{d}} := s^{-1}((T,+)) \subset D_1^{2\mathrm{d}}$. 
    Moreover, this process is reversible by taking the union of the corresponding sets for $1$- and $0$-dimensional objects and adding an extra element to $D_2^{2\mathrm{d}}$.
    \item Analogously to \Cref{rem:def3cat}, one can construct a $2$-category from the defect data $\DD^{2\mathrm{d}}$ with objects the elements in $D_2^{2\mathrm{d}}$, the $1$-morphisms (lists) of elements in $D_1^{2\mathrm{d}}$, and $2$-morphisms (lists) of elements in $D_0^{2\mathrm{d}}$, see \autocite{DKR11defbicat} for details. The transparent element $T$ plays the role of a distinguished object in this $2$-category similarly to a monoidal unit.
\end{enumerate}
\end{remarks}

\begin{definition}
    Let $\DD^{2\mathrm{d}}$ be a set of $2$-dimensional defect data. A $\DD^{2\mathrm{d}}$-labelled world sheet $\frakS$ consists of an unlabelled world sheet $\check{\frakS}$ together with the following assignments of labels to the strata of $\check{\frakS}$:
    \begin{itemize}
    \item to any $2$-stratum an element in $D_2^{2\mathrm{d}}$ which we will call its \emph{phase};
    \item to any $1$-stratum in the interior of $\check{\frakS}$ an element in $D_1^{2\mathrm{d}}$ compatible with its source and target map;
    \item to any gluing boundary $1$-stratum the element in $D_2^{2\mathrm{d}}$ labelling the neighbouring $2$-stratum in the interior;
    \item to any gluing $0$-stratum the element in $D_1^{2\mathrm{d}}$ which labels the interior $1$-stratum ending on it; 
    \item to any free boundary $1$-stratum an element in $D_1^{2\mathrm{d}}$ such that its source is the transparent element $T \in D_2^{2\mathrm{d}}$ and its target matches the neighbouring $2$-stratum in the interior;
    \item to any free boundary $0$-stratum an element in $D_0^{2\mathrm{d}}$ compatible with the labels of the surrounding strata as in \Cref{sec:3ddeftft};
    \item to any corner point the same element in $D_1^{2\mathrm{d}}$ as the unique free boundary $1$-stratum in its neighbourhood;
    \end{itemize}

A morphism of $\DD^{2\mathrm{d}}$-labelled world sheets $\frakS$ and $\frakS'$ is a morphism of unlabelled world sheets $\check{\frakS}$ and $\check{\frakS}'$ compatible with the labelling maps. 
\end{definition}
An example for an allowed labelling for the third local neighbourhood in \eqref{eq:freeboundngbh} is given by: 
\begin{equation}
    \begin{tikzpicture}[baseline=0.7cm]
        \fill[white!70!magenta] (0,0) rectangle (3,3);
        \begin{scope}[decoration={markings, mark=at position 0.55 with {\arrow{>}}}] 
            \draw[line cap=round,very thick,postaction={decorate},string-violet] (0,1.5) -- (0,0.01);
            \draw[line cap=round,very thick,postaction={decorate},string-violet] (0,2.99) -- (0,1.5);
            \draw[line cap=round,very thick,postaction={decorate},string-blue](2.99,2.5) -- (0,1.5);
            \draw[line cap=round,very thick,postaction={decorate},string-blue] (0,1.5) -- (2.99,0.5);
        \end{scope}
        \fill[black] (0,1.5) circle (0.06);
        \node at (0,1.5) [left,black] {$f$};
        \node at (0,0) [below,string-violet] { $M$};
        \node at (0,3) [above,string-violet] {$N$};
        \node at (1.5,0) [below,magenta] {$\alpha$};
        \node at (1.5,3) [above,magenta] {$\gamma$};
        \node at (3,1.5) [right,magenta] {$\beta$};
        \node at (3,0.5) [right,string-blue] {$X$};
        \node at (3,2.5) [right,string-blue] {$Y$};
    \end{tikzpicture}
\end{equation}
Here $\alpha,\beta,\gamma \in D_2^{2\mathrm{d}}$, $M,N,X,Y \in D_1^{2\mathrm{d}}$, and $f \in D_0^{2\mathrm{d}}$ are such that $j^{2\mathrm{d}}(f,+) = ((M,+),(X,+),(Y,-),(N,-))$ and $t(M,+)=s(X,+)= \alpha$, $t(X,+)=s(Y,-) = \beta$, $t(Y,-)=s(N,-) = \gamma$, and $t(N,-) = s(M,+)=T$. The linear order of the list $j^{2\mathrm{d}}(f,+) = ((N,-),(Y,-),(X,+),(M,+))$ is chosen as always starting from the free boundary $1$-stratum which goes out of the $0$-stratum, labelled by $M$ in this case, and then going counter-clockwise. This choice is always possible because every free boundary $0$-stratum will have exactly one incoming and one outgoing free boundary $1$-stratum.

By an analogous index shift as in the definition of $\bord_{3,2}^{\mathrm{def}}(\DD)$ we can also decorate compact stratified $1$-manifolds using the $2$d defect data $\DD^{2\mathrm{d}}$. Putting this together leads to the following bordism $2$-category:
\begin{definition}
    Let $\DD^{2\mathrm{d}}$ be a set of $2$-dimensional defect data as above. The $(2,1)$-category of $\DD^{2\mathrm{d}}$-labelled \emph{topological world sheets} $\mathfrak{WS}(\DD^{2\mathrm{d}})$ is the symmetric monoidal $2$-category consisting of:
    \begin{itemize}
    \item objects: $\DD^{2\mathrm{d}}$-decorated compact stratified $1$-manifolds;
    \item $1$-morphisms: for objects $\frakC$ and $\frakC'$, a $1$-morphism from $\frakC$ to $\frakC'$, denoted as $\frakC \to \frakC'$, is a topological world sheet $\frakS$, together with a parametrisation of its gluing boundary $\partial^{\mathrm{gl}}\frakS \cong -\frakC \sqcup \frakC'$ as defect $1$-manifolds with boundary;
    \item $2$-morphisms: for $1$-morphisms $\frakS \colon\frakC \to \frakC'$ and $\frakS' \colon\frakC \to \frakC'$ a $2$-morphism $[f] \colon \frakS \Rightarrow \frakS'$ is the isotopy class of an isomorphism $f \colon \frakS \to \frakS'$ of topological world sheets which is compatible with the boundary parametrisations;
    \item horizontal composition: for $1$-morphisms $\frakS_1 \colon\frakC_1 \to \frakC$ and $\frakS_2 \colon\frakC \to \frakC_2$ the horizontal composition $\frakS_2\diamond \frakS_1 \colon\frakC_1 \to \frakC_2$ is given by the topological world sheet
    \begin{equation}
    \frakS_2 \sqcup_{\frakC} \frakS_1
    \end{equation}
    obtained by gluing along $\frakC$ using the parametrisation of the gluing boundaries of $\frakS_1$ and $\frakS_2$, respectively.
    \item identity $1$-morphism: for an object $\frakC$ the identity $1$-morphism $\id_\frakC \colon \frakC \to \frakC$ is 
    \begin{equation}
    \frakC \times I
    \end{equation}
    with stratification and decoration induced from $\frakC$;
    \item vertical composition: 
    for $2$-morphisms $[f] \colon \frakS \Rightarrow \frakS'$ and $[g] \colon \frakS' \Rightarrow \frakS''$,     
    the vertical composition is given by the $2$-morphism 
    \begin{equation}
    [g\circ f];
    \end{equation}
    \item identity $2$-morphism: for a $1$-morphism $\frakS \colon\frakC \to \frakC'$ the identity $2$-morphism $\id_{\frakS} \colon \frakS \Rightarrow \frakS$ is  
    \begin{equation}
    [\id];
    \end{equation} 
    \item disjoint union as symmetric monoidal structure;
\end{itemize}
\end{definition}
An example of a $1$-morphism from an interval to the disjoint union of two defect circles is given by
\begin{equation}\label{eq:ws}
            \begin{tikzpicture}[baseline=0cm]
    \fill[white!80!magenta,opacity=0.7]
    {[rounded corners] (-0.75,-2) -- (-0.8,-0.8) -- (-1,-0.5) -- (-2,0.5) -- (-2.2,0.8) -- (-2.25,2)} --
    (-2.25,2) arc (180:0:0.75 and 0.5) --
    (-0.75,2) arc (180:360:0.75 and 1) --
    (0.75,2) arc (180:260:0.75 and 0.5) 
    {[rounded corners] ([xshift=1.5cm,yshift=2cm] 260:0.75 and 0.5) --([xshift=1.5cm,yshift=1.8cm] 260:0.75 and 0.5) -- ([yshift=-1.8cm] 110:0.75 and 1.5) -- ([yshift=-2cm] 110:0.75 and 1.5)} --
    ([yshift=-2cm] 110:0.75 and 1.5) arc (110:180:0.75 and 1.5);
    \fill[white!70!magenta,opacity=0.7]
    ([xshift=1.5cm,yshift=2cm] 70:0.75 and 0.5) arc (70:295.9:0.75 and 0.5);
    \fill[white!80!cyan,opacity=0.8]
    ([xshift=1.5cm,yshift=2cm] 70:0.75 and 0.5) arc (70:-64.1:0.75 and 0.5);
    \fill[white!70!magenta,opacity=0.7]
    (0.75,-2) arc (0:180:0.75 and 1.5) --
    (-0.75,-2) arc (180:0:0.75 and 0.5); 
    \fill[white!50!magenta,opacity=0.7]
    {[rounded corners] (1.3,-0.2) -- (0.8,-0.75) -- (0.75,-2)} --
    (0.75,-2) arc (0:50:0.75 and 1.5)  
    {[rounded corners] ([yshift=-2cm] 50:0.75 and 1.5) --
    ([xshift=0.15cm,yshift=-1.9cm] 50:0.75 and 1.5) -- (0.9,-0.3) -- (1.3,-0.2)} 
    ;
    \fill[white!60!magenta,opacity=0.7]
    {[rounded corners] (-0.75,-2) -- (-0.8,-0.8) -- (-1,-0.5) -- (-2,0.5) -- (-2.2,0.8) -- (-2.25,2)} --
    (-2.25,2) arc (180:360:0.75 and 0.5) --
    (-0.75,2) arc (180:360:0.75 and 1) --
    (0.75,2) arc (180:260:0.75 and 0.5) 
    {[rounded corners] ([xshift=1.5cm,yshift=2cm] 260:0.75 and 0.5) --([xshift=1.5cm,yshift=1.8cm] 260:0.75 and 0.5) -- ([yshift=-1.8cm] 110:0.75 and 1.5) -- ([yshift=-2cm] 110:0.75 and 1.5)} --
    ([yshift=-2cm] 110:0.75 and 1.5) arc (110:180:0.75 and 1.5)
    ;
    \fill[white!60!cyan,opacity=0.7]
    {[rounded corners] (1.4,-0.1) --  (1.4,0.1) -- ([xshift=1.6cm,yshift=0.7cm] 70:0.75 and 0.5) --([xshift=1.5cm,yshift=2cm] 295.9:0.75 and 0.5)} --([xshift=1.5cm,yshift=2cm] 295.9:0.75 and 0.5) arc (295.9:360:0.75 and 0.5) {[rounded corners] (2.25,2) --  (2.2,0.8) -- (2,0.5) -- (1.4,-0.1)}
    ;
    \fill[white!40!cyan,opacity=0.6]
    [rounded corners] ([xshift=1.5cm,yshift=2cm] 295.9:0.75 and 0.5) -- ([xshift=1.6cm,yshift=0.7cm] 70:0.75 and 0.5) --  (1.4,0.1) [sharp corners] --   (1.4,-0.1) -- (1.3,-0.2) [rounded corners] -- (0.9,-0.3) -- ([xshift=0.15cm,yshift=-1.9cm] 50:0.75 and 1.5)  [sharp corners] -- ([yshift=-2cm] 50:0.75 and 1.5) arc (50:110:0.75 and 1.5) [rounded corners] -- ([yshift=-1.8cm] 110:0.75 and 1.5) -- ([xshift=1.5cm,yshift=1.8cm] 260:0.75 and 0.5) [sharp corners] --([xshift=1.5cm,yshift=2cm] 260:0.75 and 0.5) arc (260:295.9:0.75 and 0.5);
    \begin{scope}[very thick,decoration={
                      markings,
                      mark=at position 0.45 with {\arrow{>}}}
                     ] 
    \draw[thick,string-red!70, line cap= round, dashed,postaction={decorate}] {[rounded corners] (1.4,-0.1) -- (1.4,0.1)  -- ([xshift=1.6cm,yshift=0.7cm] 70:0.75 and 0.5) -- ([xshift=1.5cm,yshift=2cm] 70:0.75 and 0.5)};
    \end{scope}
    \draw[line cap=round]{[rounded corners] (0.75,-2) -- (0.8,-0.8) -- (1,-0.5) -- (2,0.5) -- (2.2,0.8) -- (2.25,2)}
    {[rounded corners] (-0.75,-2) -- (-0.8,-0.8) -- (-1,-0.5) -- (-2,0.5) -- (-2.2,0.8) -- (-2.25,2)}
    (0.75,2) arc (0:-180:0.75 and 1);
    \begin{scope}[very thick,decoration={
                      markings,
                      mark=at position 0.5 with {\arrow{>}}}
                     ] 
    \draw[thick,string-green,postaction={decorate}] {[rounded corners]([xshift=1.5cm,yshift=2cm] 260:0.75 and 0.5) -- ([xshift=1.5cm,yshift=1.8cm] 260:0.75 and 0.5)-- ([yshift=-1.8cm] 110:0.75 and 1.5) -- ([yshift=-2cm] 110:0.75 and 1.5) };
    \end{scope}
    \draw[thick,string-green] {[rounded corners] ([yshift=-2cm] 110:0.75 and 1.5) -- ([yshift=-1.8cm] 110:0.75 and 1.5) -- ([xshift=1.5cm,yshift=1.8cm] 260:0.75 and 0.5) --([xshift=1.5cm,yshift=2cm] 260:0.75 and 0.5)};
    \draw[thick,string-red,line cap = round] {[rounded corners] ([yshift=-2cm] 50:0.75 and 1.5)  --([xshift=0.15cm,yshift=-1.9cm] 50:0.75 and 1.5) -- (0.9,-0.3) -- (1.3,-0.2)};
    \draw[thick, string-red, line cap= round] ([xshift=1.5cm,yshift=2cm] 70:0.75 and 0.5) --  ([xshift=1.5cm,yshift=2cm] 295.9:0.75 and 0.5);
    \draw[thick,white!30!magenta] (-1.5,2) ellipse (0.75 and 0.5)
    (0.75,-2) arc (0:180:0.75 and 0.5)
    ([xshift=1.5cm,yshift=2cm] 70:0.75 and 0.5) arc (70:260:0.75 and 0.5);
    \draw[thick,cyan]([xshift=1.5cm,yshift=2cm] 70:0.75 and 0.5) arc (70:-100:0.75 and 0.5);
    \begin{scope}[very thick,decoration={
                      markings,
                      mark=at position 0.502 with {\arrow{>}}}
                     ] 
    \draw[thick,line cap=round,string-violet!40!magenta,postaction={decorate}] (0.75,-2) arc (0:180:0.75 and 1.5);
    \draw[thick,line cap=round,string-blue!50!string-green,postaction={decorate}] ([yshift=-2cm] 110:0.75 and 1.5) arc (110:180:0.75 and 1.5);
    \end{scope}
    \begin{scope}[very thick,decoration={
                      markings,
                      mark=at position 0.7 with {\arrow{>}}}
                     ] 
    \draw[thick,line cap=round,string-violet,postaction={decorate}] ([yshift=-2cm] 0:0.75 and 1.5) arc (0:50:0.75 and 1.5);
    \end{scope}
    \fill[string-violet]
    (0.75,-2) circle (0.04);
   \fill[string-blue!50!string-green] (-0.75,-2) circle (0.04);
    \fill[black]([yshift=-2cm] 110:0.75 and 1.5) circle (0.04)
    ([yshift=-2cm] 50:0.75 and 1.5) circle (0.04);
    \fill[string-red]([xshift=1.5cm,yshift=2cm] 70:0.75 and 0.5) circle (0.03);
    \fill[string-green]([xshift=1.5cm,yshift=2cm] 260:0.75 and 0.5) circle (0.03);
    \draw[white!30!magenta,thick](3.2,-0.5) -- (3.2,0.5);
    \fill[]
    (3.2,-0.5) circle (0.04);
    \fill[string-violet] (3.2,0.5) circle (0.04);
    \draw[white!30!magenta,thick] (5.5,0) circle (0.5)
    (7,0.5) arc (90:270:0.5);
    \draw[cyan,thick] (7,0.5) arc (90:-90:0.5);
    \fill[string-green] (7,-0.5) circle (0.03);
    \fill[string-red] (7,0.5) circle (0.03);
    \draw[thick,->,line cap=round] (3.7,0) -- (4.7,0);
    \fill[black] (2.7,-0.07) circle (0.025)
    (2.7,0.07) circle (0.025);
\end{tikzpicture}\, .
\end{equation}
where we used different colours as an indication of different labels.

Forgetting the labels and stratification gives a symmetric monoidal $2$-functor
    \begin{equation}\label{eq:worldsheetmorphism3d}
       U \colon \WS(\DD^{2\mathrm{d}}) \to \bord_{2+\epsilon,2,1}^{\oc}
    \end{equation}
which sends a defect manifold to its underlying manifold. In the following we will denote the image under this functor simply by using the corresponding capital Greek letter, i.e.\ $U(\frakC) \equiv \Gamma$ and $U(\frakS)\equiv \Sigma$. Note that for a world sheet $\frakS \colon \frakC \to \frakC'$ its automorphism group $\mathrm{Aut}(\frakS) = \mathrm{End}_{\WS(\frakC,\frakC')}(\frakS)$ is by definition the subgroup of the mapping class group of $\Sigma$ which fixes the defects. This coincides with the notion of the mapping class group of a world sheet from \autocite[Def.\,2.17]{FSY2023RCFTstring1} as a morphism between stratified manifolds sending strata to strata.

\section{Algebraic structure underlying full 2d CFT}\label{sec:algcft}
In this section we will discuss the algebraic structure underlying a full two-dimensional CFT based on some fixed chiral data in the form of a not necessarily semisimple modular tensor category $\calc$ together with it's chiral modular functor. 

We will first review the $2$-categorical definitions of chiral and full modular functors of \autocite[Sec.\,2.1]{FSY2023RCFTstring1}. Afterwards we will give a brief overview of the main result of \autocite{HR2024modfunc} on how the $3$d TFTs of \autocite{DGGPR19} constructed from $\calc$ can be used to construct such modular functors. Finally, we will define a full CFT based on a full modular functor in $2$-categorical language and discuss why this definition encodes the data expected of a full CFT. For the relevant $2$-categorical notions we refer to \autocite{JY20twodcategories} and to \autocite[Ch.\,2]{SchommerPries2011thesis} for details on symmetric monoidal $2$-categories.
\subsection{Modular functors}\label{subsec:modfunc}
Both chiral and full modular functor are a form of categorified $2$d TFT with target the symmetric monoidal $2$-category $\cat{P}\mathrm{rof}_{\mathbbm{k}}^{\coend \mathrm{ex}}$ of \emph{linear, left exact profunctors} which which consists of
\begin{itemize}
    \item objects: finite linear categories;
    \item $1$-morphisms: for objects $\cat{A}$ and $\cat{B}$ a $1$-morphism from $\cat{A}$ to $\cat{B}$ is a left exact linear profunctor from $\cat{A} \slashedrightarrow \cat{B}$, i.e.\ a left exact linear functor $\cat{A}^{\mathrm{op}} \boxtimes \cat{B} \to \vect$;
    \item $2$-morphisms: natural transformations;
    \item horizontal composition: for $1$-morphisms $F \colon \cat{A} \slashedrightarrow \cat{B}$ and $G\colon \cat{B} \slashedrightarrow \calc$ the horizontal composition is the left exact coend
    \begin{align}
        G \diamond F (-,\sim) := \oint^{B \in \cat{B}}\! G(B,\sim) \otimes_{\mathbbm{k}} F(-,B)  \colon \cat{A} \slashedrightarrow \calc
    \end{align}
    where $-$ stands for an argument from the source category (in this case $\calA$) and $\sim$ for an argument from the target category (in this case $\calc$);
    \item identity $1$-morphism: for an object $\calA$ the identity $1$-morphism $\id_\calA \colon \calA \slashedrightarrow \calA$ is given by unique functor induced from the universal property of the Deligne product applied to the Hom functor $\Hom_\calA(-,\sim)$;
    \item vertical composition: standard vertical composition of natural transformations; 
    \item identity $2$-morphism: identity natural transformations;
    \item Deligne's tensor product of finite linear categories as the symmetric monoidal structure;
\end{itemize} 
For more details on left exact profunctors and left exact coends we refer to \autocite[Sec. 2.2]{hofer25thesis} and references therein. 
In the following we will often use $-$ for both the source and target argument as it will be clear from the context which one is which, e.g.\ we will write $\Hom_\calA(-,-)$ instead of $\Hom_\calA(-,\sim)$. 

Before we continue let us briefly explain how $\Prof$ is related to the more familiar target $2$-category $\calL\mathrm{ex}_\kk$ of finite linear categories, left exact functors, and natural transformations again with the Deligne tensor product as symmetric monoidal structure. There are two identity-on-objects $2$-functors $h_{-},h^{-} \colon \calL\mathrm{ex}_\kk \to \Prof$: The first one sends a functor $F \colon \calA \to \calB$ to the \emph{representable} profunctor $h_F \colon \calB \slashedrightarrow \calA$ induced by $\Hom_\calB(-,F(-))\colon \calB^\op \times \calA \to \vect$ on the Deligne product $\calB^\op\boxtimes \calA$. The second one sends a functor $F \colon \calA \to \calB$ to the \emph{corepresentable} profunctor $h^F \colon \calA \slashedrightarrow \calB$ induced by $\Hom_\calB(F(-),-) \colon \calA^\op \times \calB \to \vect$ on the Deligne product $\calA^\op\boxtimes \calB$. To be more precise $h_{-}$ is a $2$-functor contravariant on $1$-morphisms while $h^{-}$ is contravariant on $2$-morphisms. Compatibility of $h_{-}$ and $h^{-}$ with horizontal composition, i.e.\ functoriality, can be proved by directly using standard results on coends, see e.g.\ \autocite[Sec.\,5.2]{Loregian2021coendcalc}. 

For a fixed left exact functor $F \colon \calA \to \calB$ the two profunctors $h_F$ and $h^F$ are closely related: 
\begin{lemma}\label{lem:profadjunc}
    Let $F \colon \calA \to \calB$ be a left exact functor. The profunctor $h_F$ is left adjoint to the profunctor $h^F$ as a $1$-morphism in the $2$-category $\Prof$.
\end{lemma}
\begin{proof}
    This is a variation of the proof given in \autocite[Rem.\,5.2.1]{Loregian2021coendcalc}. For later use we will repeat some of the argument and indicate where one needs to be more careful.

    To obtain the unit $\eta \colon \Hom_\calA \Rightarrow h^{F} \diamond h_F$ first note that $h^{F} \diamond h_F$ is induced from $\oint^{B \in \calB} \Hom_\calB(B,F(-)) \otimes_\kk \Hom_\calB(F(-),B) \cong \Hom_\calB(F(-),F(-))$ where we used the Yoneda lemma. We define $\eta$ to be induced from the natural transformation $\Hom_{\calA}(-,-) \Rightarrow \Hom_\calB(F(-),F(-))$ given by the action of $F$ on morphisms. 

    The counit $\epsilon \colon h_F \diamond h^F \rightarrow \Hom_\calB$ is obtained as follows: First note that the natural transformation $\Hom_\calB(F(A),-) \otimes_\kk \Hom_\calB(-,F(A)) \Rightarrow \Hom_\calB(-,-)$ coming from the composition map is dinatural in $A$ and thus factorises over the coend $\oint^{A \in \calA} \Hom_\calB(F(A),-)  \otimes_\kk \Hom_\calB(-,F(A))$. We can now define $\epsilon$ to be the natural transformation $h_F \diamond h^F \rightarrow \Hom_\calB$ induced from this natural transformation.

    To prove the triangle identities it suffices to prove them for the natural transformations from which $\eta$ and $\epsilon$ are induced instead. This can now be done completely analogously as in \autocite[Rem.\,5.2.1]{Loregian2021coendcalc} by working component wise.
\end{proof}

Finally, we want to mention that under our finiteness assumptions the $2$-functors $h_{-}$ and $h^{-}$ are actually $2$-equivalences by the Eilenberg-Watts theorem, see \autocite[Lem.\,3.2]{Shimizu2016unimod}. The reason we use $\Prof$ as a target rather than $\calL\mathrm{ex}_\kk$ is that the former appears naturally when we define the modular functor on $1$- and $2$-morphisms. Moreover, the coends appearing in the definition of horizontal composition can physically be interpreted as a way of \emph{summing over intermediate states}, see \autocite[Sec.\,1.2]{FS16coendsCFT} for more details on this interpretation.

A \emph{chiral} modular functor has as its source a version of the $(2,1)$-category closed bordism category $\bord_{2+\epsilon,2,1}^{\chi}$ of closed $1$-manifolds, $2$-dimensional bordisms, and isotopy classes of diffeomorphisms, as well as extra data related to a gluing anomaly indicated by the superscript $\chi$.\footnote{This is the same anomaly data as the one needed for the $3$d TFT $\widehat{\rmV}_\calc$ of \autocite{DGGPR19}.} We will only tangentially work with $\bord_{2+\epsilon,2,1}^{\chi}$ here and refer to \autocite[Sec.\,4]{HR2024modfunc} for the full definition. A natural source of chiral modular functors are $3$d TFTs with embedded ribbon graphs obtained from modular tensor categories $\calc$, see \Cref{sec:modtens} for details on modular tensor categories and our conventions:
\begin{proposition}{\autocite[Thm.\,4.10]{HR2024modfunc}}
    For any modular tensor category $\calc$, the $3$d TFT with embedded ribbon graphs
\begin{equation}
        \widehat{\rmV}_\calc \colon \widehat{\bord}_{3,2}^{\chi}(\calc) \to \vect
\end{equation}
of \autocite{DGGPR19} induces a chiral modular functor 
\begin{align}
    \mathrm{Bl}^\chi_\calc \colon \bord^{\chi}_{2+\epsilon,2,1} \to  \cat{P}\mathrm{rof}_{\mathbbm{k}}^{\coend \mathrm{ex}}.
\end{align}
\end{proposition}
Let us briefly recall how $\mathrm{Bl}_\calc^{\chi}$ is constructed, for the detailed construction and proofs see \autocite[Sec.\,4]{HR2024modfunc}. On objects $\Gamma \in \bord^{\chi}_{2+\epsilon,2,1}$ we have $\Blc\left(\Gamma \right):= \calc^{\boxtimes \pi_0(\Gamma)}$ where we index the factors directly by the connected components $\pi_0(\Gamma)$ of $\Gamma$.
A $1$-morphism $\Sigma \colon \Gamma \to \Gamma'$ in $\bord^{\chi}_{2+\epsilon,2,1}$ is first turned into a family of objects $\Sigma^{(\underline{X};\underline{Y})}$ in $\widehat{\bord}_{3,2}^{\chi}(\calc)$ with $\underline{X}\in \calc^{\times \pi_0(\Gamma)}$ and $\underline{Y}\in \calc^{\times \pi_0(\Gamma')}$ by gluing a disk with a marked point into every boundary component of $\Sigma$ with label the element in the list $(\underline{X};\underline{Y})$ matching the corresponding boundary component. Applying the TFT $\widehat{\rmV}_\calc$ gives rise to a functor 
\begin{equation}\label{eq:bllex-def}
    \begin{aligned}
     \mathrm{bl}^{\chi}_{\calc} (\Sigma)
     \colon (\calc^{\mathrm{op}})^{\times\pi_0(\Gamma)} \times \calc^{\times \pi_0(\Gamma')} &\to \vect \\    
     (\underline{X};\underline{Y}) &\mapsto \widehat{\rmV}_\calc(\Sigma^{(\underline{X};\underline{Y})})
    \end{aligned}
\end{equation}
The action of $\mathrm{bl}^{\chi}_{\calc} (\Sigma)$ on morphisms is obtained by considering cylinders over $\Sigma^{(\underline{X};\underline{Y})}$ with embedded ribbon graphs obtained by the morphisms. The functor $\mathrm{bl}^{\chi}_{\calc} (\Sigma)$ can be shown to be left exact in each argument and can thus be extended to the Deligne tensor product 
\begin{equation}
    \mathrm{Bl}_\calc^{\chi}(\Sigma) \colon (\calc^{\mathrm{op}})^{\boxtimes\pi_0(\Gamma)} \boxtimes \calc^{\boxtimes \pi_0(\Gamma')} \to \vect.
\end{equation}
Finally the action of $\mathrm{Bl}_\calc^{\chi}$ on a $2$-morphism $[f] \colon \Sigma \Rightarrow \Sigma'$ is obtained by considering mapping cylinders $C_f^{(\underline{X};\underline{Y})} \colon \Sigma^{(\underline{X};\underline{Y})} \to \Sigma'^{(\underline{X};\underline{Y})}$.

A \emph{full} modular functor has as its source the $(2,1)$-category of open-closed bordisms $\bord^{\oc}_{2+\epsilon,2,1}$ discussed in \Cref{subsec:topworld}. Note that we can pull back any full modular functor along the forgetful functor $\WS(\DD) \to \bord^{\oc}_{2+\epsilon,2,1}$. Next recall the \emph{orientation double} $2$-functor
\begin{align*}
    \widehat{(-)} \colon \bord^{\oc}_{2+\epsilon,2,1} \to \bord^{\chi}_{2+\epsilon,2,1}.
\end{align*}
discussed in \autocite[Sec.\,5.3]{HR2024modfunc} which sends an object $\Gamma \in \bord^{\oc}_{2+\epsilon,2,1}$ to
\begin{align}
    \widehat{\Gamma} = \Gamma \sqcup -\Gamma / \sim \quad \text{with} \quad (p,+) \sim (p,-) \quad \text{for} \quad p \in \partial \Gamma,
\end{align} 
and a $1$-morphism $\Sigma \colon \Gamma \to \Gamma'$ to 
\begin{align}
    \widehat{\Sigma} = \Sigma \sqcup -\Sigma / \sim \quad \text{with} \quad (p,+) \sim (p,-) \quad \text{for} \quad p \in \partial^f \Sigma.
\end{align}
We can pull back any chiral modular functor $\mathrm{Bl}^\chi \colon \bord^{\chi}_{2+\epsilon,2,1} \to  \cat{P}\mathrm{rof}_{\mathbbm{k}}^{\coend \mathrm{ex}}$ along $\widehat{(-)}$ to obtain a full modular functor. In the following we will be interested in the full modular functor $\mathrm{Bl}_\calc^{\chi} \circ \widehat{(-)}$ constructed from a modular tensor category $\calc$. To this end, it will be beneficial to describe the action of $\mathrm{Bl}_\calc^{\chi} \circ \widehat{(-)}$ on a $1$-morphism $\Sigma \colon \Gamma \to \Gamma'$ in $\bord^{\oc}_{2+\epsilon,2,1}$ a bit more explicitly. According to the description of $\mathrm{Bl}_\calc^{\chi}$ above we first glue disks with $\calc$-labelled marked points into the boundary components of $\widehat{\Sigma}$. However, for some of the boundary components it is more natural to consider them as pairs $S^1 \sqcup - S^1 = \widehat{S^1}$ since they arise as images of a single $S^1$-boundary components of $\Sigma$ under $\widehat{(-)}$. Moreover, by \autocite[Prop.\,5.5]{HR2024modfunc} we know that $\mathrm{Bl}_\calc^{\chi} (\widehat{S^1}) \simeq \calc \boxtimes \overline{\calc}$, where $\overline{\calc}$ denotes the \emph{reversed} ribbon category of $\calc$. It is therefore more natural to think of a $\widehat{S^1}$-boundary as a single boundary component labelled with an object in $\calc \times \overline{\calc}$ instead.\footnote{We can use $\calc \times \overline{\calc}$ instead of $\calc \boxtimes \overline{\calc}$ since in the full construction an extension from the Cartesian to the Deligne product is performed anyway \autocite[Def.\,4.4]{HR2024modfunc}.}

\subsection{Correlators and field content}\label{subsec:oplaxnattrafo}
We can now define the main object of interest of this article. 
\begin{definition}\label{def:fullcft} 
    Let $\mathrm{Bl}^\chi\colon\bord^{\chi}_{2+\epsilon,2,1} \to \Prof$ be a chiral modular functor and let $\DD^{2\mathrm{d}}$ be a set of $2$-dimensional defect data. Denote with $\Bl \colon \WS(\DD^{2\mathrm{d}}) \to \Prof$ the composition 
\begin{equation}
\WS(\DD^{2\mathrm{d}}) \stackrel{U}{\to} \bord^\oc_{2+\epsilon,2,1} \stackrel{\mathrm{\widehat{(-)}}}{\to} \bord^{\chi}_{2+\epsilon,2,1} \stackrel{\Bl^\chi}{\to} \Prof.
\end{equation}
    A \emph{full conformal field theory}, with chiral data governed by $\mathrm{Bl}^\chi$ and boundary conditions and topological defects encoded in $\DD^{2\mathrm{d}}$, is a braided monoidal oplax natural transformation
\[\begin{tikzcd}[ampersand replacement=\&]
	\WS(\DD^{2\mathrm{d}})\&\& {\Prof}
	\arrow[""{name=0, anchor=center, inner sep=0}, "{\Delta_{\kk}}", curve={height=-24pt}, from=1-1, to=1-3]
	\arrow[""{name=1, anchor=center, inner sep=0}, "{\mathrm{Bl}}"', curve={height=24pt}, from=1-1, to=1-3]
	\arrow["\Cor", shorten <=6pt, shorten >=6pt, Rightarrow, from=0, to=1]
\end{tikzcd}\]
where $\Delta_\kk \colon \WS(\DD^{2\mathrm{d}})\to \Prof$ is the constant symmetric monoidal $2$-functor sending every object to $\vect$.
\end{definition}
The rest of this section will be devoted to unwrapping this definition. An oplax natural transformation $\Cor$ as above consists of the following data \autocite[Def.\,4.3.1]{JY20twodcategories}:
\begin{itemize}
    \item $1$-morphism components: A left exact profunctor $\mathrm{Cor}_{\frakC} \colon \Delta_{\kk}(\frakC) \slashedrightarrow \mathrm{Bl}(\frakC)$ ;
    \item $2$-morphism components: A natural transformation
\[\begin{tikzcd}
	{\Delta_\kk(\frakC)} && {\Delta_\kk(\frakC')} \\
	\\
	{\mathrm{Bl}(\frakC)} && {\mathrm{Bl}(\frakC')}
	\arrow["{\Delta_\kk(\frakS)}", from=1-1, to=1-3]
	\arrow["{\Cor_\frakC}"', from=1-1, to=3-1]
	\arrow["{\Cor_\frakS}"{description}, Rightarrow, from=1-3, to=3-1]
	\arrow["{\Cor_{\frakC'}}", from=1-3, to=3-3]
	\arrow["{\mathrm{Bl}(\frakS)}"', from=3-1, to=3-3]
\end{tikzcd}\]
for every world sheet $\frakS \colon\frakC \to \frakC'$ in $\WS(\DD^{2\mathrm{d}})$;
\end{itemize}
A braided monoidal oplax natural transformation further includes \autocite[Def.\,2.7]{SchommerPries2011thesis}:
\begin{itemize}
    \item A modification $\Pi$ with components for any $\frakC,\frakC' \in \WS(\DD^{2\mathrm{d}})$ given by natural isomorphisms
\[\begin{tikzcd}
	& {\mathrm{Bl}(\frakC)\boxtimes\Delta_\kk(\frakC')} \\
	\\
	{\Delta_\kk(\frakC)\boxtimes\Delta_\kk(\frakC')} && {\mathrm{Bl}(\frakC)\boxtimes\mathrm{Bl}(\frakC')} \\
	\\
	{\Delta_\kk(\frakC\sqcup \frakC')} && {\mathrm{Bl}(\frakC\sqcup\frakC')}
	\arrow["{\id_{\mathrm{Bl}(\frakC)} \boxtimes \mathrm{Cor}_{\frakC'}}"{description}, from=1-2, to=3-3]
	\arrow["{ \mathrm{Cor}_{\frakC}\boxtimes\id_{\Delta_\kk(\frakC')}}"{description}, from=3-1, to=1-2]
	\arrow["\cong"', from=3-1, to=5-1]
	\arrow["{\Pi_{\frakC,\frakC'}}"{description}, Rightarrow, from=3-3, to=5-1]
	\arrow["\cong", from=3-3, to=5-3]
	\arrow["{\Cor_{\frakC\sqcup \frakC'}}"', from=5-1, to=5-3]
\end{tikzcd}\]
    where the unlabelled isomorphisms are part of the symmetric monoidal structure of $\Delta_\kk$ and $\mathrm{Bl}$;
    \item A natural isomorphism 
\[\begin{tikzcd}
	\vect & {\Delta_\kk(\varnothing)} \\
	\vect & {\mathrm{Bl}(\varnothing)}
	\arrow["\cong", from=1-1, to=1-2]
	\arrow[equals, from=1-1, to=2-1]
	\arrow["\Upsilon"{description}, Rightarrow, from=1-2, to=2-1]
	\arrow["{\Cor_\varnothing}", from=1-2, to=2-2]
	\arrow["\cong"', from=2-1, to=2-2]
\end{tikzcd}\]
    where the unlabelled isomorphisms are part of the symmetric monoidal structure of $\Delta_\kk$ and $\mathrm{Bl}$;
\end{itemize}
This data needs to satisfy the following axioms:
\begin{enumerate}[(i)]
    \item Naturality of $2$-morphism components: For any $2$-morphism $[f] \colon \frakS \Rightarrow \frakS'$ in $\WS(\DD^{2\mathrm{d}})$ the following diagram commutes:
\[\begin{tikzcd}[ampersand replacement=\&]
	{\mathrm{Cor}_{\frakC'}\diamond {\Delta_{\mathbb{k}}}(\frakS)} \& {\mathrm{Bl}({\frakS})\diamond\mathrm{Cor}_{\frakC}} \\
	{\mathrm{Cor}_{\frakC'}\diamond {\Delta_{\mathbb{k}}(\frakS})} \& {\mathrm{Bl}({\frakS})\diamond\mathrm{Cor}_{\frakC}}
	\arrow["{\mathrm{Cor}_{\frakS}}"', from=2-1, to=2-2]
	\arrow["{\mathrm{Cor}_{\frakS}}", from=1-1, to=1-2]
	\arrow["{\mathrm{Cor}_{\frakC'}\diamond \Delta_{\mathbb{k}}([f])}"', from=1-1, to=2-1]
	\arrow["{\mathrm{Bl}([f])\diamond\mathrm{Cor}_{\frakC}}", from=1-2, to=2-2]
\end{tikzcd}\]
\item Oplax naturality: For any pair of composable $1$-morphisms $\frakS_1 \colon \frakC_1 \rightarrow \frakC$ and $\frakS_2\colon\frakC \rightarrow\frakC_2$ in $\WS(\DD^{2\mathrm{d}})$ the following diagram commutes:

\hspace*{-1.5cm}
\begin{tikzcd}[ampersand replacement=\&]
	\& {(\mathrm{Bl}(\frakS_2)\diamond \mathrm{Cor}_{\frakC})\diamond\Delta_{\mathbb{k}}(\frakS_1)} \\
	{(\mathrm{Cor}_{\frakC_2}\diamond \Delta_{\mathbb{k}}(\frakS_2))\diamond\Delta_{\mathbb{k}}(\frakS_1)} \&\& {\mathrm{Bl}(\frakS_2)\diamond (\mathrm{Cor}_{\frakC}\diamond\Delta_{\mathbb{k}}(\frakS_1))} \\
	{\mathrm{Cor}_{\frakC_2}\diamond (\Delta_{\mathbb{k}}(\frakS_2)\diamond\Delta_{\mathbb{k}}(\frakS_1))} \&\& {\mathrm{Bl}(\frakS_2)\diamond (\mathrm{Bl}(\frakS_1)\diamond\mathrm{Cor}_{\frakC_1})} \\
	{\mathrm{Cor}_{\frakC_2}\diamond (\Delta_{\mathbb{k}}(\frakS_2\sqcup_{\frakC}\frakS_1))} \&\& {(\mathrm{Bl}(\frakS_2)\diamond \mathrm{Bl}(\frakS_1))\diamond\mathrm{Cor}_{\frakC_1}} \\
	\& {(\mathrm{Bl}(\frakS_2\sqcup_{\frakC}\frakS_1))\diamond\mathrm{Cor}_{\frakC_1}}
	\arrow["\cong", from=1-2, to=2-3]
	\arrow["{\mathrm{Cor}_{\frakS_2}\diamond \mathrm{id}}", from=2-1, to=1-2]
	\arrow["\cong"', from=2-1, to=3-1]
	\arrow["{\mathrm{id} \diamond \mathrm{Cor}_{\frakS_1}}", from=2-3, to=3-3]
	\arrow["\cong"', from=3-1, to=4-1]
	\arrow["\cong", from=3-3, to=4-3]
	\arrow["{\mathrm{Cor}_{\frakS_2\sqcup_{\frakC}\frakS_1}}"', from=4-1, to=5-2]
	\arrow["\cong", from=4-3, to=5-2]
\end{tikzcd}

where the isomorphisms correspond to either the associators of horizontal composition in $\Prof$ or the $2$-functor data of $\Delta_\kk$ and $\mathrm{Bl}$, respectively. 
\item
Oplax unitality:
\[\begin{tikzcd}[ampersand replacement=\&]
	{\mathrm{Cor}_{\frakC}\diamond \mathrm{id}_{\Delta_{\mathbb{k}}(\frakC)}} \& {\mathrm{Cor}_{\frakC}} \& {\mathrm{id}_{\mathrm{Bl}(\frakC)}\diamond\mathrm{Cor}_{\frakC}} \\
	{\mathrm{Cor}_{\frakC}\diamond {\Delta_{\mathbb{k}}(\mathrm{id}_{\frakC}})} \&\& {\mathrm{Bl}({\mathrm{id}_{\frakC}})\diamond\mathrm{Cor}_{\frakC}}
	\arrow["\Cor_{\id_\frakC}"', from=2-1, to=2-3]
	\arrow["\cong"', from=1-1, to=2-1]
	\arrow["\cong", from=1-1, to=1-2]
	\arrow["\cong", from=1-2, to=1-3]
	\arrow["\cong", from=1-3, to=2-3]
\end{tikzcd}\]
where the isomorphisms correspond to either the unitors of horizontal composition in $\Prof$ or the $2$-functor data of $\Delta_\kk$ and $\mathrm{Bl}$, respectively. 
\end{enumerate}
As well as four other axioms involving $\Pi$ and $\Upsilon$ which we will not discuss further, see \autocite[Def.\,2.7]{SchommerPries2011thesis}.

Let us take a moment to see that this data actually corresponds to something which deserves to be called a full conformal field theory. As mentioned above we will denote the manifolds underlying a defect $1$-manifold and a topological world sheet with the corresponding capital Greek letter. First, note that $\Cor_\frakC$ is given by a (left exact) functor $\mathrm{Bl}(\frakC) \to \vect$ because $\Delta_\kk(\frakC) = \vect$ for any $\frakC \in \WS(\DD^{2\mathrm{d}})$ and $\vect^\op \simeq \vect$. Thus we get an essentially unique object $\mathbbm{F}_\frakC \in \mathrm{Bl}(\frakC) = \mathrm{Bl}(\Gamma)$ such that $\mathrm{Cor}_{\frakC}(-) \cong \Hom_{\mathrm{Bl}(\Gamma)}(\mathbbm{F}_\frakC,-)$ by left exactness of $\Cor_\frakC$. Moreover, these objects factorise as $\mathbbm{F}_{\frakC \sqcup \frakC'} \cong \mathbbm{F}_{\frakC} \boxtimes \mathbbm{F}_{\frakC'}$ using the monoidality provided by $\Pi_{\frakC,\frakC'}$. Thus we can restrict our attention to $\frakC$ being a defect interval or circle. The category $\mathrm{Bl}(I) = \mathrm{Bl}^\chi(S^1)$ should correspond to the representation category of the chiral symmetry algebra while $\mathrm{Bl}(S^1) \simeq \mathrm{Bl}^\chi(S^1)\boxtimes \overline{\mathrm{Bl}^\chi(S^1)}$ corresponds to its double. This means that the objects $\mathbbm{F}_\frakC$ are automatically equipped with an action of the correct symmetry algebra. This observation suggests to interpret the $1$-morphism component of $\Cor$ as the \emph{field content} of the full CFT. 

Now by combining the isomorphism $\mathrm{Cor}_{\frakC}(-) \cong \Hom_{\mathrm{Bl}(\Gamma)}(\mathbbm{F}_\frakC,-)$ with the Yoneda lemma twice we get the following isomorphism of vector spaces
    \begin{equation}\label{eq:cor-block}
        \mathrm{Nat}(\Cor_{\frakC'}\diamond\Delta_{\kk}(\frakS),\mathrm{Bl}(\frakS)\diamond \Cor_\frakC) \cong \mathrm{Bl}(\frakS)(\mathbbm{F}_\frakC;\mathbbm{F}_{\frakC'})
    \end{equation}
for every world sheet $\frakS \colon \frakC \to \frakC'$. Thus the $2$-morphism component $\Cor_\frakS$ is automatically a vector in the space of conformal blocks of $\frakS$ with field insertions given by the field content. This suggests to think of the $2$-morphism components of $\Cor$ as the actual correlators of the theory.

Before we explain how the axioms $\Cor$ needs to satisfy should be interpreted, it will be beneficial to consider the following chain of abstract isomorphisms
    \begin{equation}\label{eq:cor-adj}
        \begin{aligned}
            \mathrm{Nat}(\Cor_{\frakC'}\diamond\Delta_{\kk}(\frakS),\mathrm{Bl}(\frakS)\diamond \Cor_\frakC) &\cong \mathrm{Nat}(\Cor_{\frakC'},\mathrm{Bl}(\frakS)\diamond \Cor_\frakC) \\
            &\cong \mathrm{Nat}(\Cor_{\frakC'}\diamond \Cor_\frakC^{\dagger},\mathrm{Bl}(\frakS))\\
            &\cong \mathrm{Nat}(\Cor_{\frakC'}\otimes_\kk \Cor_\frakC^{\dagger},\mathrm{Bl}(\frakS))
        \end{aligned}
    \end{equation}
where $\Cor_\frakC^\dagger(-) \cong \Hom_{\mathrm{Bl(\Gamma)}}(-,\mathbbm{F}_\frakC)$ is the (right) adjoint of $\Cor_\frakC$ in $\Prof$ from \Cref{lem:profadjunc}.
In the first step we used that $\Delta_\kk$ acts trivially, in the second one the adjunction, and in the final one that horizontal composition over $\vect$ is just the tensor product. From now on will suppress the isomorphisms \eqref{eq:cor-block} and \eqref{eq:cor-adj} from our notation and will always denote with $\Cor_\frakS$ the corresponding element in any of the three vector spaces as it will be clear from the context. 

Under the isomorphism \eqref{eq:cor-adj} the axioms $\Cor$ needs to satisfy become:
\begin{enumerate}[(i)]
    \item Naturality of $2$-morphism components reduces to commutativity of 
    \begin{equation}\label{diag:2natural}
        \begin{tikzcd}[ampersand replacement=\&]
    	\&\& {\mathrm{Bl}({\frakS})} \\
    	{\mathrm{Cor}_{\frakC'}\otimes_\kk\mathrm{Cor}_{\frakC}^\dagger} \\
    	\&\& {\mathrm{Bl}({\frakS}')}
    	\arrow["{\mathrm{Bl}([f])}", from=1-3, to=3-3]
    	\arrow["{\mathrm{Cor}_{\frakS}}", from=2-1, to=1-3]
    	\arrow["{\mathrm{Cor}_{\frakS'}}"', from=2-1, to=3-3]
    \end{tikzcd}.
    \end{equation}
    This is the \emph{covariance of correlators} and can be seen as a more general version of mapping class group action invariance. 
    \item 
    For simplicity we will only consider the oplax naturality axiom in a strictified version of $\Prof$, i.e.\ with trivial associators and unitors. In this setting oplax naturality becomes commutativity of
    \begin{align}\label{diag:horinatural}
    \begin{tikzcd}[ampersand replacement=\&]
	{\mathrm{Cor}_{\frakC_2}\otimes_\kk \mathrm{Cor}_{\frakC}^\dagger\diamond \mathrm{Cor}_{\frakC_1}\otimes_\kk \mathrm{Cor}_{\frakC}^\dagger} \&\& {\mathrm{Bl}(\frakS_2)\diamond \mathrm{Bl}(\frakS_1)} \\
	{\mathrm{Cor}_{\frakC_2}\otimes_\kk\mathrm{Cor}_{\frakC_1}^\dagger} \&\& {\mathrm{Bl}(\frakS_2\sqcup_{\frakC}\frakS_1)}
	\arrow["{\mathrm{Cor}_{\frakS_2}\diamond\mathrm{Cor}_{\frakS_1}}", from=1-1, to=1-3]
	\arrow["\cong", from=1-3, to=2-3]
	\arrow["{\id\otimes_\kk\eta_{\Cor_{\frakC}}\otimes_\kk\id}", from=2-1, to=1-1]
	\arrow["{\mathrm{Cor}_{\frakS_2\sqcup_{\frakC}\frakS_1}}"', from=2-1, to=2-3]
    \end{tikzcd}
    \end{align}
    where $\eta_{\Cor_{\frakC}} \colon \id_{\vect} \Rightarrow \Cor_{\frakC}^\dagger \diamond \Cor_{\frakC}$ denotes the unit of the adjunction $\Cor_{\frakC} \dashv \Cor_{\frakC}^\dagger$.
    This encodes the behaviour of the $2$-morphism components under gluing of world sheets and should be interpreted as the \emph{factorisation of correlators}. 
    \item Under the same strictification assumption as above, the oplax unitality axioms corresponds to
    \begin{equation}\label{diag:unitality}
        \Cor_{\id_\frakC} = \epsilon_{\Cor_\frakC}
    \end{equation}
    with $\epsilon_{\Cor_\frakC} \colon \Cor_\frakC \otimes_\kk \Cor_\frakC^\dagger \Rightarrow \mathrm{Bl}(\id_\frakC)$ the counit of the adjunction $\Cor_{\frakC} \dashv \Cor_{\frakC}^\dagger$. This can be interpreted as a non-degeneracy axiom for $2$-point correlators. 
\end{enumerate}
The axioms involving $\Pi$ essentially correspond to the statement that the correlator of the disjoint union of two world sheets should be the tensor product of the individual correlators. Finally, the axiom for $\Upsilon$ guarantees that the correlator for the empty manifold $\varnothing$, viewed as a world sheet $\varnothing \colon \varnothing\to\varnothing$ is trivial. These are precisely the consistency conditions one would expect from a consistent system of correlators.

In summary we see that the $1$-morphism component of $\Cor$ corresponds to the field content while the $2$-morphism component are the actual correlators of the full CFT.

Our definition of full CFT is closely related to the notions of twisted or relative field theories considered in \autocite{ST11twist,FT12relQFT,JFS17relTFT}. More precisely a full CFT, as defined above, can be understood as an open-closed, defect variant of a not fully extended relative field theory.

\section{TFT construction of full 2d CFT}\label{sec:tftconstcft}
In this section we will discuss how to use certain $3$-dimensional defect TFTs, obtained from the non-semisimple TFT of \autocite{DGGPR19}, to construct a full conformal field theory, leading to the main result of this article in the form of \Cref{thm:fullcft}. This construction can be seen as a non-semisimple extension of \autocite{FFFS2002,FRSI,FRSII,FRSIII,FRSIV,FjFRS,FFS12} with a particular emphasis on its topological nature. 

We will start with a discussion of the defect TFTs to which our construction can be applied and how the labelling data for the world sheet category $\WS(\DD_\calc)$ is obtained. Afterwards, we will construct the data of a full CFT, as in \Cref{def:fullcft}, by evaluating the $3$d defect TFT on certain manifolds in \Cref{sec:fullcftconst}. The rest of this chapter is devoted to checking that this data satisfies \Cref{def:fullcft}, i.e.\ we will prove that it satisfies the axioms of a braided monoidal oplax natural transformation discussed in \Cref{subsec:oplaxnattrafo}. 

\subsection{Allowed defect TFTs}\label{sec:alloweddefTFT}
As before we fix $\calc$ a modular tensor category and consider
\begin{align}
\widehat{\rmV}_{\calc} \colon \widehat{\bord}_{3,2}^{\chi}(\calc) \to \vect
\end{align}
the version of the TFT constructed from $\calc$ with admissibility condition on bordism components disjoint from the outgoing boundary viewed as defect TFT as in \autocite[Sec.\,3]{HR2024modfunc}. 

Instead of working with an explicit construction of a non-semisimple defect TFT we will assume there exists a defect TFT ``based on'' $\widehat{\rmV}_{\calc}$ satisfying certain algebraic assumptions. The reason we do this is because we want to be able to apply our construction to possible defect TFTs for which we currently do not have an explicit construction. Nonetheless, we do have two explicit non-semisimple defect TFTs in mind to which our construction applies, these will be discussed further below. Apart from this, taking the defect TFT as a ``blackbox'' also has the added benefit of highlighting the topological nature of our construction. 

Let us now first make precise what we mean with $3$d defect TFTs ``based on'' $\widehat{\rmV}_{\calc}$. 
\begin{definition}\label{def:deftftextension}
Let
 \begin{equation}
        \rmZ_\calc \colon \bord_{3,2}^{\chi,\mathrm{ def}}(\DD_\calc) \to \vect
\end{equation}
be a $3$d defect TFT such that its defect data $\DD_\calc$ contains the defect data of $\widehat{\rmV}_\calc$ as a subset in the sense of \autocite[Sec.\,2.3.1]{CRS17deforbi}. Let 
\begin{equation}
   \iota \colon \widehat{\bord}_{3,2}^{\chi}(\calc)  \to \bord_{3,2}^{\chi,\mathrm{ def}}(\DD_\calc)
\end{equation}
be the symmetric monoidal functor obtained from this inclusion of defect data. Then $\rmZ_\calc$ is said to \emph{extend} $\widehat{\rmV}_{\calc}$ if there exists a monoidal natural isomorphism $\Phi \colon \rmZ_\calc \circ \iota \Rightarrow \widehat{\rmV}_\calc$. We call a pair $(\rmZ_\calc,\Phi)$ an \emph{extension} of $\widehat{\rmV}_\calc$. 
\end{definition}
On the level of defect $3$-categories from \Cref{rem:def3cat} an extension $(\rmZ_\calc,\Phi)$ of $\widehat{\rmV}_\calc$ gives a way of exhibiting $B^2\calc$ as a sub-$3$-category of $\calT_{\rmZ_\calc}$.

The simplest example of an extension $(\rmZ_\calc,\Phi)$ is $(\widehat{\rmV}_\calc,\id)$ itself. A more interesting example can be obtained by introducing non-trivial surface defects into $\widehat{\rmV}_\calc$ using the so-called orbifold construction  \autocite{CRS17defRT}. The first example will be discussed in detail in \Cref{sec:logcardy} while the second example will be treated in detail in upcoming work of the first author.

From now on let $(\rmZ_\calc,\Phi)$ be a fixed extension of $\widehat{\rmV}_\calc$. Moreover, we will assume that $D_3$ is a singleton set corresponding to a unique phase of $\rmZ_\calc$. This is because we want to construct full CFTs for which the chiral and anti-chiral sector are governed by the same chiral data, see \autocite[Sec.\,1.1.4]{hofer25thesis} for some remarks on how this condition could be weakened. The aforementioned algebraic assumptions on $\rmZ_\calc$ will be discussed in detail when they come up.
In the following we will often view $\widehat{\bord}_{3,2}^{\chi}(\calc) $ as a (non-full) subcategory of $\bord_{3,2}^{\chi,\mathrm{ def}}(\DD_\calc)$ by identifying objects in $\widehat{\bord}_{3,2}^{\chi}(\calc)$ with their image under $\iota$ in $\bord_{3,2}^{\chi,\mathrm{ def}}(\DD_\calc)$ and accordingly suppress the functor $\iota$. 

\begin{remark}\label{rem:admisscond}
We want to note here that the $3$d defect TFT $\rmZ_\calc$ might also only be defined on a subcategory of $\bord_{3,2}^{\chi,\mathrm{ def}}(\DD_\calc)$ in analogy to how $\widehat{\rmV}_\calc$ is only defined on such an admissible subcategory. However, all bordisms in $\bord_{3,2}^{\chi,\mathrm{ def}}(\DD_\calc)$ which will be needed for our considerations will automatically be in such a subcategory as they will always contain an outgoing boundary.
\end{remark}

We could now use $\rmZ_\calc$ to construct a chiral modular functor
\begin{equation}
    \Bl_{\rmZ_\calc}^\chi \colon \bord_{2+\epsilon,2,1}^{\chi} \to \Prof
\end{equation}
in analogy to the previous chapter. However, this chiral modular functor will automatically be isomorphic to the chiral modular functor
\begin{align}
    \Bl_\calc^{\chi} \colon \bord_{2+\epsilon,2,1}^{\chi} \to \Prof
\end{align} 
constructed from $\widehat{\rmV}_\calc$. To see this note that in the construction we only need surfaces and $3$-bordisms which do not contain codimension $1$-strata, i.e.\ we only work with $\widehat{\bord}_{3,2}^{\chi}(\calc)$ and not all of $\bord_{3,2}^{\chi,\mathrm{ def}}(\DD_\calc)$. This means we can use the natural isomorphism $\Phi$ to construct an isomorphism of chiral modular functors. 

Next we need to discuss the allowed $2$-dimensional defect data $\DD^{2\mathrm{d}}$. As explained in the introduction we know that a full CFT should correspond to a surface defect in the $3$d TFT $\rmZ_\calc$. For this reason we will be interested in labelling data for topological world sheets such that we can view them as surface defects in the $3$d defect TFT $\rmZ_\calc$. Since we already assumed that $D_3$ is a singleton set we can use $\DD^{2\mathrm{d}} = \DD_\calc $ directly. In particular we have a canonical choice for the transparent element $T \in D_2^{2\mathrm{d}}$ of $\DD_\calc$: the unique element in $D_2$ we get from the inclusion of $\calc$ in $\DD_\calc$ as defect data. 

\begin{remark}
In terms of the higher categories of defects this amounts to taking the $2$-category of defects of the full CFT to be the endomorphism $2$-category of the unique object in the $3$-category $\calT_{\rmZ_\calc}$ of defects of $\rmZ_\calc$. Moreover, the canonical choice for the transparent surface defect $T$ is now simply the identity $1$-morphism of the unique object of $\calT_{\rmZ_\calc}$.
\end{remark}

Let $\WS(\DD_\calc)$ be the $(2,1)$-category of topological world sheets with labelling data obtained from $\DD_\calc$ and let us denote with $\Bl_\calc$ the composition
\begin{align}
    \WS(\DD_\calc) \stackrel{U}{\longrightarrow} \bord_{2+\epsilon,2,1}^{\oc} \stackrel{\widehat{(-)}}{\longrightarrow} \bord_{2+\epsilon,2,1}^{\chi} \stackrel{\Bl_\calc^{\chi}}{\longrightarrow}\Prof.
\end{align}
We will now construct a full CFT
\begin{align}
    \begin{tikzcd}[ampersand replacement=\&]
	\WS(\DD_\calc) \&\& {\Prof}
	\arrow[""{name=0, anchor=center, inner sep=0}, "{\Delta_{\kk}}", curve={height=-24pt}, from=1-1, to=1-3]
	\arrow[""{name=1, anchor=center, inner sep=0}, "{\mathrm{Bl_\calc}}"', curve={height=24pt}, from=1-1, to=1-3]
	\arrow["\Cor", shorten <=6pt, shorten >=6pt, Rightarrow, from=0, to=1]
    \end{tikzcd}
\end{align}
as in \Cref{def:fullcft} for the modular functor $\Bl_\calc$.

\subsection{Construction of full CFT}\label{sec:fullcftconst}
In this section we will explain how to construct the $1$- and $2$-morphism components of the  full CFT $\Cor$ for $\Bl_\calc$. As in \autocite{FFFS2002,FRSI}, the basic idea is to obtain the correlators by evaluating the $3$d defect TFT $\rmZ_\calc$ on the so-called \emph{connecting manifolds}. In contrast to their work we will not only consider this on the level of world sheets and the resulting $3$-manifolds but already one dimension lower in order to also obtain the field content.
\subsubsection{Field content}
Let $\frakC \in \WS(\DD_\calc)$ and denote with $\Gamma$ the underlying compact $1$-manifold. We want to construct a left exact profunctor $\Cor_\frakC \colon \Delta_\kk(\frakC) \slashedrightarrow \Bl_\calc(\frakC)$. By definition this is a left exact functor $\calc^{\boxtimes\pi_0(\widehat{\Gamma})} \to \vect$ since $\Bl_\calc(\frakC) = \Bl_\calc^\chi(\widehat{\Gamma}) = \calc^{\boxtimes\pi_0(\widehat{\Gamma})}$. To this end consider the \emph{connecting manifold} $M_\frakC$ of $\frakC$ defined as the defect $2$-manifold with underlying manifold
\begin{equation}
    M_\Gamma = \Gamma \times [-1,1] / \sim \quad \text{with} \quad (p,t) \sim (p,-t) \quad \text{for} \quad p \in \partial \Gamma
\end{equation}
and stratification induced from the inclusion $\Gamma \cong \Gamma \times \{0\} \hookrightarrow M_\Gamma$, see Figure\,\ref{fig:conman1} for an illustration.\footnote{The name connecting manifold comes from the idea that it ``connects'' a manifold $X$ with its orientation double $\widehat{X}$ without adding homotopical information.} Note that by construction $\partial M_\Gamma \cong \widehat{\Gamma}$.
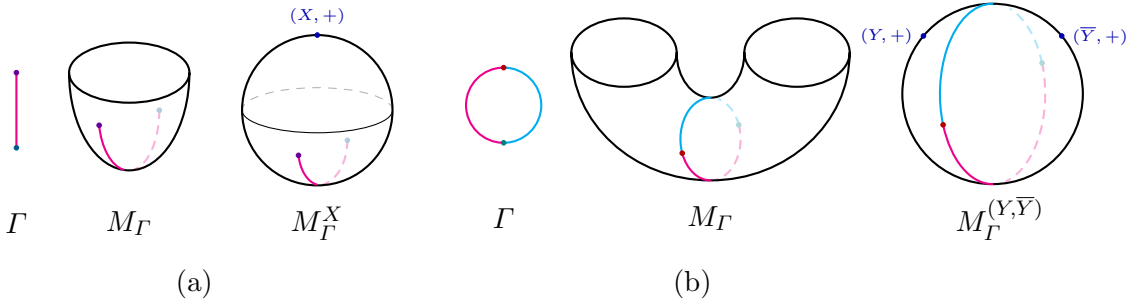
\begin{figure}[t]
    \hspace{0.2cm}
    \begin{subfigure}{0.4\textwidth}
         \centering
            \begin{tikzpicture}
        \draw[thick,magenta] (0,0) -- (0,1);
        \filldraw[string-blue!30!string-green] (0,0) circle (0.03);
        \filldraw[string-violet](0,1) circle (0.03);
        \node at (0,-1) {\small $\Gamma$};
        \begin{scope}[shift={(-1.5,0)}]
        \draw[thick,magenta,dashed,line cap= round] (3.4,0.5) arc (0:-80:0.4 and 0.8);
        \filldraw[string-blue!30!string-green] (3.4,0.5) circle (0.03);
        \fill[white,opacity=0.7](3.8,1) arc (0:-180:0.8 and 1.3);
        \draw[thick] (3,1) ellipse (0.8 and 0.4)
        (3.8,1) arc (0:-180:0.8 and 1.3);
        \draw[thick,magenta,line cap= round] (2.6,0.3) arc (190:260:0.4 and 0.73);
        \filldraw[string-violet] (2.6,0.3) circle (0.03);
        \node at (3,-1) {\small $M_\Gamma$};
        \end{scope}
        \begin{scope}[shift={(-3,0)}]        
        \draw[thick,magenta,line cap=round,dashed] (7.4,0.1) arc (0:-80:0.4 and 0.6);        
        \draw[dashed,line cap=round] (8,0.5) arc (0:180:1 and 0.3); 
        \filldraw[string-blue!30!string-green] (7.4,0.1) circle (0.03);
        \fill[white,opacity=0.7] (7,0.5) circle (1);
        \draw[thick] (7,0.5) circle (1);
        \draw (8,0.5) arc (0:-180:1 and 0.3);
        \draw[thick,magenta,line cap=round] (6.75,-0.1) arc (190:260:0.3 and 0.48);
        \filldraw[string-violet] (6.75,-0.1) circle (0.03);
        \filldraw[string-blue] (7,1.5) circle (0.03);
        \node at (7,1.5) [above,string-blue] {\tiny $(X,+)$};
        \node at (7,-1) {\small $M_\Gamma^X$};
        \end{scope}
    \end{tikzpicture}
        \caption{}\label{fig:connectint}
    \end{subfigure}
    \hspace{0.3cm}
    \begin{subfigure}{0.4\textwidth}
              \centering
                \begin{tikzpicture}
        \draw[thick,magenta,line cap=round] (0.5,0.8) arc (90:270:0.5);
        \draw[thick,cyan,line cap=round] (0.5,0.8) arc (90:-90:0.5);
        \filldraw[string-green] (0.5,-0.2) circle (0.03);
        \filldraw[string-red] (0.5,0.8) circle (0.03);
        \node at (0.5,-1.2) {\small $\Gamma$};
        \begin{scope}[shift={(-1,0)}]
        \draw[thick,magenta,dashed,line cap= round] ([xshift=4.25cm,yshift=-0.15cm] 20:0.4 and 0.55) arc (20:-78:0.4 and 0.55);
        \draw[thick,cyan,dashed,line cap= round] ([xshift=4.25cm,yshift=-0.15cm] 20:0.4 and 0.55) arc (20:75:0.4 and 0.55);
        \filldraw[string-green]
        ([xshift=4.25cm,yshift=-0.15cm] 20:0.4 and 0.55) circle (0.03);
        \fill[white,opacity=0.7](4.25,-1) rectangle (5,1);
        \draw[thick,line cap=round] (3.1,1) ellipse (0.7 and 0.455)
        (5.4,1) ellipse (0.7 and 0.455)
        (4.7,1) arc (0:-180:0.45 and 0.6)
        (6.1,1) arc (0:-180: 1.85 and 1.7);
        \draw[thick,magenta,line cap= round] ([xshift=4.25cm,yshift=-0.15cm] 200:0.4 and 0.55) arc (200:270:0.4 and 0.55);
        \draw[thick,cyan,line cap= round] ([xshift=4.25cm,yshift=-0.15cm] 200:0.4 and 0.55) arc (200:90:0.4 and 0.55);
        \filldraw[string-red] ([xshift=4.25cm,yshift=-0.15cm] 200:0.4 and 0.55) circle (0.03);
        \node at (4.25,-1.2) {\small $M_\Gamma$};
        \end{scope}
        \begin{scope}[shift={(-1.6,0)}]
        \draw[thick,magenta,dashed,line cap= round] ([xshift=8.6cm,yshift=0.45cm] 20:0.7 and 1.2) arc (20:-75:0.7 and 1.2);
        \draw[thick,cyan,dashed,line cap= round] ([xshift=8.6cm,yshift=0.45cm] 20:0.7 and 1.2) arc (20:70:0.7 and 1.2);
        \filldraw[string-green] ([xshift=8.6cm,yshift=0.45cm] 20:0.7 and 1.2) circle (0.03);
        \fill[white,opacity=0.7] (8.6,0.45) circle (1.2);
        \draw[thick] (8.6,0.45) circle (1.2);
        \draw[thick,magenta,line cap= round] ([xshift=8.6cm,yshift=0.45cm] 200:0.7 and 1.2) arc (200:270:0.7 and 1.2);
        \draw[thick,cyan,line cap= round] ([xshift=8.6cm,yshift=0.45cm] 90:0.7 and 1.2) arc (90:200:0.7 and 1.2);
        \filldraw[string-red] ([xshift=8.6cm,yshift=0.45cm] 200:0.7 and 1.2) circle (0.03);
        \filldraw[string-blue] ([xshift=8.6cm,yshift=0.45cm] 140:1.2) circle (0.03)
        ([xshift=8.6cm,yshift=0.45cm] 40:1.2) circle (0.03);
        \node at ([xshift=8.6cm,yshift=0.45cm] 140:1.2) [left,string-blue] {\tiny $(Y,+)$};
        \node at ([xshift=8.6cm,yshift=0.45cm] 40:1.2) [right,string-blue] {\tiny $(\overline{Y},+)$};
        \node at (8.7,-1.2) {\small $M_\Gamma^{(Y,\overline{Y})}$};
        \end{scope}
    \end{tikzpicture}
            \caption{}\label{fig:connectcirc}
        \end{subfigure}
    \caption{Connecting manifold $M_\Gamma$ for $\Gamma$ an interval (a) and a defect circle (b) with and without disc(s) glued in. Here the different colours are used to indicate different $\DD$-labels of the strata except for the blue marked and labelled points which are always labelled by elements in $\calc$. }
    \label{fig:conman1}
\end{figure}
Now let $\underline{X} \in \calc^{\times \pi_0(\widehat{\Gamma})}$. We construct an object $M_\Gamma^{\underline{X}} \in \bord_{3,2}^{\chi,\mathrm{def}}(\DD_\calc)$ as in \autocite[Sec.\,4.2]{HR2024modfunc} by gluing discs $D_\Gamma^{\underline{X}}$ bounding $\widehat{\Gamma}$ with $\underline{X}$-labelled, positively oriented, marked points into the boundary of $M_\Gamma$, see \Cref{fig:conman1} for an illustration.\footnote{Choosing the marked points to be positively oriented is a convention, which we use to ensure that we get covariant functors.}
By a similar argument as in \autocite[Sec.\,4.2]{HR2024modfunc} this can be extended to a functor
\begin{equation}
    \begin{aligned}
        \mathrm{cor}_\Gamma \colon \calc^{\times \pi_0(\widehat{\Gamma})} &\to \vect \\
            \underline{X} &\mapsto \rmZ_\calc(M_\Gamma^{\underline{X}}).
    \end{aligned}
\end{equation}
For later use let us also consider the functor $\mathrm{cor}_{\Gamma}^\dagger$ obtained analogously as $\mathrm{cor}_{\Gamma}$ but for $-M_\Gamma$ instead, i.e.\ all marked points are negatively oriented leading to a contravariant functor.

Since we do not have an explicit construction of $\rmZ_\calc$ we cannot guarantee that $\mathrm{cor}_\Gamma$ is left exact in general. This will give us to the first assumption on $\rmZ_\calc$: 
\begin{assumption}\label{ass:1}
    For any $\frakC \in \WS(\DD_\calc)$, the functors $\mathrm{cor}_\Gamma\colon \calc^{\times \pi_0(\widehat{\Gamma})} \to \vect$ and $\mathrm{cor}_{\Gamma}^\dagger\colon (\calc^\op)^{\times \pi_0(\widehat{\Gamma})} \to \vect$ are linear and left exact in each variable.
    
    As the notation suggests we further assume the extensions of $\mathrm{cor}_{\Gamma}$ and $\mathrm{cor}_{\Gamma}^\dagger$ to the Deligne product to be adjoint profunctors as in \Cref{lem:profadjunc}.
\end{assumption}

Under \Cref{ass:1} we can now extend $\mathrm{cor}_\Gamma$ to the Deligne tensor product.
\begin{definition}[$\Cor$ on objects]
    For $\frakC \in \WS(\DD_\calc)$ we define
    \begin{equation}
            \Cor_\frakC \colon \Delta_\kk(\frakC) \slashedrightarrow \Bl_\calc(\frakC) 
    \end{equation}
    to be the functor induced by $\mathrm{cor}_\Gamma$ on the Deligne tensor product $\calc^{\boxtimes \pi_0(\widehat{\Gamma})}$. 
\end{definition}

\subsubsection{Correlators}
Next let us construct the $2$-morphism components of $\Cor$. Let $\frakS \colon \frakC \to \frakC$ be a $1$-morphism in $\WS(\DD_\calc)$ and denote with $\Sigma \colon \Gamma \to \Gamma'$ its underlying open-closed bordism. The \emph{connecting manifold} $M_\frakS$ of $\frakS$ is the defect $3$-manifold (with corners) with underlying $3$-manifold (with corners)
\begin{equation}
    M_\Sigma = \Sigma \times [-1,1] / \sim \quad \text{with} \quad (p,t) \sim (p,-t) \quad \text{for} \quad p \in \partial^{\mathrm{f}} \Sigma
\end{equation}
and stratification induced from the embedding $\Sigma \cong \Sigma \times \{0\} \hookrightarrow M_\Sigma$, see Figure\,\ref{fig:conmanbord} for an illustration in the example of the $1$-morphism from \eqref{eq:worldsheetmorphism3d}.
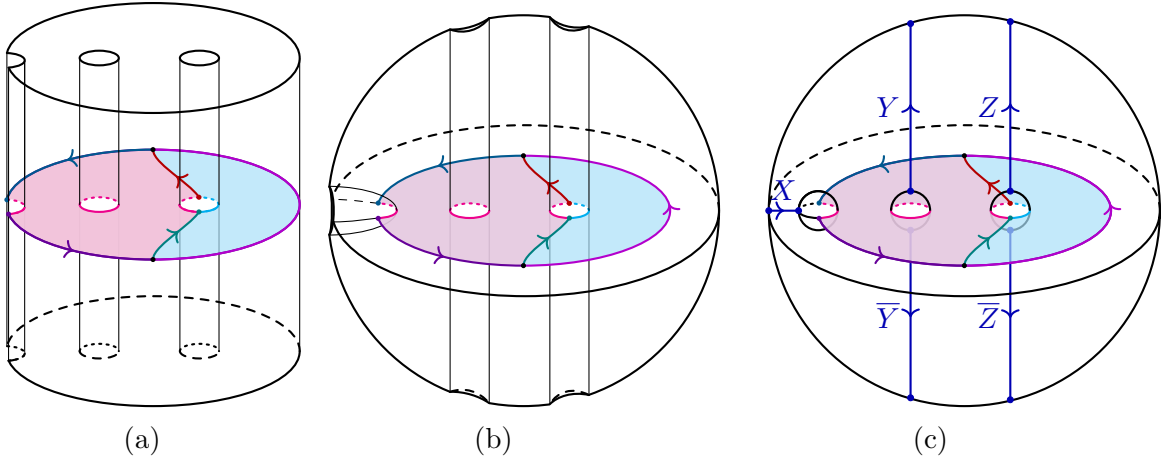
\begin{figure}[t]
    \begin{subfigure}{0.3\textwidth}
     \centering
            \begin{tikzpicture}[scale=0.65]
        \draw[thick] (0,3) ellipse (3 and 1.125)
        (-1.1,3) ellipse (0.4 and 0.15)
        (0.95,3) ellipse (0.4 and 0.15);
        \draw[thick,line cap=round] (3,-3) arc (0:-180:3 and 1.125);
        \draw[thick,dashed,line cap=round] (3,-3) arc (0:180:3 and 1.125);
        \draw[thick,dashed,line cap=round] (-1.5,-3) arc (180:360:0.4 and 0.15)
        (0.55,-3) arc (180:360:0.4 and 0.15);
        \draw[thick,dotted,line cap=round] (-1.5,-3) arc (180:0:0.4 and 0.15)
        (0.55,-3) arc (180:0:0.4 and 0.15);
        \draw[line cap=round] (-1.5,-3) -- (-1.5,0)
        (-0.7,-3) -- (-0.7,0)
        (1.35,-3) -- (1.35,0)
        (0.55,-3) -- (0.55,0)
        (-2.625,-3.05) -- (-2.625,2.95);
        \fill[white] ([yshift=3cm] 190:3 and 1.125) arc (-80:280:0.4 and 0.15)
        ([yshift=-3cm] 190:3 and 1.125) arc (-80:280:0.4 and 0.15);
        \filldraw[fill=white,draw=black,thick,line cap= round] ([yshift=3cm] 190:3 and 1.125) arc (-80:85:0.4 and 0.15);
        \draw[thick,line cap= round,dashed]([yshift=-3cm] 190:3 and 1.125) arc (-80:0:0.4 and 0.15);
        \draw[thick,dotted,line cap=round] ([yshift=-3cm] 190:3 and 1.125) arc (-80:85:0.4 and 0.15);
        \draw (3,3) -- (3,-3);
        \fill[cyan!30,opacity=0.7] (0,0) ellipse (3 and 1.125);
        \fill[magenta!30,opacity=0.7] {[rounded corners] (0,-1.125) -- (0,-1) -- (1,-0.125) -- (1,0) -- (1,0.125) -- (0,1) -- (0,1.125)}
        (0,1.125) arc (90:270:3 and 1.125);
        \fill[magenta!30,opacity=0.7] {[rounded corners] (0,-1.125) -- (0,-1) -- (1,-0.125) -- (1,0) -- (1,0.125) -- (0,1) -- (0,1.125)}
        (0,1.125) arc (90:270:3 and 1.125);
        \draw[draw=string-violet,thick] (0,0) ellipse (3 and 1.125);
        \begin{scope}[very thick,decoration={
                      markings,
                      mark=at position 0.52 with {\arrow{>}}}
                     ] 
        \draw[thick,string-green,postaction={decorate}](0,-1.125) .. controls (0,-0.8) and (1,-0.325) .. (1,0);
        \draw[thick,string-red,postaction={decorate}] (1,0) .. controls (1,0.325) and (0,0.8) .. (0,1.125);
        \draw[thick,string-blue!30!string-green,postaction={decorate}] (0,1.125) arc (90:180:3 and 1.125);
        \draw[thick,string-violet,postaction={decorate}] (-3,0) arc (180:270:3 and 1.125);
        \draw[thick,string-violet!50!magenta] (0,-1.125) arc (-90:90:3 and 1.125);
        \end{scope}
        \filldraw[white] ( 190:3 and 1.125) arc (-80:280:0.4 and 0.15);
        \fill[white] (190:3 and 1.125) arc (-80:85:0.4 and 0.15);
        \draw[draw=magenta,thick] (190:3 and 1.125) arc (-80:0:0.4 and 0.15);
        \draw[draw=magenta,thick,dotted,line cap=round] (190:3 and 1.125) arc (-80:85:0.4 and 0.15);
        \fill[white] (0.95,0) ellipse (0.4 and 0.15);
        \filldraw[fill=white,draw=cyan,thick] (0.94,-0.15) arc (-90:0:0.4 and 0.15);
        \draw[cyan,thick,dotted,line cap=round] (0.94,-0.15) arc (-90:90:0.4 and 0.15);
        \filldraw[fill=white,draw=magenta,thick] (0.54,0) arc (180:270:0.4 and 0.15);
        \draw[draw=magenta,thick,dotted,line cap=round] (0.94,0.15) arc (90:270:0.4 and 0.15);
        \fill[white] (-1.1,0) ellipse (0.4 and 0.15);
        \draw[magenta,thick] (-1.5,0) arc (180:360:0.4 and 0.15);
        \draw[magenta,thick,dotted,line cap=round] (-1.5,0) arc (180:0:0.4 and 0.15);
        \draw[line cap= round] ([yshift=3cm] 175:3 and 1.125) -- ([yshift=-3cm] 175:3 and 1.125)
        (-1.5,3) -- (-1.5,0)
        (-0.7,3) -- (-0.7,0)
        (1.35,3) -- (1.35,0)
        (0.55,3) -- (0.55,0);
        \filldraw[string-blue!30!string-green,thick] (175:3 and 1.125) circle (0.03);
        \draw[line cap= round] ([yshift=3cm] 190:3 and 1.125) -- ([yshift=-3cm] 190:3 and 1.125);
        \filldraw[thick] (0,-1.125) circle (0.03);
        \filldraw[thick](0,1.125) circle (0.03);
        \filldraw[string-violet,thick](190:3 and 1.125) circle (0.03);
        \filldraw[string-green,thick](0.94,-0.15) circle (0.03);
        \filldraw[string-red,thick](0.94,0.15) circle (0.03);
        \draw (-2.625,-0.05) -- (-2.625,2.95);
        \end{tikzpicture}
    \caption{}
     \label{fig:cyl1}
    \end{subfigure}
    \begin{subfigure}{0.3\textwidth}
     \centering
            \begin{tikzpicture}[scale=0.65]
        \draw[thick] (0,0) circle (4)
        (4,0) arc (0:-173: 3.96 and 1.75);
        \draw[thick,dashed,line cap=round,] (4,0) arc (0:183: 3.98 and 1.75);
        \fill[white] (1.335,-3.65) -- (1.335,-3.95) -- (0.545,-4) -- (0.545,-3.65) -- cycle;
        \fill[white] (-1.5,-3.65) -- (-1.5,-3.95) -- (-0.7,-4) -- (-0.7,-3.65) -- cycle;
        \draw (1.34,0) -- (1.34,-3.75)
        (0.54,0) -- (0.54,-3.95)
        (-1.5,0) -- (-1.5,-3.7)
        (-0.7,0) -- (-0.7,-3.95);
        \fill[white] (-4.1,-0.49) -- (-3.95,-0.49) -- (-3.9,0.49) -- (-4.1,0.49) -- cycle;
        \draw (-2.59,0) arc (0:-87:1.45 and 0.5);
        \fill[cyan!30,opacity=0.7] (0,0) ellipse (3 and 1.125);
        \fill[magenta!30,opacity=0.7] {[rounded corners] (0,-1.125) -- (0,-1) -- (1,-0.125) -- (1,0) -- (1,0.125) -- (0,1) -- (0,1.125)}
        (0,1.125) arc (90:270:3 and 1.125);
        \begin{scope}[very thick,decoration={
                      markings,
                      mark=at position 0.52 with {\arrow{>}}}
                     ] 
        \draw[thick,string-green,postaction={decorate}](0,-1.125) .. controls (0,-0.8) and (1,-0.325) .. (1,0);
        \draw[thick,string-red,postaction={decorate}] (1,0) .. controls (1,0.325) and (0,0.8) .. (0,1.125);
        \draw[thick,string-blue!30!string-green,postaction={decorate}] (0,1.125) arc (90:180:3 and 1.125);
        \draw[thick,string-violet,postaction={decorate}] (-3,0) arc (180:270:3 and 1.125);
        \draw[thick,string-violet!50!magenta,postaction={decorate}] (0,-1.125) arc (-90:90:3 and 1.125);
        \end{scope}
        \filldraw[white](-3,0) circle (0.15);
        \fill[white] (-2.99,-0.15) arc (-90:90:0.4 and 0.15);
        \draw[magenta,thick] (-2.99,-0.15) arc (-90:0:0.4 and 0.15);
        \draw[magenta,thick,dotted,line cap=round] (-2.59,0) arc (0:90:0.4 and 0.15);
        \filldraw[white](0.94,0) ellipse (0.4 and 0.15);
        \draw[draw=cyan,thick] (0.94,-0.15) arc (-90:0:0.4 and 0.15);
        \draw[draw=cyan,thick,dotted,line cap=round] (1.34,0) arc (0:90:0.4 and 0.15);
        \draw[magenta,thick,dotted,line cap=round] (0.94,0.15) arc (90:180:0.4 and 0.15);
        \draw[magenta,thick] (0.54,0) arc (180:270:0.4 and 0.15);
        \fill[white] (-1.1,0) ellipse (0.4 and 0.15);
        \draw[magenta,thick] (-1.5,0) arc (180:360:0.4 and 0.15);
        \draw[magenta,thick,dotted,line cap=round] (-1.5,0) arc (180:0:0.4 and 0.15);
        \draw (-2.59,0) arc (0:87:1.45 and 0.5);
        \draw (-2.97,-0.15) .. controls (-3.25,-0.18) and (-3.5,-0.23) .. (-3.9,-0.23);
        \draw[dashed,line cap=round] (-2.97,0.15) .. controls (-3.25,0.18) and (-3.5,0.23) .. (-3.9,0.23);
        \filldraw[thick] (0,-1.125) circle (0.03);
        \filldraw[thick](0,1.125) circle (0.03);
        \filldraw[string-violet,thick](188:3 and 1.125) circle (0.03);
        \filldraw[string-blue!30!string-green,thick](172:3 and 1.125) circle (0.03);
        \filldraw[string-green,thick](0.94,-0.15) circle (0.03);
        \filldraw[string-red,thick](0.94,0.15) circle (0.03);
        \fill[white] (1.335,3.65) -- (1.335,3.95) -- (0.545,4) -- (0.545,3.65) -- cycle;
        \fill[white] (-1.5,3.65) -- (-1.5,3.95) -- (-0.7,4) -- (-0.7,3.65) -- cycle;
        \draw (1.34,0) -- (1.34,3.765)
        (0.54,0) -- (0.54,3.945)
        (-1.5,0) -- (-1.5,3.7)
        (-0.7,0) -- (-0.7,3.95);
        \draw[thick,line cap=round]  (1.337,3.772) .. controls (1,3.6) and (0.7,3.7) .. (0.54,3.96) ;
        \draw[thick,line cap=round]  (1.337,3.772) .. controls (1,3.7) and (0.7,3.8) .. (0.54,3.96) ;
        \draw[thick,line cap=round]  (-1.5,3.709) .. controls (-1.2,3.6) and (-0.9,3.7) .. (-0.7,3.94) ;
        \draw[thick,line cap=round]  (-1.5,3.707) .. controls (-1.2,3.7) and (-0.9,3.8) .. (-0.7,3.94) ;
        \draw[thick,dashed,line cap=round,line cap=round]  (1.337,-3.772) .. controls (1,-3.6) and (0.7,-3.7) .. (0.54,-3.96) ;
        \draw[thick,line cap=round]  (1.337,-3.772) .. controls (1,-3.7) and (0.7,-3.8) .. (0.54,-3.96) ;
        \draw[thick,dashed,line cap=round,line cap=round]  (-1.5,-3.709) .. controls (-1.2,-3.6) and (-0.9,-3.7) .. (-0.7,-3.94) ;
        \draw[thick,line cap=round]  (-1.5,-3.707) .. controls (-1.2,-3.7) and (-0.9,-3.8) .. (-0.7,-3.94);
        \draw[thick,line cap=round] (-3.97,-0.49) .. controls (-3.85,-0.35) and (-3.85,0.35) .. (-3.97,0.49);
        \draw[thick,line cap=round] (-3.97,-0.49) .. controls (-3.9,-0.35) and (-3.9,0.35) .. (-3.97,0.49);
    \end{tikzpicture}
    \caption{}
    \label{fig:conman2}
    \end{subfigure}
    \hspace{1cm}
    \begin{subfigure}{0.3\textwidth}
     \centering
            \begin{tikzpicture}[scale=0.65]
        \draw[black,thick] (-1.1,0) circle (0.4)
        (0.94,0) circle (0.4)
        (-2.99,0) circle (0.4);
        \filldraw[string-blue]
        (0.94,-0.4) circle (0.06)
        (-1.1,-0.4) circle (0.06);
        \begin{scope}[very thick,decoration={
                      markings,
                      mark=at position 0.52 with {\arrow{>}}}
                     ] 
        \draw[string-blue,thick,postaction={decorate}] (0.94,-0.4) -- (0.94,-3.86);  
        \draw[string-blue,thick,postaction={decorate}] (-1.1,-0.4) -- (-1.1,-3.83); 
        \end{scope}
        \fill[cyan!30,opacity=0.7] (0,0) ellipse (3 and 1.125);
        \fill[magenta!30,opacity=0.7] {[rounded corners] (0,-1.125) -- (0,-1) -- (1,-0.125) -- (1,0) -- (1,0.125) -- (0,1) -- (0,1.125)}
        (0,1.125) arc (90:270:3 and 1.125);
        \draw[draw=string-violet,thick] (0,0) ellipse (3 and 1.125);
        \begin{scope}[very thick,decoration={
                      markings,
                      mark=at position 0.52 with {\arrow{>}}}
                     ] 
        \draw[thick,string-green,postaction={decorate}](0,-1.125) .. controls (0,-0.8) and (1,-0.325) .. (1,0);
        \draw[thick,string-red,postaction={decorate}] (1,0) .. controls (1,0.325) and (0,0.8) .. (0,1.125);
        \draw[thick,string-blue!30!string-green,postaction={decorate}] (0,1.125) arc (90:180:3 and 1.125);
        \draw[thick,string-violet,postaction={decorate}] (-3,0) arc (180:270:3 and 1.125);
        \draw[thick,string-violet!50!magenta,postaction={decorate}] (0,-1.125) arc (-90:90:3 and 1.125);
        \end{scope}
        \filldraw[white](-3,0) circle (0.15);
        \fill[white] (-2.99,-0.15) arc (-90:90:0.4 and 0.15);
        \draw[magenta,thick] (-2.99,-0.15) arc (-90:0:0.4 and 0.15);
        \draw[magenta,thick,dotted,line cap=round] (-2.59,0) arc (0:90:0.4 and 0.15);
        \filldraw[white](0.94,0) ellipse (0.4 and 0.15);
        \draw[draw=cyan,thick] (0.94,-0.15) arc (-90:0:0.4 and 0.15);
        \draw[draw=cyan,thick,dotted,line cap=round] (1.34,0) arc (0:90:0.4 and 0.15);
        \draw[magenta,thick,dotted,line cap=round] (0.94,0.15) arc (90:180:0.4 and 0.15);
        \draw[magenta,thick] (0.54,0) arc (180:270:0.4 and 0.15);
        \fill[white] (-1.1,0) ellipse (0.4 and 0.15);
        \draw[magenta,thick] (-1.5,0) arc (180:360:0.4 and 0.15);
        \draw[magenta,thick,dotted,line cap=round] (-1.5,0) arc (180:0:0.4 and 0.15);
        \draw[black,thick] (-0.7,0) arc (0:180:0.4)
        (-2.59,0) arc (0:180:0.4)
        (1.34,0) arc (0:180:0.4)
        (-3.39,0) arc (180:270:0.4 and 0.15);
        \draw[thick,dotted,line cap=round](-3.39,0) arc (180:90:0.4 and 0.15);
        \filldraw[thick] (0,-1.125) circle (0.03);
        \filldraw[thick](0,1.125) circle (0.03);
        \filldraw[string-violet,thick](188:3 and 1.125) circle (0.03);
        \filldraw[string-blue!30!string-green,thick](172:3 and 1.125) circle (0.03);
        \filldraw[string-green,thick](0.94,-0.15) circle (0.03);
        \filldraw[string-red,thick](0.94,0.15) circle (0.03);
        \draw[thick] (0,0) circle (4)
        (4,0) arc (0:-180: 4 and 1.75);
        \draw[thick,dashed,line cap=round] (4,0) arc (0:180: 4 and 1.75);
        \filldraw[string-blue] (-4,0) circle (0.06)
        (0.94,3.87) circle (0.06)
        (-1.1,3.83) circle (0.06)
        (0.94,-3.87) circle (0.06)
        (-1.1,-3.83) circle (0.06)
        (-3.39,0) circle (0.06)
        (0.94,0.4) circle (0.06)
        (-1.1,0.4) circle (0.06); 
        \begin{scope}[very thick,decoration={
                      markings,
                      mark=at position 0.52 with {\arrow{>}}}
                     ] 
        \draw[string-blue,thick,postaction={decorate}] (0.94,0.4) -- (0.94,3.86);  
        \draw[string-blue,thick,postaction={decorate}] (-1.1,0.4) -- (-1.1,3.83); 
        \draw[string-blue,thick,postaction={decorate}] (-4,0) -- (-3.39,0); 
        \end{scope}
        \node at (-3.7,0) [above,string-blue] {\small $X$};
        \node at (-1.1,2.13) [left,string-blue] {\small $Y$};
        \node at (-1.1,-2.13) [left,string-blue] {\small $\overline{Y}$};
        \node at (0.94,2.13) [left,string-blue] {\small $Z$};
        \node at (0.94,-2.13) [left,string-blue] {\small $\overline{Z}$};
    \end{tikzpicture}
    \caption{}
    \label{fig:conman3}
    \end{subfigure}
    \caption{Schematic construction of the connecting bordism $M_\Sigma^{\underline{X},\underline{Y}}$ for $\Sigma$ the $1$-morphism in $\WS(\DD_\calc)$ from \eqref{eq:worldsheetmorphism3d}: (a) cylinder over $\Sigma$ with $\Sigma$ as surface defect, (b) the connecting manifold $M_\Sigma$ with corners, (c) connecting bordism $M_\Sigma^{\underline{X},\underline{Y}}$;}
    \label{fig:conmanbord}
\end{figure}

Note that the boundary of $M_\Sigma$ naturally decomposes as 
\begin{equation}
    \partial M_\Sigma \cong  -M_\Gamma \cup \widehat{\Sigma} \cup M_{\Gamma'}
\end{equation} 
and $\widehat{\Gamma} \sqcup \widehat{\Gamma}'$ as corner points. This can be seen by noting that
\begin{equation}
    \partial (\Sigma \times [-1,1]) \cong \partial \Sigma \times [-1,1] \sqcup_{(\Gamma\sqcup\Gamma') \times \{\pm1\}} \Sigma \times \{\pm1\}
\end{equation}
and tracking how the equivalence relation identifies points on the boundary, see \Cref{fig:cyl1} and \ref{fig:conman2} for an illustration. 

In particular, the boundary parametrisation $\partial \Sigma \cong -\Gamma \sqcup \Gamma'$ induces a parametrisation of the $M_\Gamma$ and $M_{\Gamma'}$ boundary components. Using this we can turn $M_\Sigma$ into a morphism in $\borddef$ as follows: Let $\underline{X}\in \calc^{\times \pi_0(\widehat{\Gamma})}$ and $\underline{Y}\in \calc^{\times \pi_0(\widehat{\Gamma}')}$, and let $M_\Gamma^{\underline{X}}$ and $M_{\Gamma'}^{\underline{Y}}$ be the corresponding objects in $\borddef$ obtained by gluing in discs as above. By construction these objects come with embeddings $M_\Gamma \hookrightarrow M_\Gamma^{\underline{X}}$ and $M_{\Gamma'} \hookrightarrow M_{\Gamma'}^{\underline{Y}}$. Now consider the cylinders $M_\Gamma^{\underline{X}} \times [-1,1]$ and $M_{\Gamma'}^{\underline{Y}} \times [-1,1]$ with standard boundary parametrisations as identity morphisms in $\borddef$. Using the embeddings $M_{\Gamma'} \hookrightarrow M_{\Gamma'}^{\underline{Y}}$ we want to view them as bordisms 
\begin{equation}
\begin{aligned}
    M_\Gamma^{\underline{X}} \times [-1,1] &\colon M_\Gamma^{\underline{X}} \to M_\Gamma \sqcup_{\widehat{\Gamma}} D_{\Gamma}^{\underline{X}} \\
    M_{\Gamma'}^{\underline{Y}} \times [-1,1] &\colon M_{\Gamma'}^{\underline{Y}} \to M_{\Gamma'} \sqcup_{\widehat{\Gamma}'} D_{\Gamma'}^{\underline{Y}}
\end{aligned}
\end{equation}
where $D_{\Gamma}^{\underline{X}}$ and $D_{\Gamma'}^{\underline{Y}}$ are the discs with marked points which were glued into $M_\Gamma$ and $M_{\Gamma'}$, see \Cref{fig:cylindersforconnmani} for an illustration.

\begin{figure}[t!]
    \centering
        \begin{tikzpicture}[scale=1.2]
        \draw[black,thick]
        (-3.8,0) arc (0:-180:0.4) 
        (-1.1,0) circle (0.4)
        (2,0) circle (0.4);
        \fill[cyan!30,opacity=0.7] ([shift={(2,0)}]240:0.4 and 0.15) --
        ([shift={(2,0)}]240:0.8 and 0.3) arc (240:60:0.8 and 0.3) -- ([shift={(2,0)}]60:0.4 and 0.15) arc (60:240:0.4 and 0.15);
        \fill[magenta!30,opacity=0.7] ([shift={(2,0)}]240:0.4 and 0.15) --
        ([shift={(2,0)}]240:0.8 and 0.3) arc (240:60:0.8 and 0.3) -- ([shift={(2,0)}]60:0.4 and 0.15) arc (60:240:0.4 and 0.15);
        \fill[cyan!30,opacity=0.7] ([shift={(2,0)}]240:0.4 and 0.15) --
        ([shift={(2,0)}]240:0.8 and 0.3) arc (240:420:0.8 and 0.3) -- ([shift={(2,0)}]420:0.4 and 0.15) arc (420:240:0.4 and 0.15);
        \filldraw[white](2,0) ellipse (0.4 and 0.15);
        \draw[cyan,thick] (2.4,0) arc (0:-120:0.4 and 0.15)
        (2.8,0) arc (0:-120:0.8 and 0.3);
        \draw[cyan,thick,dotted,line cap=round] (2.4,0) arc (0:60:0.4 and 0.15)
        (2.8,0) arc (0:60:0.8 and 0.3);
        \draw[string-red,thick]([shift={(2,0)}]60:0.4 and 0.15) --
        ([shift={(2,0)}]60:0.8 and 0.3);
        \filldraw[string-green,thick]([shift={(2,0)}]240:0.4 and 0.15) --
        ([shift={(2,0)}]240:0.8 and 0.3);
        \draw[magenta,thick,dotted,line cap=round] ([shift={(2,0)}]60:0.4 and 0.15) arc (60:180:0.4 and 0.15)
        ([shift={(2,0)}]60:0.8 and 0.3) arc (60:180:0.8 and 0.3);
        \draw[magenta,thick,line cap=round] (1.6,0) arc (180:240:0.4 and 0.15)
         ([shift={(2,0)}]180:0.8 and 0.3) arc (180:240:0.8 and 0.3);
        \fill[white] (-1.1,0) ellipse (0.4 and 0.15);
        \filldraw[string-green,thick]([shift={(2,0)}]240:0.4 and 0.15) circle (0.03)
        ([shift={(2,0)}]240:0.8 and 0.3) circle (0.03);
        \filldraw[string-red,thick] ([shift={(2,0)}]60:0.4 and 0.15) circle (0.03)
        ([shift={(2,0)}]60:0.8 and 0.3) circle (0.03);
        \fill[cyan!30,opacity=0.7, even odd rule] (-1.1,0) ellipse (0.4 and 0.15)
        (-1.1,0) ellipse (0.8 and 0.3);
        \fill[magenta!30,opacity=0.7, even odd rule] (-1.1,0) ellipse (0.4 and 0.15)
        (-1.1,0) ellipse (0.8 and 0.3);
        \draw[magenta,thick] (-1.5,0) arc (180:360:0.4 and 0.15)
        (-1.9,0) arc (180:360:0.8 and 0.3);
        \draw[magenta,thick,dotted,line cap=round] (-1.5,0) arc (180:0:0.4 and 0.15)
        (-1.9,0) arc (180:0:0.8 and 0.3);
        \draw[black,thick] (-0.7,0) arc (0:180:0.4)
        (2.4,0) arc (0:180:0.4);
        \draw[white,ultra thick] ([shift={(-4.2,-1.2)}]30:0.8 and 2.5) arc (-30:-150:0.8 and 0.6);
        \draw ([shift={(-4.2,-1.2)}]30:0.8 and 2.5) arc (-30:-150:0.8 and 0.6);
        \draw[dotted, line cap=round] ([shift={(-4.2,-1.2)}]30:0.8 and 2.5) arc (30:150:0.8 and 0.5);
        \draw[thick] (-3.4,-1.2) arc (0:180:0.8 and 2.5);
        \draw[thick] (-3.8,0) arc (0:-180:0.4 and 0.15);
        \draw[thick,dotted, line cap=round] (-3.8,0) arc (0:180:0.4 and 0.15);
        \fill[cyan!30,opacity=0.7] (-3.8,0) arc (0:180:0.4) -- ([shift={(-4.2,-1.2)}]150:0.8 and 2.5) arc (150:30:0.8 and 2.5) -- cycle;
        \fill[magenta!30,opacity=0.7] (-3.8,0) arc (0:180:0.4) -- ([shift={(-4.2,-1.2)}]150:0.8 and 2.5) arc (150:30:0.8 and 2.5) -- cycle;
        \draw[magenta,thick] (-3.8,0) arc (0:180:0.4);
        \draw[magenta,thick] ([shift={(-4.2,-1.2)}]30:0.8 and 2.5) arc (30:150:0.8 and 2.5); 
        \draw[string-blue!30!string-green,thick] (-3.8,0) -- ([shift={(-4.2,-1.2)}]30:0.8 and 2.5); 
        \draw[string-violet,thick] (-4.6,0) -- ([shift={(-4.2,-1.2)}]150:0.8 and 2.5);
        \filldraw[string-blue!30!string-green,thick](-3.8,0) circle (0.03) ([shift={(-4.2,-1.2)}]30:0.8 and 2.5) circle (0.03);
        \filldraw[string-violet,thick](-4.6,0) circle (0.03) ([shift={(-4.2,-1.2)}]150:0.8 and 2.5) circle (0.03);
        \filldraw[string-blue] (2,1.5) circle (0.06)
        (-1.1,1.5) circle (0.06)
        (2,-1.5) circle (0.06)
        (-1.1,-1.5) circle (0.06)
        (2,0.4) circle (0.06)
        (-1.1,0.4) circle (0.06)
        (2,-0.4) circle (0.06)
        (-1.1,-0.4) circle (0.06)
        (-4.2,-0.4) circle (0.06)
        (-4.2,-1.2) circle (0.06); 
        \draw[dashed, line cap=round,thick] (2.8,-1.5) arc (0:180:0.8 and 0.3)
        (-0.3,-1.5) arc (0:180:0.8 and 0.3)
        (-3.4,-1.2) arc (0:180:0.8 and 0.3);
        \begin{scope}[very thick,decoration={
                      markings,
                      mark=at position 0.55 with {\arrow{>}}}
                     ] 
        \draw[string-blue,thick,postaction={decorate}] (2,0.4) -- (2,1.5);  
        \draw[string-blue,thick,postaction={decorate}] (-1.1,0.4) -- (-1.1,1.5);
        \draw[string-blue,thick,postaction={decorate}] (2,-0.4) -- (2,-1.5);  
        \draw[string-blue,thick,postaction={decorate}] (-1.1,-0.4) -- (-1.1,-1.5);
        \end{scope}
        \begin{scope}[very thick,decoration={
                      markings,
                      mark=at position 0.65 with {\arrow{>}}}
                     ] 
        \draw[string-blue,thick,postaction={decorate}] (-4.2,-1.2) -- (-4.2,-0.4) ; 
        \end{scope}
        \draw[white,ultra thick] (2.8,1.5) arc (0:-180:0.8 and 0.3);
        \draw[thick] (2.8,1.5) -- (2.8,-1.5)    
        (1.2,1.5) -- (1.2,-1.5);
        \draw[thick] (2,1.5) ellipse (0.8 and 0.3);
        \draw[thick] (2.8,-1.5) arc (0:-180:0.8 and 0.3);
        \draw[white,ultra thick] (-0.3,1.5) arc (0:-180:0.8 and 0.3);
        \draw[thick] (-0.3,1.5) -- (-0.3,-1.5)    
        (-1.9,1.5) -- (-1.9,-1.5);
        \draw[thick] (-1.1,1.5) ellipse (0.8 and 0.3);
        \draw[thick] (-0.3,-1.5) arc (0:-180:0.8 and 0.3)
        (-3.4,-1.2) arc (0:-180:0.8 and 0.3); 
        \node at (2,0.9) [string-blue,left] {\scriptsize $Z$}; 
        \node at (-1.1,0.9) [string-blue,left] {\scriptsize $Y$};
        \node at (2,-0.85) [string-blue,left] {\scriptsize $\overline{Z}$}; 
        \node at (-1.1,-0.85) [string-blue,left] {\scriptsize $\overline{Y}$}; 
        \node at (-4.2,-0.75)  [string-blue,left] {\scriptsize $X$}; 
    \end{tikzpicture}
    \caption{Cylinders over $M_\Gamma^{\underline{X}}$ and $M_{\Gamma'}^{\underline{Y}}$ for $\Gamma$ and $\Gamma'$ the source and target of the $1$-morphism in $\WS(\DD_\calc)$ from \eqref{eq:worldsheetmorphism3d}.}
    \label{fig:cylindersforconnmani}
\end{figure}
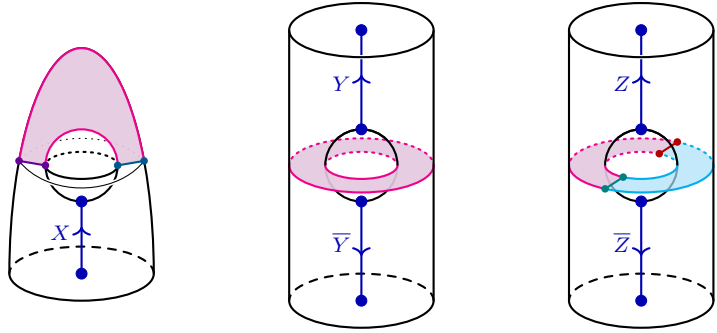

Using this parametrisation we can now glue these cylinders along their outgoing boundary to $M_\Sigma$ and obtain
\begin{align}
    M_\Sigma^{\underline{X},\underline{Y}} := (M_\Gamma^{\underline{X}}\times [-1,1]) \sqcup_{M_\Gamma} M_\Sigma
  \sqcup_{M_{\Gamma'}} (M_{\Gamma'}^{\underline{Y}}\times [-1,1]).
\end{align}
By construction the boundary of $M_\Sigma^{\underline{X},\underline{Y}}$ is given by
\begin{align}
    \partial M_\Sigma^{\underline{X},\underline{Y}} \cong -M_{\Gamma}^{\underline{X}} \sqcup D_{\Gamma}^{\underline{X}} \sqcup_{\widehat{\Gamma}} \widehat{\Sigma} \sqcup_{\widehat{\Gamma}'} D_{\Gamma'}^{\underline{Y}} \sqcup M_{\Gamma'}^{\underline{Y}}
\end{align}
where $D_{\Gamma}^{\underline{X}}$ and $D_{\Gamma'}^{\underline{Y}}$ are the discs with marked points which were glued into $M_\Gamma$ and $M_{\Gamma'}$. Technically we would need to choose a smoothing of $M_\Sigma^{\underline{X},\underline{Y}}$, however, as mentioned in the conventions, we will instead work in the topological category here.

Let $\widehat{\Sigma}^{\underline{X},\underline{Y}} := D_{\Gamma}^{\underline{X}} \sqcup_{\widehat{\Gamma}} \widehat{\Sigma} \sqcup_{\widehat{\Gamma}'} D_{\Gamma'}^{\underline{Y}}$ be the object in $\borddef$ obtained by gluing discs into the boundary components of $\widehat{\Sigma}$. Using this we can turn $M_{\Sigma}^{\underline{X},\underline{Y}}$ into a morphism
\begin{align}
    M_{\Gamma'}^{\underline{Y}} \sqcup -M_\Gamma^{\underline{X}} \to \widehat{\Sigma}^{\underline{X},\underline{Y}}
\end{align}
in $\borddef$. Note here that the difference in orientation between $M_\Gamma^{\underline{X}} $ and $M_{\Gamma'}^{\underline{Y}}$ is a direct result from $\Gamma$ being the incoming boundary and $\Gamma'$ being the outgoing boundary of $\Sigma$. For an illustration of this whole procedure see \Cref{fig:conman3}, in particular the difference in the orientation of the line defects is induced by the difference in orientations of the incoming boundary. In the following we will call $M_{\Sigma}^{\underline{X},\underline{Y}}$ the \emph{connecting bordism} of $\frakS$.

Finally, observe that $\widehat{\Sigma}^{\underline{X},\underline{Y}}\in \mathrm{Bord}_{3,2}^{\chi}(\calc)$ is precisely the marked surface discussed in the description of the full modular functor from \Cref{subsec:modfunc} and since $\rmZ_\calc ( \widehat{\Sigma}^{\underline{X},\underline{Y}}) \cong \widehat{\rmV}_\calc ( \widehat{\Sigma}^{\underline{X},\underline{Y}})$ the functor 
\begin{equation}
\begin{aligned}
   \mathrm{bl}_Z(\widehat{\Sigma}) \colon (\calc^{\mathrm{op}})^{\times\pi_0(\widehat{\Gamma})} \times \calc^{\times \pi_0(\widehat{\Gamma'})} &\to \vect \\
     (\underline{X},\underline{Y}) &\mapsto \rmZ_\calc ( \widehat{\Sigma}^{\underline{X},\underline{Y}})
\end{aligned}
\end{equation}
lifts to the full modular functor $\Bl_\calc(\frakS)\colon (\calc^\op)^{\boxtimes \pi_0(\widehat{\Gamma})} \boxtimes (\calc)^{\boxtimes \pi_0(\widehat{\Gamma}')} \to \vect$.

It is straightforward to check that the family of linear maps
\begin{align}
    \mathrm{cor}_\Sigma^{\underline{X},\underline{Y}} \colon \rmZ_\calc (M_{\Gamma'}^{\underline{Y}} \sqcup -M_\Gamma^{\underline{X}}) \to \rmZ_\calc ( \widehat{\Sigma}^{\underline{X},\underline{Y}})
\end{align}
defines a natural transformation
\begin{align}
    \mathrm{cor}_\frakS \colon \mathrm{cor}_{\Gamma'}\otimes_\kk \mathrm{cor}_{\Gamma}^{\dagger} \Rightarrow \mathrm{bl}_Z(\widehat{\Sigma})
\end{align}
of functors from $(\calc)^{\times\pi_0(\widehat{\Gamma}')} \times (\calc^{\op})^{\times\pi_0(\widehat{\Gamma})}$ to $\vect$. 

\begin{definition}[$\Cor$ on morphisms]
    For $\frakS \colon \frakC \to \frakC$ a $1$-morphism in $\WS(\DD_\calc)$ we define
    \begin{equation}
            \Cor_\frakS \colon \mathrm{Cor}_{\frakC}\otimes_\kk \mathrm{Cor}_{\frakC}^{\dagger} \Rightarrow \mathrm{Bl}_\calc(\frakS) 
    \end{equation}
    to be the natural transformation induced by $\mathrm{cor}_\Sigma$ on the Deligne tensor product $(\calc)^{\boxtimes \pi_0(\widehat{\Gamma}')} \boxtimes (\calc^\op)^{\boxtimes \pi_0(\widehat{\Gamma})}$. 
    
\end{definition}
\begin{remark}\label{rem:anomalyconnectingbordism}
    As explained in \Cref{sec:3ddeftft} the bordism category $\borddef$ contains extra data needed to cancel a gluing anomaly of the TFT $\rmZ_\calc$. This data can be included in our construction in precisely the same way as in \autocite[Sec.\,3.1]{FRSII}. This is possible because for any object $\frakC\in\WS(\DD_\calc)$ the connecting manifold $M_\frakC^{\underline{X}}$ is topologically a finite disjoint union of $2$-spheres and thus has trivial first homology. Thus the extra boundary components of our connecting bordism in comparison to the one of \autocite[Sec.\,3.1]{FRSII} do not lead to extra contributions.
\end{remark}

The rest of this chapter will be devoted to proving that $\Cor$ satisfies the conditions of a braided monoidal oplax natural transformation under three more technical assumptions on $\rmZ_\calc$. To this end it will be useful to recall from \Cref{subsec:oplaxnattrafo} that for every object $\frakC \in \WS(\DD_\calc)$ there exists an essentially unique object $\mathbbm{F}_\frakC \in \Bl_\calc(\frakC)$ such that
\begin{align}
     \mathrm{Cor}_{\frakC}(-) \cong \Hom_{\Bl_\calc(\frakC)}(\mathbbm{F}_\frakC,-)
\end{align} 
by left exactness of $ \mathrm{Cor}_{\frakC}(-)$.

\subsection{Monoidality}
Let us first discuss how to obtain the braided monoidal structure on $\Cor$ in the form of $\Pi$ and $\Upsilon$. For $\Upsilon$ we have $M_\varnothing = \varnothing$ which directly leads to $\Cor_\varnothing \cong \id_\Vect$ as desired. For $\Pi_{\frakC,\frakC'}$ first note that $\mathrm{cor}_{\Gamma\sqcup \Gamma'}(-,\sim) \cong \mathrm{cor}_{\Gamma}(-) \otimes_\kk  \mathrm{cor}_{\Gamma'}(\sim)$ by monoidality of $\rmZ_\calc$ and thus 
\begin{equation}  
\begin{aligned}
    \mathrm{Cor}_{\frakC\sqcup \frakC'}(-\boxtimes\sim) &\cong \mathrm{Cor}_{\frakC}(-) \otimes_\kk \mathrm{Cor}_{\frakC'}(\sim) \\
    &\cong \Hom_{\Bl_\calc(\frakC)}(\mathbbm{F}_\frakC,-) \otimes_\kk \Hom_{\Bl_\calc(\frakC')}(\mathbbm{F}_{\frakC'},\sim).
\end{aligned}
\end{equation}
On the other hand a direct computation using the Yoneda lemma leads to
\begin{align}
    (\mathrm{Cor}_{\frakC} \boxtimes \id_{\Bl_\calc(\frakC')}) \diamond (\id_{\Delta_\kk(\frakC)} \boxtimes \mathrm{Cor}_{\frakC'})(- \boxtimes \sim) \cong \Hom_{\Bl_{\calc}(\frakC)\boxtimes\Bl_{\calc}(\frakC')}(\mathbbm{F}_{\frakC} \boxtimes \mathbbm{F}_{\frakC'},- \boxtimes \sim).
\end{align}
Combining this with the natural isomorphism
\begin{align}
       \Hom_{\Bl_\calc(\frakC) \boxtimes \Bl_\calc(\frakC')} (-\boxtimes -, -\boxtimes -)\cong \Hom_{\Bl_\calc(\frakC)}(-,-) \otimes_\kk \Hom_{\Bl_\calc(\frakC')}(-,-) 
\end{align}
coming from the universal property of the Deligne product \autocite[Prop.\,1.11.2]{EGNO} we finally get the desired isomorphism
\begin{align}
    \mathrm{Cor}_{\frakC\sqcup \frakC'}(-\boxtimes\sim) \cong (\mathrm{Cor}_{\frakC} \boxtimes \id_{\Bl_\calc(\frakC')}) \diamond (\id_{\Delta_\kk(\frakC)} \boxtimes \mathrm{Cor}_{\frakC'})(- \boxtimes \sim).
\end{align}
For the modification axiom of $\Pi_{\frakC,\frakC'}$, see e.g.\ \autocite[Def.\,4.4.2]{JY20twodcategories}, we observe that for $\frakS \colon \frakC \to \frakC'$ and $\widetilde{\frakS} \colon \widetilde{\frakC} \to \widetilde{\frakC'}$ the connecting manifold of their disjoint union $M_{\Sigma \sqcup \widetilde{\Sigma}}$ is the disjoint union of the connecting manifolds $M_{\Sigma} \sqcup M_{\widetilde{\Sigma}}$. Since $\Pi_{\frakC,\frakC'}$ is constructed using the monoidal structure of $\rmZ_\calc$ this implies the required relation between $\Cor_{\frakS \sqcup \widetilde{\frakS}}$ and $\Cor_{\frakS} \sqcup \Cor_{\widetilde{\frakS}}$.

\subsection{MCG covariance / 2-morphism naturality}
Next let us discuss $2$-morphism naturality in the form of commutativity of Diagram \ref{diag:2natural}. 
Let $f \colon \frakS \Rightarrow \frakS'$ be a $2$-morphism between $1$-morphisms $\frakS,\frakS' \colon \frakC \to \frakC'$ in $\WS(\DD_\calc)$. Since Diagram \ref{diag:2natural} consists of natural transformations we can do this component wise. Now all the functors and natural transformations in this diagram are specified via their action on pure tensors in the Deligne product. These actions are in turn induced from $\rmZ_\calc$ being evaluated on various manifolds which leads us to study the diagram
\begin{align}
    \begin{tikzcd}[ampersand replacement=\&]
	\& {\rmZ_\calc(\widehat{\Sigma}^{\underline{X},\underline{Y}})} \\
	{\rmZ_\calc (M_{\Gamma'}^{\underline{Y}} ) \otimes_\mathbb{k} \rmZ_\calc( -M_\Gamma^{\underline{X}})} \\
	\& {\rmZ_\calc(\widehat{\Sigma}'^{\underline{X},\underline{Y}})}
	\arrow["{\rmZ_\calc(\widehat{C_{f}}^{\underline{X},\underline{Y}})}", from=1-2, to=3-2]
	\arrow["{\rmZ_\calc(M_{\Sigma}^{\underline{X},\underline{Y}})}", from=2-1, to=1-2]
	\arrow["{\rmZ_\calc(M_{\Sigma'}^{\underline{X},\underline{Y}})}"', from=2-1, to=3-2]
\end{tikzcd}
\end{align}
in $\vect$ with $\underline{X}\in \calc^{\times\pi_0(\widehat{\Gamma})}$ and $\underline{Y}\in \calc^{\times\pi_0(\widehat{\Gamma}')}$. Here $\widehat{C_f}^{\underline{X},\underline{Y}}$ is the mapping cylinder bordism induced from the mapping class $f$ as in \autocite[Sec.\,4.3]{HR2024modfunc}. Using functoriality of $\rmZ_\calc$ we can instead check the diagram
\begin{align}
    \begin{tikzcd}[ampersand replacement=\&]
	\& {\widehat{\Sigma}^{\underline{X},\underline{Y}}} \\
	{M_{\Gamma'}^{\underline{Y}} \sqcup -M_\Gamma^{\underline{X}}} \\
	\& {\widehat{\Sigma}'^{\underline{X},\underline{Y}}}
	\arrow["{\widehat{C_{f}}^{\underline{X},\underline{Y}}}", from=1-2, to=3-2]
	\arrow["{M_{\Sigma}^{\underline{X},\underline{Y}}}", from=2-1, to=1-2]
	\arrow["{M_{\Sigma'}^{\underline{X},\underline{Y}}}"', from=2-1, to=3-2]
    \end{tikzcd}
\end{align}
in $\borddef$ for commutativity. In other words we need to find an isomorphism of defect $3$-bordisms $M_{\Sigma'}^{\underline{X},\underline{Y}} \cong \widehat{C_{f}}^{\underline{X},\underline{Y}} \sqcup_{\widehat{\Sigma}^{\underline{X},\underline{Y}}} M_{\Sigma}^{\underline{X},\underline{Y}}$. 

To this end first note that $f \times \id_{[-1,1]} \colon \Sigma \times [-1,1] \to \Sigma' \times [-1,1]$ descends to the quotients $\widetilde{f} \colon M_\Sigma \to M_{\Sigma'}$ because $f$ is a morphism of world sheets and thus compatible with the free boundaries. Moreover, since $f$ commutes with the boundary parametrisations we can extend $\widetilde{f}$ to $F \colon M_{\Sigma}^{\underline{X},\underline{Y}} \to M_{\Sigma'}^{\underline{X},\underline{Y}}$ which is an isomorphism of defect $3$-manifolds because $f$ takes the surface defect $\frakS$ in $M_\Sigma^{\underline{X},\underline{Y}}$ to $\frakS'$ in $M_{\Sigma'}^{\underline{X},\underline{Y}}$. Now finally note that $\widehat{C_{f}}^{\underline{X},\underline{Y}} \sqcup_{\widehat{\Sigma}^{\underline{X},\underline{Y}}} M_{\Sigma}^{\underline{X},\underline{Y}}$ differs from $ M_{\Sigma}^{\underline{X},\underline{Y}}$ only in the outgoing boundary parametrisation and $F$ is compatible with the one of $\widehat{C_{f}}^{\underline{X},\underline{Y}} \sqcup_{\widehat{\Sigma}^{\underline{X},\underline{Y}}} M_{\Sigma}^{\underline{X},\underline{Y}}$. Thus $F\colon \widehat{C_{f}}^{\underline{X},\underline{Y}} \sqcup_{\widehat{\Sigma}^{\underline{X},\underline{Y}}} M_{\Sigma}^{\underline{X},\underline{Y}} \to  M_{\Sigma'}^{\underline{X},\underline{Y}}$ gives the desired isomorphism of defect $3$-bordisms.

\subsection{Factorisation / oplax naturality}\label{sec:factorisation}
We now want to study oplax naturality in the form of Diagram \ref{diag:horinatural}, or in other words the behaviour of $\Cor$ under the gluing of world sheets. Let $\frakS_1 \colon \frakC_1 \to \frakC$ and $\frakS_2 \colon \frakC \to \frakC_2$ be $1$-morphisms in $\WS(\DD_\calc)$. 
In analogy to \autocite[Sec.\,4.4]{HR2024modfunc} we can restrict our attention to the case where $|\pi_0(\Gamma)| = 1$ since gluing is a local procedure. This leaves us with two distinct cases to consider: $\frakC$ is either a defect interval or a defect circle. We will refer to these cases as \emph{boundary} and \emph{bulk factorisation}, respectively. 

As in the previous section we want to reduce Diagram \ref{diag:horinatural} to a corresponding diagram in $\borddef$ by considering it component wise and then use functoriality of $\rmZ_\calc$. However, this is less straightforward here because Diagram \ref{diag:horinatural} contains two morphisms whose components are not necessarily in the image of $\rmZ_\calc$: The composition $\Cor_{\frakS_2}\diamond \Cor_{\frakS_1}$ and the unit $\eta_{\Cor_{\frakC}} \colon \id_{\vect} \Rightarrow \Cor_{\frakC}^\dagger \diamond \Cor_{\frakC}$ of the adjunction $\Cor_{\frakC} \dashv \Cor_{\frakC}^\dagger$ . 

Let us start with the composition $\Cor_{\frakS_2}\diamond \Cor_{\frakS_1}$. For this recall that the horizontal composition of $\Cor_{\frakS_1} \colon \Cor_{\frakC}\otimes_\kk \Cor_{\frakC_1}^{\dagger} \Rightarrow \mathrm{Bl}_\calc(\frakS_1) $ and $\Cor_{\frakS_2} \colon \Cor_{\frakC_2}\otimes_\kk \mathrm{Cor}_{\frakC}^{\dagger} \Rightarrow \mathrm{Bl}_\calc(\frakS_2)$ is defined via the universal property of the coend as the unique natural transformation which makes the following diagram commute:
\begin{align}\label{diag:corgluing}
    \begin{tikzcd}[ampersand replacement=\&]
	{\mathrm{Cor}_{\frakC_2}(\underline{Z})\otimes_\kk \mathrm{Cor}_{\frakC}^\dagger(\underline{Y}) \otimes_\kk \mathrm{Cor}_{\frakC}(\underline{Y})\otimes_\kk \mathrm{Cor}_{\frakC_1}^\dagger}(\underline{X}) \&\& {\mathrm{Bl}(\frakS_2)(\underline{Y},\underline{Z})\otimes_\kk \mathrm{Bl}(\frakS_1)(\underline{X},\underline{Y})} \\
	{\mathrm{Cor}_{\frakC_2}(\underline{Z})\otimes_\kk \mathrm{Cor}_{\frakC}^\dagger\diamond \mathrm{Cor}_{\frakC}\otimes_\kk \mathrm{Cor}_{\frakC_1}^\dagger}(\underline{X}) \&\& {(\mathrm{Bl}_\calc(\frakS_2)\diamond \mathrm{Bl}_\calc(\frakS_1))(\underline{X},\underline{Z})}
	\arrow["{\mathrm{Cor}_{\frakS_2}^{\underline{Y},\underline{Z}}\otimes_\kk\mathrm{Cor}_{\frakS_1}^{\underline{X},\underline{Y}}}", from=1-1, to=1-3]
	\arrow["{\id\otimes_\kk\iota_{\underline{Y}}\otimes_\kk \id}"', from=1-1, to=2-1]
	\arrow["{I_{\underline{Y}}}^{(\underline{X},\underline{Z})}", from=1-3, to=2-3]
	\arrow["{\exists!}", dashed, from=2-1, to=2-3]
\end{tikzcd}
\end{align}
Here $\underline{X}\in \calc^{\times\pi_0(\widehat{\Gamma_1})}$, $\underline{Y}\in \calc^{\times\pi_0(\widehat{\Gamma})}$, $\underline{Z}\in \calc^{\times\pi_0(\widehat{\Gamma_2})}$, and $\iota_{\underline{Y}}$ and $I_{\underline{Y}}^{(\underline{X},\underline{Z})}$ are the universal dinatural transformations. From \autocite[Prop.\,4.8]{HR2024modfunc} we know that the components of $I_{\underline{Y}}$ are obtained by evaluating the $3$d TFT $\widehat{\rmV}_\calc$ on (the double of) the gluing bordisms
\begin{equation}
    G_{\underline{Y}}^{(\underline{X},\underline{Z})} \colon \widehat{\Sigma_2}^{(\underline{Y},\underline{Z})} \sqcup\widehat{\Sigma_1}^{(\underline{X},\underline{Y})} \to (\widehat{\Sigma_2} \sqcup_{\widehat{\Gamma}} \widehat{\Sigma_1})^{(\underline{X},\underline{Z})} 
\end{equation}
defined in \autocite[Sec.\,4.4]{HR2024modfunc}. More concretely, the underlying manifold of $G_{\underline{Y}}^{(\underline{X},\underline{Z})}$ is given by
\begin{equation}
    (\widehat{\Sigma_2}^{(\underline{Y},\underline{Z})} \sqcup\widehat{\Sigma_1}^{(\underline{X},\underline{Y})}) \times I /\sim
\end{equation}
where $\sim$ identifies the discs glued into the components of $\widehat{\Gamma}$. For $\Gamma$ an interval a local visualisation of $G_{\underline{Y}}^{(\underline{X},\underline{Z})}$ is given in \Cref{fig:gluingbordloc}. 
\begin{figure}[t]
    \centering
    \includegraphics[width=0.4\linewidth]{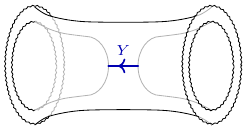}
    \caption{Local structure of the gluing bordism $G_{\underline{Y}}^{(\underline{X},\underline{Z})}$ for $\frakC$ a defect interval. The incoming boundary surface of this bordism is drawn inside and the blue $Y$-labelled ribbon connects it to itself. The curly lines indicate interfaces at which the displayed piece is connected to other parts of the bordisms.}
    \label{fig:gluingbordloc}
\end{figure}
For $\frakC$ a defect circle one needs to use the double of the gluing bordism or equivalently the composition of the gluing bordisms for each component of $\widehat{\Gamma}$. In both cases $G_{\underline{Y}}^{(\underline{X},\underline{Z})}$ is a cylinder outside of the depicted local region. This allows us to extend its definition to the defect bordism category $\borddef$. If the dinatural transformation $\iota_{\underline{Y}}$ is also induced by a family of gluing bordisms
\begin{align}\label{eq:gluebord1}
    \widetilde{G}_{\underline{Y}} \colon -M_{\Gamma}^{\underline{Y}} \sqcup M_{\Gamma}^{\underline{Y}} \to -M_{\Gamma} \sqcup_{\widehat{\Gamma}} M_{\Gamma},
\end{align}
we can consider the corresponding diagram in $\borddef$:
\begin{align}\label{diag:gluebordcat}
    \begin{tikzcd}[ampersand replacement=\&]
	{M_{\Gamma_2}^{\underline{Z}} \sqcup -M_{\Gamma}^{\underline{Y}}\sqcup M_{\Gamma}^{\underline{Y}}\sqcup -M_{\Gamma_1}^{\underline{X}}} \&\& {\widehat{\Sigma_2}^{\underline{Y},\underline{Z}} \sqcup \widehat{\Sigma_1}^{\underline{X},\underline{Y}}} \\
	\\
	{M_{\Gamma_2}^{\underline{Z}} \sqcup -M_{\Gamma}\sqcup_{\widehat{\Gamma}} M_{\Gamma}\sqcup -M_{\Gamma_1}^{\underline{X}}} \&\& {(\widehat{\Sigma_2}\sqcup_{\widehat{\Gamma}}\widehat{\Sigma_1})^{\underline{X},\underline{Z}}}
	\arrow["{M_{\Sigma_2}^{\underline{Y},\underline{Z}}\sqcup M_{\Sigma_1}^{\underline{X},\underline{Y}}}", from=1-1, to=1-3]
	\arrow["\id \sqcup \widetilde{G}_{\underline{Y}} \sqcup \id"', from=1-1, to=3-1]
	\arrow["{G_{\underline{Y}}^{(\underline{X},\underline{Z})}}", from=1-3, to=3-3]
	\arrow[dashed, from=3-1, to=3-3]
\end{tikzcd}
\end{align}
Moreover, let us further assume that we can find a bordism for the dashed arrow then functoriality of $\rmZ_\calc$ and the uniqueness coming from the universal property will guarantee that this bordism gets sent to the composition $\Cor_{\frakS_2}\diamond \Cor_{\frakS_1}$. This motivates the next algebraic condition on $\rmZ_\calc$:

\begin{assumption}\label{ass:3}
    For $\frakC \in \WS(\DD_\calc)$ a defect interval or circle, the universal dinatural transformation $\iota_{\underline{Y}} \colon \Cor_{\frakC}(\underline{Y})^\dagger \otimes_\kk \Cor_{\frakC}(\underline{Y}) \to \Cor_{\frakC}^\dagger \diamond \Cor_{\frakC}$ is induced by evaluating $\rmZ_\calc$ on the family of gluing bordisms
    \begin{align}
    \widetilde{G}_{\underline{Y}} \colon -M_{\Gamma}^{\underline{Y}} \sqcup M_{\Gamma}^{\underline{Y}} \to -M_{\Gamma} \sqcup_{\widehat{\Gamma}} M_{\Gamma}
    \end{align}
    in $\borddef$.
\end{assumption}

Now for the second problem recall from \Cref{lem:profadjunc} that $\eta_{\Cor_{\frakC}}$ is given by the action of the constant functor $\mathbbm{F}_{\frakC} \colon \vect \to \Bl_\calc(\frakC)$ on morphisms, i.e.\ by the linear map sending $1\in\kk$ to $\mathrm{id}_{\mathbbm{F}_{\frakC}}$ in $\Hom_{\Bl_\calc(\frakC)}(\mathbbm{F}_{\frakC},\mathbbm{F}_{\frakC}) \cong \Cor_{\frakC}^\dagger \diamond \Cor_{\frakC}$. Next note that the vector space $\Cor_{\frakC}^\dagger \diamond \Cor_{\frakC}$ is isomorphic to $\rmZ_\calc(-M_{\Gamma} \sqcup_{\widehat{\Gamma}} M_{\Gamma})$ by \Cref{ass:3}. We thus want to find a bordism $\eta_{M_\Gamma} \colon \varnothing \to -M_{\Gamma} \sqcup_{\widehat{\Gamma}} M_{\Gamma}$ such that $\rmZ_\calc(\eta_{M_\Gamma})$ corresponds to 
$\mathrm{id}_{\mathbbm{F}_{\frakC}}$ under the isomorphism $\Hom_{\Bl_\calc(\frakC)}(\mathbbm{F}_{\frakC},\mathbbm{F}_{\frakC}) \cong \Cor_{\frakC}^\dagger \diamond \Cor_{\frakC}$. Moreover by monoidality it suffices to show this for $\frakC$ a defect interval or a defect circle. This naturally leads us to the next condition on $\rmZ_\calc$:
\begin{assumption}\label{ass:2}
For $\frakC \in \WS(\DD_\calc)$ a defect interval or circle, the unit 
\begin{equation}
    \eta_{\Cor_{\frakC}} \colon \id_{\vect} \Rightarrow \Cor_{\frakC}^\dagger \diamond \Cor_{\frakC}
\end{equation}
of the adjunction $\Cor_{\frakC} \dashv \Cor_{\frakC}^\dagger$ corresponds to the linear map $\rmZ_\calc(\eta_{M_\Gamma})$, with 
\begin{equation}
    \eta_{M_\Gamma} \colon \varnothing \to -M_{\Gamma} \sqcup_{\widehat{\Gamma}} M_{\Gamma}
\end{equation}
as in \Cref{fig:unitadj}, under the isomorphism $\Cor_{\frakC}^\dagger \diamond \Cor_{\frakC} \cong \rmZ_\calc(-M_{\Gamma} \sqcup_{\widehat{\Gamma}} M_{\Gamma})$. In other words we require that $\rmZ_\calc(\eta_{M_\Gamma})\in \rmZ_\calc(-M_{\Gamma} \sqcup_{\widehat{\Gamma}} M_{\Gamma})$ corresponds to $\mathrm{id}_{\mathbbm{F}_{\frakC}} \in \Hom_{\Bl_\calc(\frakC)}(\mathbbm{F}_{\frakC},\mathbbm{F}_{\frakC})$.
\end{assumption}

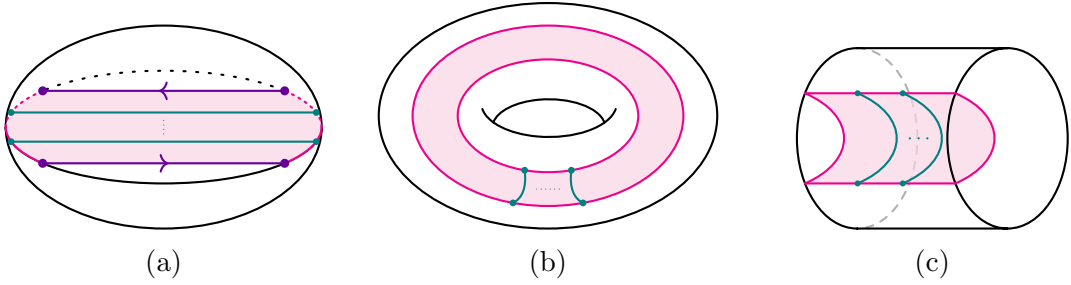
\begin{figure}[t]
    \centering
    \begin{subfigure}{0.3\textwidth}
         \centering
             \begin{tikzpicture}[scale=1.5]
        \draw[thick] (0,0) ellipse (1.4 and 0.9)
        (1.4,0) arc (0:-180:1.4 and 0.5);
        \fill [magenta!20,opacity=0.7] (140:1.4 and 0.5) -- (220:1.4 and 0.5) -- (320:1.4 and 0.5) -- (40:1.4 and 0.5) -- cycle
        (140:1.4 and 0.5) arc (140:220:1.4 and 0.5)
        (320:1.4 and 0.5) arc (320:400:1.4 and 0.5);
        \draw[thick,loosely dotted,line cap=round] (40:1.4 and 0.5) arc (40:140:1.4 and 0.5);
        \begin{scope}[very thick,decoration={
                      markings,
                      mark=at position 0.52 with {\arrow{>}}}
                     ]  
        \draw[string-violet,thick,postaction={decorate}](220:1.4 and 0.5)  -- (320:1.4 and 0.5);
        \draw[string-violet,thick,postaction={decorate}] (40:1.4 and 0.5) -- (140:1.4 and 0.5);
        \draw[string-green,thick,line cap=round] (165:1.4 and 0.5) -- (15:1.4 and 0.5)
        (195:1.4 and 0.5) -- (-15:1.4 and 0.5);
        \end{scope}
        \draw[magenta,thick,line cap=round] (-1.4,0) arc (180:220:1.4 and 0.5)
        (1.4,0) arc (0:-40:1.4 and 0.5);
        \draw[magenta,thick, dotted,line cap=round] (-1.4,0) arc (180:140:1.4 and 0.5)
        (1.4,0) arc (0:40:1.4 and 0.5);
        \fill[string-violet] (140:1.4 and 0.5) circle (0.04)
        (40:1.4 and 0.5) circle (0.04);
        \fill[string-violet] (320:1.4 and 0.5) circle (0.04)
        (220:1.4 and 0.5) circle (0.04);
        \draw[string-green,dash pattern=on 0pt off 4\pgflinewidth,line cap=round] (0,0.06) -- (0,-0.08);
        \fill[string-green] (165:1.4 and 0.5) circle (0.03)
        (15:1.4 and 0.5) circle (0.03)
        (195:1.4 and 0.5) circle (0.03) 
        (-15:1.4 and 0.5) circle (0.03);
    \end{tikzpicture}
        \caption{}\label{fig:intervalunit}
    \end{subfigure}
    \hspace{0.3cm}
    \begin{subfigure}{0.3\textwidth}
              \centering
             \begin{tikzpicture}[scale=1.5]
        \fill[magenta!20, opacity=0.7, even odd rule] 
        (0,0) ellipse (1.2 and 0.8)
        (0,0) ellipse (0.8 and 0.5);
        \draw[magenta,thick] (0,0) ellipse (1.2 and 0.8)
        (0,0) ellipse (0.8 and 0.5);
        \draw[thick,line cap = round] (0,0) ellipse (1.5 and 1)
           (0.6,0.06) arc (-10:-170:0.6 and 0.3)
           (0.5,-0.06) arc (10:170:0.5 and 0.25);
        \draw[string-green,thick,line cap=round] (255:1.2 and 0.8) .. controls (258:1.1 and 0.75) and (258:0.9 and 0.65) .. (255:0.8 and 0.5)
        (285:1.2 and 0.8) .. controls (282:1.1 and 0.75) and (282:0.9 and 0.65) .. (285:0.8 and 0.5);
        \draw[string-green,dash pattern=on 0pt off 4\pgflinewidth,line cap=round] (0.1,-0.65) -- (-0.1,-0.65);
        \fill[string-green] (285:0.8 and 0.5) circle (0.03)
        (285:1.2 and 0.8) circle (0.03)
        (255:1.2 and 0.8) circle (0.03)
        (255:0.8 and 0.5) circle (0.03);
    \end{tikzpicture}
            \caption{}\label{fig:circleunit}
    \end{subfigure}
    \hspace{0.3cm}
    \begin{subfigure}{0.3\textwidth}
              \centering
            \begin{tikzpicture}
        \draw[thick,dashed, line cap=round] (0,1.2) arc (90:-90:0.8 and 1.2);
        \fill[white,opacity=0.7] (0,-1.2) rectangle (2,1.2);
        \fill[magenta!20, opacity=0.7] (150:0.8 and 1.2) -- 
        ([shift={(2,0)}]150:0.8 and 1.2) .. controls (2,0.3) and (2,-0.3) .. ([shift={(2,0)}]210:0.8 and 1.2) -- 
        (210:0.8 and 1.2) .. controls (0,-0.3) and (0,0.3) .. (150:0.8 and 1.2);
        \draw[thick, line cap=round] (0,1.2) arc (90:270:0.8 and 1.2)
        (2,0) ellipse (0.8 and 1.2)
        (0,1.2) -- (2,1.2)
        (0,-1.2) -- (2,-1.2);
        \draw[thick,magenta,line cap=round] (150:0.8 and 1.2) -- ([shift={(2,0)}]150:0.8 and 1.2)
        (210:0.8 and 1.2) -- ([shift={(2,0)}]210:0.8 and 1.2);
        \draw[thick,magenta,line cap=round](150:0.8 and 1.2) .. controls (0,0.3) and (0,-0.3) .. (210:0.8 and 1.2)
         ([shift={(2,0)}]150:0.8 and 1.2) .. controls (2,0.3) and (2,-0.3) .. ([shift={(2,0)}]210:0.8 and 1.2);
        \draw[thick, string-green]  ([shift={(0.7,0)}]150:0.8 and 1.2) .. controls (0.7,0.3) and (0.7,-0.3) .. ([shift={(0.7,0)}]210:0.8 and 1.2)
        ([shift={(1.3,0)}]150:0.8 and 1.2) .. controls (1.3,0.3) and (1.3,-0.3) .. ([shift={(1.3,0)}]210:0.8 and 1.2);
        \fill[string-green] ([shift={(0.7,0)}]150:0.8 and 1.2) circle (0.04)
        ([shift={(0.7,0)}]210:0.8 and 1.2) circle (0.04)
        ([shift={(1.3,0)}]150:0.8 and 1.2) circle (0.04)
        ([shift={(1.3,0)}]210:0.8 and 1.2) circle (0.04);
        \draw[string-green,dash pattern=on 0pt off 4\pgflinewidth,line cap=round,thick] (0.7,0) -- (1,0);
    \end{tikzpicture}
            \caption{}
    \end{subfigure}
    \caption{The unit morphism $\eta_{M_\Gamma} \colon \varnothing \to -M_{\Gamma'} \sqcup_{\widehat{\Gamma}'} M_{\Gamma'}$ for $\frakC$ a defect interval (a) and a defect circle (b). As well as a cylinder (c) cut out of (b) to illustrate the stratification.}
    \label{fig:unitadj}
\end{figure}

\begin{remark}\label{rem:unitbord321}
    The bordisms depicted in \Cref{fig:unitadj} as candidates for the unit $\eta_{\Cor_{\frakC}}$ of the adjunction $\Cor_{\frakC} \dashv \Cor_{\frakC}^\dagger$ can be informally motivated from once-extended $3$d TFTs. Let us make this a bit more precise for $\Gamma'$ a defect interval. The underlying manifold of $M_{\Gamma}$ is a disc which we will view as the cup $1$-morphism $\varnothing \to S^1$ in the once-extended $3$-dimensional bordism $2$-category $\bord_{3,2,1}$. Now recall from the conjectured generators and relations description of $\bord_{3,2,1}$ discussed in \autocite[Sec.\,3]{BDSV15extTFT} that this cup is one of the $1$-morphism generators of $\bord_{3,2,1}$ and its two-sided adjoint is given by a cap $S^1 \to \varnothing$. Moreover, the unit of this adjunction is given by a $3$-ball, referred to as $\nu$ in \autocite{BDSV15extTFT}. But this is precisely the $3$-manifold underlying \Cref{fig:unitadj} (a). For $\Gamma$ a defect circle the idea is the same but one has to consider the composition of some of the generators because the underlying manifold of $M_{\Gamma}$ is the composition of the cup generator with the up-side-down pair of pants generator. We will not discuss this point further here.
\end{remark}

Under Assumptions \ref{ass:3} and \ref{ass:2} we are thus lead to study the following diagram in $\borddef$:
\begin{align}\label{diag:oplaxbordcat}
    \begin{tikzcd}[ampersand replacement=\&]
	{M_{\Gamma_2}^{\underline{Z}} \sqcup -M_{\Gamma}^{\underline{Y}}\sqcup M_{\Gamma}^{\underline{Y}}\sqcup -M_{\Gamma_1}^{\underline{X}}} \&\& {\widehat{\Sigma_2}^{\underline{Y},\underline{Z}} \sqcup \widehat{\Sigma_1}^{\underline{X},\underline{Y}}} \\
	\\
	{M_{\Gamma_2}^{\underline{Z}} \sqcup -M_{\Gamma}\sqcup_{\widehat{\Gamma}} M_{\Gamma'}\sqcup -M_{\Gamma_1}^{\underline{X}}} \&\& {(\widehat{\Sigma_2}\sqcup_{\widehat{\Gamma}}\widehat{\Sigma_1})^{\underline{X},\underline{Z}}} \\
	\\
	\& {M_{\Gamma_2}^{\underline{Z}} \sqcup -M_{\Gamma_1}^{\underline{X}}}
	\arrow["{M_{\Sigma_2}^{\underline{Y},\underline{Z}}\sqcup M_{\Sigma_1}^{\underline{X},\underline{Y}}}", from=1-1, to=1-3]
	\arrow["{\id \sqcup \widetilde{G}_{\underline{Y}} \sqcup \id}"', from=1-1, to=3-1]
	\arrow["{G_{\underline{Y}}^{(\underline{X},\underline{Z})}}", from=1-3, to=3-3]
	\arrow[dashed, from=3-1, to=3-3]
	\arrow["{\mathrm{id}\sqcup \eta_{M_{\Gamma}} \sqcup \mathrm{id}}", from=5-2, to=3-1]
	\arrow["{M_{\Sigma_2\sqcup_{\Gamma} \Sigma_1}^{\underline{X},\underline{Z}}}"', from=5-2, to=3-3]
\end{tikzcd}
\end{align}
We will now show that this diagram commutes by studying the two cases mentioned above explicitly.

\subsubsection{Boundary factorisation}
\begin{figure}[p!]
\hspace*{-1.5cm}\begin{tikzcd}[ampersand replacement=\&]
	{M_{\Gamma_2}^{\underline{Z}} \sqcup -M_{\Gamma}^{\underline{Y}}\sqcup M_{\Gamma}^{\underline{Y}}\sqcup -M_{\Gamma_1}^{\underline{X}}} \&\&\&\& {\widehat{\Sigma_2}^{\underline{Y},\underline{Z}} \sqcup \widehat{\Sigma_1}^{\underline{X},\underline{Y}}} \\
	\\
	\\
	\\
	\\
	\\
	{M_{\Gamma_2}^{\underline{Z}} \sqcup -M_{\Gamma}\sqcup_{\widehat{\Gamma}} M_{\Gamma}\sqcup -M_{\Gamma_1}^{\underline{X}}} \&\&\&\& {(\widehat{\Sigma_2}\sqcup_{\widehat{\Gamma}}\widehat{\Sigma_1})^{\underline{X},\underline{Z}}} \\
	\\
	\\
	\\
	\\
	\\
	\& {M_{\Gamma_2}^{\underline{Z}} \sqcup -M_{\Gamma_1}^{\underline{X}}}
	\arrow["{\includegraphics{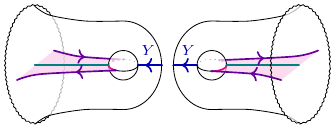}}", from=1-1, to=1-5]
	\arrow["{\includegraphics{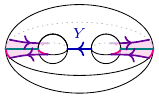}}"', from=1-1, to=7-1]
	\arrow["{\includegraphics{figs/gluingbordismlocal.pdf}}", from=1-5, to=7-5]
	\arrow["{\includegraphics{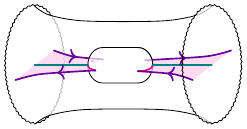}}",from=7-1, to=7-5]
	\arrow["{\includegraphics{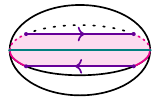}}", from=13-2, to=7-1]
	\arrow["{\includegraphics{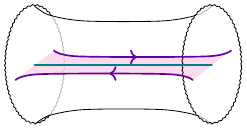}}"', from=13-2, to=7-5]
\end{tikzcd}
    \caption{Diagram \ref{diag:oplaxbordcat} for $\frakC$ a defect interval in a neighbourhood of $\frakC$. For both arrows on the left only the non-identity bordism components are drawn. The curly lines indicate interfaces at which the displayed piece is connected to other parts of the bordisms.}
    \label{fig:boundfact}
\end{figure}
\noindent Let us start with boundary factorisation, i.e.\ the case where $\frakC$ is a defect interval. In contrast to \autocite[Sec.\,4]{FFS12} we cannot assume that $\frakC$ has no $0$-strata in its interior as we do not have an explicit construction of the defect TFT $\rmZ_\calc$. 
Instead we will restrict our attention to the case where $\frakC$ has exactly one $0$-stratum in the interior. The general case can be treated completely analogously but its pictorial presentation is more crowded. 

The key observation to understand Diagram \ref{diag:oplaxbordcat} is to note that all bordisms only differ in a neighbourhood of where the gluing procedure takes place. This is essentially a direct consequence of the locality of gluing, as in \autocite[Sec.\,4.4]{HR2024modfunc}. In particular, it is sufficient to study Diagram \ref{diag:oplaxbordcat} locally. To do this we will only draw the corresponding local segments of the bordisms in analogy to the local illustration of the gluing bordism given in Figure \ref{fig:gluingbordloc}. The resulting local version of Diagram \ref{diag:oplaxbordcat} is illustrated in \Cref{fig:boundfact}. 

Now first note that the illustrated bordism for the lower horizontal line indeed makes the square commute. As discussed above, this implies that this bordism represents the composition  $\Cor_{\frakS_2}\diamond \Cor_{\frakS_1}$. Informally we can obtain this bordism by cutting the left gluing bordism out of the composition of the top horizontal and the right gluing bordism. Finally, it is clear that with this bordism also the triangle commutes.

\subsubsection{Bulk factorisation}
Let us now come to bulk factorisation. As before we will focus on the case where $\frakC$ is a defect circle with exactly one $0$-stratum. To study Diagram \ref{diag:oplaxbordcat} we will again restrict our attention to a local neighbourhood of the gluing process. 

Before we continue with the local version of Diagram \ref{diag:oplaxbordcat} it will be beneficial to introduce the ``wedge presentation'' of defect $2$- and $3$-manifolds building on the one introduced in \autocite[Sect.\,5.1]{FjFRS}. The idea behind this presentation is analogous to the presentation of a torus as a rectangle with opposite sides identified. In the wedge presentation we draw a horizontal disc as a disc sector with an identification of the legs of the sector is implied. As above we will use curly lines to indicate interfaces at which the illustrated piece is connected to other parts of the manifold and we will use dashed lines to indicate further identifications of the legs. Moreover, for clarity we will explicitly highlight the boundary components in this presentation using a light gray colouring. Examples of a cylindrical part of a defect surface, a defect ball, and the defect solid torus from \Cref{fig:circleunit} are illustrated in \Cref{fig:wedgepres}.
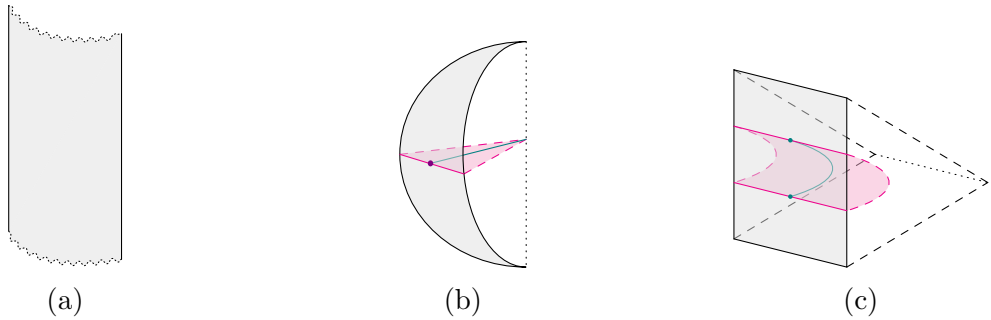
\begin{figure}[t!]
    \centering\begin{subfigure}{0.3\textwidth}\centering
        \begin{tikzpicture}[scale=1.5]
    \fill [lightgray!50, opacity =0.5,decoration={snake,segment length=1.5mm,amplitude=0.3mm}] 
    (-1,-1) -- (-1,1) decorate{
     .. controls (-0.8,0.7) and (-0.4,0.68) .. (0,0.75)};
     \fill [lightgray!50, opacity =0.5,decoration={snake,segment length=1.5mm,amplitude=0.3mm}] 
    (0,0.75) -- (0,-1.25) decorate{
     .. controls (-0.4,-1.32) and (-0.8,-1.3) .. (-1,-1)};
     \draw[line cap=round,densely dotted ,decoration={snake,segment length=1.5mm,amplitude=0.3mm}] 
    decorate{(-1,1)
     .. controls (-0.8,0.7) and (-0.4,0.68) .. (0,0.75)};
     \draw[line cap=round,densely dotted 
     ,decoration={snake,segment length=1.5mm,amplitude=0.3mm}] 
   decorate{
    (0,-1.25) .. controls (-0.4,-1.32) and (-0.8,-1.3) .. (-1,-1)};
    \draw[line cap=round] (-1,1) -- (-1,-1)
    (0,0.75) -- (0,-1.25);

    \end{tikzpicture}
    \caption{}
     \label{fig:wedgecyl}
    \end{subfigure}
    \hfill
    \begin{subfigure}{0.3\textwidth}\centering
        \begin{tikzpicture}[xscale=2.4,yscale=2]
    \draw [dotted,line cap = round] (0,0.75) -- (0,-0.75); 
    \fill[magenta!20] (180:0.7 and 0.75) -- (0,0.1) -- (190:0.35 and 0.75) -- cycle;
    \draw [magenta,line cap=round,dashed,line join=round] (180:0.7 and 0.75) -- (0,0.1) -- (190:0.35 and 0.75);
    \draw[string-green,line cap=round] (0,0.1) -- (-0.53,-0.06);
    \fill [lightgray!50,opacity=0.5] (0,0.75) arc (90:270:0.35 and 0.75) 
    (0,-0.75) arc (270:90:0.7 and 0.75);
    \draw (0,0.75) arc (90:270:0.35 and 0.75)
    (0,-0.75) arc (270:90:0.7 and 0.75);
    \draw [magenta,line cap=round] (180:0.7 and 0.75) -- (190:0.35 and 0.75);
    \fill[violet] (-0.53,-0.06) ellipse (0.015 and 0.02);
    \end{tikzpicture}
    \caption{}
    \label{fig:wedgepresball}
    \end{subfigure}
    \hfill
    \begin{subfigure}{0.3\textwidth}\centering
        \begin{tikzpicture}[scale=1.5]
    \draw [dashed,line cap = round] (-1.25,0.5) -- (0,-0.25) 
    (1,-0.5) -- (-0.25,0.25)
    (-1.25,-1) -- (0,-0.25) 
    (-0.25,-1.25) -- (1,-0.5); 
    \draw[dotted, line cap= round] (0,-0.25) -- (1,-0.5);
    \fill [magenta!40,opacity=0.5] (-1.25,0) .. controls (-0.75,-0.15) and (-0.75,-0.35) .. (-1.25,-0.5) -- (-0.25,-0.75) .. controls(0.25,-0.6) and (0.25,-0.4) .. (-0.25,-0.25);
    \draw[string-green,line cap=round] (-0.75,-0.125) .. controls (-0.25,-0.275) and (-0.25,-0.475) .. (-0.75,-0.625);
    \draw[magenta,line cap=round,dashed] (-1.25,0) .. controls (-0.75,-0.15) and (-0.75,-0.35) .. (-1.25,-0.5);
    \fill[lightgray!50,opacity=0.5] (-1.25,0.5) -- (-0.25,0.25) -- (-0.25,-1.25) --(-1.25,-1) -- cycle;
    \draw (-1.25,0.5) -- (-0.25,0.25) -- (-0.25,-1.25) --(-1.25,-1) -- cycle;
    \draw[magenta,line cap=round] (-1.25,0) -- (-0.25,-0.25)
    (-1.25,-0.5) -- (-0.25,-0.75);
    \draw[magenta,line cap=round,dashed] (-0.25,-0.25) .. controls (0.25,-0.4) and (0.25,-0.6) .. (-0.25,-0.75);
    \fill[string-green] (-0.75,-0.125) circle (0.02)
    (-0.75,-0.625) circle (0.02);
    \end{tikzpicture}
    \caption{}
    \label{fig:wedgeprestorus}
    \end{subfigure}
    \caption{Examples of a cylindrical part of a defect surface (a), a defect ball (b), and (c) the defect solid torus from \Cref{fig:circleunit}.}
    \label{fig:wedgepres}
\end{figure}

\begin{figure}[p]
\hspace*{-1cm}\begin{tikzcd}[ampersand replacement=\&]
	{M_{\Gamma_2}^{\underline{Z}} \sqcup -M_{\Gamma}^{\underline{Y}}\sqcup M_{\Gamma}^{\underline{Y}}\sqcup -M_{\Gamma_1}^{\underline{X}}} \&\&\& {\widehat{\Sigma_2}^{\underline{Y},\underline{Z}} \sqcup \widehat{\Sigma_1}^{\underline{X},\underline{Y}}} \\
	\\
	\\
    \\
    \\
	\\
	\\
	\\
	{M_{\Gamma_2}^{\underline{Z}} \sqcup -M_{\Gamma}\sqcup_{\widehat{\Gamma}} M_{\Gamma}\sqcup -M_{\Gamma_1}^{\underline{X}}} \&\&\& {(\widehat{\Sigma_2}\sqcup_{\widehat{\Gamma}}\widehat{\Sigma_1})^{\underline{X},\underline{Z}}} \\
	\\
	\\
	\\
	\\
	\\
	\& {M_{\Gamma_2}^{\underline{Z}} \sqcup -M_{\Gamma_1}^{\underline{X}}}
	\arrow["{\includegraphics[scale=0.8]{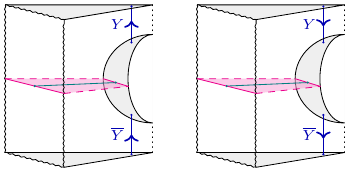}}", from=1-1, to=1-4]
	\arrow["{\includegraphics[scale=0.8]{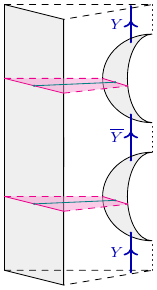}}"', from=1-1, to=9-1]
	\arrow["{\includegraphics[scale=0.8]{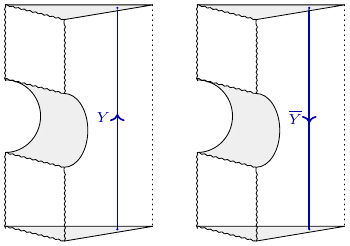}}", from=1-4, to=9-4]
	\arrow["{\includegraphics[scale=0.8]{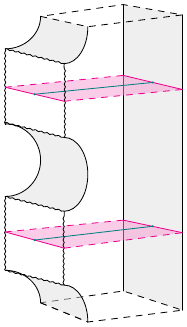}}",from=9-1, to=9-4]
	\arrow["{\includegraphics{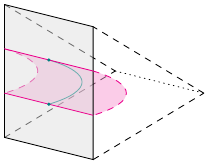}}", from=15-2, to=9-1]
	\arrow["{\includegraphics[scale=0.8]{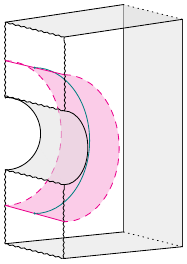}}"', from=15-2, to=9-4]
\end{tikzcd}
    \caption{Wedge presentation of diagram \ref{diag:oplaxbordcat} for $\frakC$ a defect circle in a neighbourhood of $\frakC$. For both arrows on the left only the non-identity bordism components are drawn.}
    \label{fig:bulkfact}
\end{figure}

With this preparation we can now study Diagram \ref{diag:oplaxbordcat}. The local version in the wedge presentation is illustrated in \Cref{fig:bulkfact}.

To compose bordisms in the wedge presentation we identify the corresponding gray shaded surface areas. In particular the composition of the top arrow with the right arrow amounts to 
\begin{equation}
    \left(\vcenter{\hbox
    {\includegraphics[scale=0.9]{figs/circlegluerightarrow.pdf}}}\right) 
    \circ  \left(\vcenter{\hbox{\includegraphics[scale=0.9]{figs/circlegluetoparrow.pdf}}}\right) \cong \vcenter{\hbox{\includegraphics[scale=0.9]{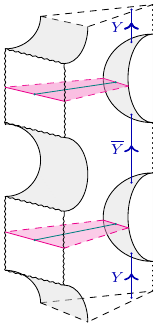}}}
\end{equation}
\noindent
Note here that the top and bottom triangle are identified.

Analogously the composition of the left arrow with the dashed arrow in the middle is given by
\begin{equation}
    \left(\vcenter{\hbox
    {\includegraphics[scale=0.9]{figs/circlegluedashedarrow.pdf}}}\right) 
    \circ  \left(\vcenter{\hbox{\includegraphics[scale=0.9]{figs/circleglueleftarrow.pdf}}}\right) \cong \vcenter{\hbox{\includegraphics[scale=0.9]{figs/circlegluecalc.pdf}}}
\end{equation}
Which shows that the dashed arrow indeed makes the square commute and thus represents the composition  $\Cor_{\frakS_2}\diamond \Cor_{\frakS_1}$. 

\noindent 
Finally, for the triangle we have
\begin{equation}
    \begin{aligned}
        \left(\vcenter{\hbox
    {\includegraphics[scale=0.9]{figs/circlegluedashedarrow.pdf}}}\right) 
    \circ  \left(\vcenter{\hbox{\includegraphics[scale=1.1]{figs/wedgeprestorus.pdf}}}\right) &\cong  \vcenter{\hbox{\includegraphics[scale=0.9]{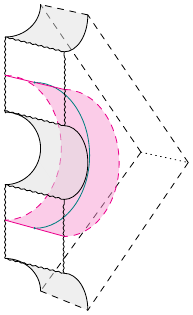}}}\\
    &\cong \vcenter{\hbox{\includegraphics[scale=0.9]{figs/circleglued.pdf}}}
    \end{aligned}
\end{equation}
where the second isomorphism uses the identification of the rectangular regions, see also \autocite[Sec.\,5.6]{FjFRS}. With this we find that also the triangle commutes.

\subsection{Non-degeneracy / oplax unitality}
The final axiom we need to check is the oplax unitality axiom in the form of \Cref{{diag:unitality}}. This is a purely algebraic condition on $\rmZ_\calc$ and thus leads to our final assumption:
\begin{assumption}\label{ass:4}
    For any $\frakC \in \WS(\DD_\calc)$, the components of the counit $\epsilon_{\Cor_\frakC} \colon \Cor_\frakC \otimes_\kk \Cor_\frakC^\dagger \Rightarrow \mathrm{Bl}_\calc(\id_\frakC)$ of the adjunction $\Cor_{\frakC} \dashv \Cor_{\frakC}^\dagger$ are given by $\rmZ_\calc(M_{\Gamma \times I}^{\underline{X},\underline{Y}})$.
\end{assumption}
\begin{remark}\label{rem:unitality}
For later use it is beneficial to note that $\epsilon_{\Cor_\frakC} \in \mathrm{Nat}(\Cor_\frakC \otimes_\kk \Cor_\frakC^\dagger, \mathrm{Bl}_\calc(\id_\frakC))$ gets sent to $\id_{\mathbbm{F}_\frakC} \in \Hom_{\mathrm{Bl}_\calc(\frakC)}(\mathbbm{F}_\frakC,\mathbbm{F}_\frakC)$ under the isomorphism $\mathrm{Nat}(\Cor_\frakC \otimes_\kk \Cor_\frakC^\dagger, \mathrm{Bl}_\calc(\id_\frakC)) \cong \Hom_{\mathrm{Bl}_\calc(\frakC)}(\mathbbm{F}_\frakC,\mathbbm{F}_\frakC)$ obtained by combining the isomorphisms \ref{eq:cor-block} and \ref{eq:cor-adj} with unitality of the modular functor, i.e.\ $\mathrm{Bl}_\calc(\id_\frakC)(-,-) \cong \Hom_{\mathrm{Bl}_\calc(\frakC)}(-,-)$.
\end{remark}
Putting everything together we arrive at the main result of this article:
\begin{theorem}\label{thm:fullcft}
Let 
\begin{equation}
        \rmZ_\calc \colon \bord_{3,2}^{\chi,\mathrm{ def}}(\DD_\calc) \to \vect
\end{equation}
be a $3$d defect TFT extending the $3$d TFT with embedded ribbon graphs
\begin{equation}
        \widehat{\rmV}_\calc \colon \widehat{\bord}_{3,2}^{\chi}(\calc) \to \vect
\end{equation}
of \autocite{DGGPR19} in the sense of \Cref{def:deftftextension}. If $\rmZ_\calc$ satisfies Assumptions \ref{ass:1}, \ref{ass:3}, \ref{ass:2}, and \ref{ass:4} then evaluating it on the connecting manifolds defines a full conformal field theory based on $\mathrm{Bl}_\calc$, i.e.\ a braided monoidal oplax natural transformation
\begin{align}
    \begin{tikzcd}[ampersand replacement=\&]
	\WS(\DD_\calc) \&\& {\Prof}
	\arrow[""{name=0, anchor=center, inner sep=0}, "{\Delta_{\kk}}", curve={height=-24pt}, from=1-1, to=1-3]
	\arrow[""{name=1, anchor=center, inner sep=0}, "{\mathrm{Bl_\calc}}"', curve={height=24pt}, from=1-1, to=1-3]
	\arrow["\Cor", shorten <=6pt, shorten >=6pt, Rightarrow, from=0, to=1]
    \end{tikzcd}.
\end{align}
\end{theorem}

\begin{remark}
For $\calc$ semisimple our construction recovers the work of \autocite{FRSI} by considering the defect TFT constructed in \autocite{CRS17defRT}. To be more precise the $2$-morphism components of $\Cor$ recover their correlators while the $1$-morphism components of $\Cor$ correspond to the vector spaces used in the analysis of the field content explained in \autocite[Sec.\,3]{FRSIV}. A more detailed comparison will be given in upcoming work. We also want to mention that the naturality proofs above are strongly informed by the topological considerations in the proofs given in \autocite{FjFRS,FFS12}. In particular, our discussion of bulk factorisation is heavily inspired by the ideas and argument in \autocite[Sec.\,5]{FjFRS}. The reason the algebraic arguments in their proofs are absent from our discussion is because they are hidden in the defect TFT.
\end{remark}

%
%
   
\section{Example: Diagonal CFT}\label{sec:logcardy}
In this final section we study the simplest non-trivial situation to which our construction applies, the one with defect TFT given by $\rmZ_\calc = \widehat{\rmV}_\calc \colon \widehat{\bord}^\chi_{3,2}(\calc) \to \vect$. In CFT terminology this example corresponds to the so-called \emph{diagonal}/\emph{charge-conjugate}/\emph{Cardy} case. 

As explained at the end of \Cref{sec:3ddeftft} the defect data $\DD_\calc$ is completely encoded in the modular tensor category $\calc$ itself with the label set of $1$-strata $D_1$ the objects of $\calc$, the label set of $0$-strata $D_0$ the morphisms of $\calc$, and the other two label sets one element sets which we take to be $\{\unit\}$. In particular the free boundaries are also labelled with objects in $\calc$. Let us denote the corresponding world sheet $2$-category by $ \WS_\calc:= \WS(\DD_\calc)$.

Before we can compute CFT quantities we first need to check that $\widehat{\rmV}_\calc$ satisfies the conditions of \Cref{thm:fullcft}. Assumptions \ref{ass:1} and \ref{ass:3} are \autocite[Lem.\,4.3]{HR2024modfunc} and \autocite[Prop.\,4.8]{HR2024modfunc}, respectively. For \Cref{ass:2} we need to compare the image of the bordisms depicted in \Cref{fig:unitadj} for any defect interval and defect circle under $\widehat{\rmV}_\calc$ with the identity of the corresponding field content. Due to this we will first compute the field content in \Cref{sec:fieldcont} and then verify \Cref{ass:2}. Afterwards we will discuss \Cref{ass:4} in detail in \Cref{sec:twopointcor}. Finally, in \Cref{sec:cftstuff} we will compute some quantities of physical interest including boundary states as well as partition functions and compare our results to the ones the proposed in \autocite{FGSS18cardy}. We will also compute the action of a line defect on bulk fields, which to our knowledge has not been considered in the non-semisimple setting before.

\subsection{Field content}\label{sec:fieldcont}
Let us start with the field content. First note that monoidality of $\calc$ allows us to reduce our discussion to understanding the case of a defect interval with no point defects in its interior and of a defect circle with one point defect. Informally this can be seen as the result of fusing the defects.

\subsubsection*{Boundary fields}
Let $m,n \in \calc$ and let $I_{n,m} \in \WS_\calc$ be the defect interval with underlying manifold the standard interval $I = [0,1]$ oriented from $0$ to $1$, and free boundaries labelled with $m$ at $\{0\} \subset I $ and $n$ at $\{1\} \subset I$. For $X \in \calc$ the corresponding connecting manifold $M_{I_{n,m}}^{X}$ is a three punctured sphere with the $X$ and $m$ puncture positively oriented and the $n$ puncture negatively oriented. The reason for these orientations is because we want to think of as the embedded interval as ingoing while the $X$-labelled marked point should be outgoing. From this we get
\begin{equation}
    \begin{aligned}
        \widehat{\rmV}_\calc\left(M_{I_{n,m}}^{\underline{X}}\right) &\cong \Hom_\calc(n, m \otimes X) \\
        &\cong \Hom_\calc(  m^*\otimes n ,X),
    \end{aligned}
\end{equation}
where we used that state spaces of the TFT $\widehat{\rmV}_\calc$ are isomorphic to morphism spaces in $\calc$ \autocite[Prop.\,4.17]{DGGPR19}.
We can now read off the boundary field content as
\begin{equation}
    \mathbbm{F}_{I_{n,m}} = m^* \otimes n.
\end{equation}
For \Cref{ass:2} we need to consider the bordism from \Cref{fig:intervalunit}. In our setting this reduces to
\begin{equation}
    \eta_{M_{I_{n,m}}} = 
    \begin{tikzpicture}[baseline=0cm]
        \draw[loosely dashed,thick,line cap=round] (1,0) arc (0:180:1 and 0.4);
        \fill[white,opacity=0.7] (0,0) circle (1); 
         \begin{scope}[decoration={
                      markings,
                      mark=at position 0.52 with {\arrow{>}}}
                     ] 
            \draw[string-violet,thick,postaction={decorate}] (-0.3,0.7) -- (-0.3,-0.7);
            \draw[string-violet,thick,postaction={decorate}] (0.3,-0.55) -- (0.3,0.85);
        \end{scope}
        \filldraw[string-violet] (-0.3,0.7) circle (0.04)
        (-0.3,-0.7) circle (0.04)
        (0.3,-0.55) circle (0.04)
        (0.3,0.85) circle (0.04);
        \draw[white, ultra thick] (1,0) arc (0:-180:1 and 0.4);
         \draw[thick] (0,0) circle (1)
         (1,0) arc (0:-180:1 and 0.4);
         \node at (-0.3,0) [left,string-violet] {\scriptsize $m$};
         \node at (0.3,0.05) [right,string-violet] {\scriptsize $n$};
    \end{tikzpicture}
    \colon \varnothing \longrightarrow \begin{tikzpicture}[baseline=0cm]
         \draw[thick, loosely dashed,line cap=round] (1,0) arc (0:180:1 and 0.4);
        \fill[white,opacity=0.7] (0,0) circle (1); 
        \draw[white, ultra thick] (1,0) arc (0:-180:1 and 0.4);
         \draw[thick] (0,0) circle (1)
         (1,0) arc (0:-180:1 and 0.4);
         \filldraw[string-violet] (120:1) circle (0.04)
        (60:1) circle (0.04)
        (240:1) circle (0.04)
        (300:1) circle (0.04);
         \node at (120:0.9) [above left,string-violet] {\scriptsize $(m,-)$};
         \node at (240:0.9) [below left,string-violet] {\scriptsize $(m,-)$};
         \node at (60:0.9) [above right,string-violet] {\scriptsize $(n,+)$};
         \node at (300:0.9) [below right,string-violet] {\scriptsize $(n,+)$};
    \end{tikzpicture}
\end{equation}
It is now straightforward to see that $\rmZ_\calc(\eta_{M_{I_{n,m}}})$ corresponds to $\id_{m^*\otimes n}$ 
under the isomorphism between the state spaces of $\widehat{\rmV}_\calc$ with morphism spaces in $\calc$, as desired.

\subsubsection*{Bulk and disorder fields}
Let $k\in \calc$ and let $S^1_k\in \WS_\calc$ be a defect circle with a single negatively oriented $0$-stratum at $(0,-1) \in S^1\subset \mathbbm{R}^2$ labelled with $k$. For $(X,\overline{X})\in \calc \times \overline{\calc} \cong \calc^{\pi_0(\widehat{S^1})}$ the connecting manifold $M_{S^1}^{(X,\overline{X})}$ 
is topologically a three-punctured $2$-sphere. We want to stress here that $X$ and $\overline{X}$ are independent of each other and do not correspond to the same object in the category underlying $\calc$ and $\overline{\calc}$. With this we compute 
\begin{equation}
    \begin{aligned}
        \widehat{\rmV}_\calc\left(M_{S^1_k}^{(X,\overline{X})}\right) &\cong \Hom_\calc(k,X\otimes \overline{X}) \\
        &\cong \Hom_{\calc\boxtimes\overline{\calc}}(F(k),X\boxtimes \overline{X})
    \end{aligned}
\end{equation}
where $F \colon \calc \to \calc \boxtimes \overline{\calc}$ is the two-sided adjoint of the monoidal product functor $\otimes \colon \calc\boxtimes \overline{\calc} \to \calc$ as discussed in \Cref{subsec:reptheory}. From this we get  
\begin{equation}
    \mathbbm{F}_{S^1_k} \cong F(k) \cong \int^{X \in \calc} X^*\otimes k \boxtimes X \in \calc \boxtimes \overline{\calc}.
\end{equation}
In particular for $k=\unit$ we get the bulk field content 
\begin{equation}
    \mathbbm{F}_{S^1_\unit} \cong \int^{X \in \calc} X^*\boxtimes X = L \in \calc \boxtimes \overline{\calc}.
\end{equation}
For $\calc$ semisimple this is simply $\coend \cong \bigoplus_{i\in \mathrm{Irr}} i^* \boxtimes i$ 
where $\mathrm{Irr}$ is a set of representatives of simple objects in $\calc$. With this we reproduced the known bulk fields for the rational full CFT with charge-conjugate partition function justifying the name also in the non-semisimple setting. We note that $\mathbbm{F}_{I_{n,m}}$ and $\mathbbm{F}_{S^1_\unit}$ are precisely the field content proposed in \autocite{FGSS18cardy}.

For \Cref{ass:2} first recall the equivalence $\calc \boxtimes \overline{\calc} \simeq \calc_\coend$ to the category of right $\coend$-modules in $\calc$ from \Cref{subsec:reptheory}. Under this equivalence $F(-)$ is simply the free functor $-\otimes \coend$. Next note that $\widehat{\rmV}_\calc\left(M_{S^1_k}^\dagger \sqcup_{\widehat{S_k^1}} M_{S^1_k}\right) \cong \Hom_\calc(k\otimes\coend, k)  \cong \Hom_\coend(k\otimes \coend, k \otimes \coend)$. Where the (inverse of the) second isomorphism is explicitly given by
\begin{equation}\label{eq:coendadjunc}
    \begin{aligned}
        \Hom_\coend(X,k\otimes \coend) &\to \Hom_\calc(X,k) \\
        f &\mapsto (\id_k \otimes \lambda) \circ f
    \end{aligned}
\end{equation}
and exhibits $- \otimes \coend$ as the right adjoint of the forgetful functor $\calc_\coend \to \calc$.
The bordism we have to consider is
\begin{equation}
    \eta_{S^1_k} = 
    \begin{tikzpicture}[baseline=0cm]
    \draw[thick] (0,0) ellipse (1.5 and 1)
           (0.5,0.06) arc (-10:-170:0.5 and 0.3)
           (0.4,-0.06) arc (10:170:0.4 and 0.25);
         \begin{scope}[decoration={
                      markings,
                      mark=at position 0.52 with {\arrow{>}}}
                     ] 
            \draw[string-green,thick,postaction={decorate}] (-0.6,-0.4) arc (180:360: 0.6 and 0.2);
        \end{scope}
        \filldraw[string-green] 
        (-0.6,-0.4) circle (0.5mm)
        (0.6,-0.4) circle (0.5mm);
         \node at (0,-0.35) [string-green] {\scriptsize $k$};
    \end{tikzpicture}
    \colon \varnothing \longrightarrow 
    \begin{tikzpicture}[baseline=0cm]
    \draw[thick] (0,0) ellipse (1.5 and 1)
           (0.5,0.06) arc (-10:-170:0.5 and 0.3)
           (0.4,-0.06) arc (10:170:0.4 and 0.25);
        \filldraw[string-green] 
        (-0.6,-0.4) circle (0.5mm)
        (0.6,-0.4) circle (0.5mm);
        \node at (-0.5,-0.4) [above left,string-green] {\scriptsize $(k,-)$};
        \node at (0.5,-0.4) [above right,string-green] {\scriptsize $(k,+)$};
    \end{tikzpicture}
\end{equation}
We can get $\widehat{\rmV}_\calc\left( \eta_{S^1_k} \right)$ using the sliding trick employed in \autocite[Sec.\,4.4]{HR2024modfunc} for $X=\unit$, or more precisely by setting $X=\unit$ in \autocite[Eq.\,4.26]{HR2024modfunc}. After a straightforward calculation in $\calc$ this then leads to $\widehat{\rmV}_\calc\left(\eta_{S^1_k}\right) = \lambda \otimes \id_k \in \Hom_\calc(\coend\otimes k, k)$. Which is the image of $\id_{\coend\otimes k} \in \Hom_\coend(\coend\otimes k, \coend \otimes k)$ under \eqref{eq:coendadjunc}.

\subsection{Two point correlators}\label{sec:twopointcor}
We now turn to study \Cref{ass:4}. To this end first recall \Cref{lem:profadjunc} that the components of the counit $\epsilon_{\Cor_\frakC} \colon \Cor_\frakC \otimes_\kk \Cor_\frakC^\dagger \Rightarrow \mathrm{Bl}_\calc(\id_\frakC)$ are induced by the composition map
\begin{equation}
       \Hom_{\Bl_\calc(\frakC)}(\underline{X},\mathbbm{F}_{\frakC}) \otimes_\kk  \Hom_{\Bl_\calc(\frakC)}(\mathbbm{F}_{\frakC},\underline{Y}) \to \Hom_{\Bl_\calc(\frakC)}(\underline{X},\underline{Y}).
\end{equation}

\subsubsection{Boundary}
Let us start with $I_{n,m}$ the defect interval considered above. For $X,Y\in \calc$ the connecting bordism
\begin{equation}
    M_{\id_{I_{n,m}}}^{(X,Y)} \colon -M_{I_{n,m}}^{X} \sqcup M_{I_{n,m}}^{Y} \to \widehat{\id_{I_{n,m}}}^{(X,Y)}
\end{equation}
of $\id_{I_{n,m}}$ is given by
\begin{equation}
        \begin{tikzpicture}
    \draw[thick,loosely dotted,line cap=round](0:2.5 and 0.9) arc (0:180:2.5 and 0.9);
    \fill[white,opacity=0.7] (0,0) ellipse (2.5 and 1.5);
    \draw[thick] (0,0) ellipse (2.5 and 1.5)
    (2.5,0) arc (0:-180:2.5 and 0.9);
    \draw[thick] (0.9,0) circle (0.5)
    (-0.9,0) circle (0.5);
    \draw[thick,loosely dotted,line cap=round](0.4,0) arc (180:0:0.5 and 0.2)
    (-0.4,0) arc (0:180:0.5 and 0.2); 
    \begin{scope}[very thick,decoration={
                      markings,
                      mark=at position 0.6 with {\arrow{>}}}
                     ]  
        \draw[string-blue,thick,postaction={decorate}] (-2.5,0) -- (-1.4,0);
        \draw[string-blue,thick,postaction={decorate}] (1.4,0) -- (2.5,0);
        \draw[string-violet,thick,postaction={decorate}] ([xshift=0.9cm]110:0.5 and 0.2)  .. controls (0,0.3) .. ([xshift=-0.9cm]70:0.5 and 0.2);
    \end{scope}
    \fill[string-violet] ([xshift=-0.9cm]70:0.5 and 0.2) circle (0.04)
    ([xshift=0.9cm]110:0.5 and 0.2) circle (0.04);
    \fill[white,opacity=0.7] (-0.4,0) arc (0:180:0.5)
    (1.4,0) arc (0:180:0.5);
    \draw[thick] (-0.4,0) arc (0:180:0.5)
    (1.4,0) arc (0:180:0.5)
    (0.4,0) arc (180:360:0.5 and 0.2)
    (-0.4,0) arc (0:-180:0.5 and 0.2);
    \draw (-0.4,0) arc (0:180:0.5);
    \begin{scope}[very thick,decoration={
                      markings,
                      mark=at position 0.6 with {\arrow{>}}}
                     ]  
        \draw[string-violet,thick,postaction={decorate}] ([xshift=-0.9cm]290:0.5 and 0.2) .. controls (0,-0.3) .. ([xshift=0.9cm]250:0.5 and 0.2);
    \end{scope}
    \fill[string-violet] ([xshift=-0.9cm]290:0.5 and 0.2) circle (0.04)
    ([xshift=0.9cm]250:0.5 and 0.2) circle (0.04);
    \fill[string-blue](-2.5,0) circle (0.04)
     (-1.4,0) circle (0.04)
     (1.4,0) circle (0.04)
     (2.5,0) circle (0.04);
     \node at (0,0.4) [string-violet, above] {\scriptsize $m$};
     \node at (0,-0.4) [string-violet, below] {\scriptsize $n$};
     \node at (-1.9,-0.05) [string-blue, below] {\scriptsize $X$};
     \node at (1.9,-0.05) [string-blue, below] {\scriptsize $Y$};
     %
     %
     %
     \begin{scope}[shift={(4.5,0)}]
    \draw[thick] (0.9,0) circle (0.5)
    (-0.9,0) circle (0.5);
    \draw[thick,loosely dotted,line cap=round](0.4,0) arc (180:0:0.5 and 0.2)
    (-0.4,0) arc (0:180:0.5 and 0.2); 
    \fill[white,opacity=0.7] (-0.4,0) arc (0:180:0.5)
    (1.4,0) arc (0:180:0.5);
    \draw[thick] (-0.4,0) arc (0:180:0.5)
    (1.4,0) arc (0:180:0.5)
    (0.4,0) arc (180:360:0.5 and 0.2)
    (-0.4,0) arc (0:-180:0.5 and 0.2);
    \draw[thick] (-0.4,0) arc (0:180:0.5);
    \fill[string-violet] ([xshift=-0.9cm]60:0.5) circle (0.04)
    ([xshift=0.9cm]240:0.5) circle (0.04)
    ([xshift=-0.9cm]120:0.5) circle (0.04)
    ([xshift=0.9cm]300:0.5) circle (0.04);
    \fill[string-blue](-0.9,-0.5) circle (0.04)
     (0.9,0.5) circle (0.04);
     \node at (-0.9,-0.45) [string-blue,below] {\scriptsize $(X,-)$};
     \node at ([xshift=-0.9cm]70:0.4) [string-violet,above right] {\scriptsize $(n,+)$};
     \node at ([xshift=-0.9cm]110:0.4) [string-violet,above left] {\scriptsize $(m,-)$};
     \node at (0.9,0.5) [string-blue,above] {\scriptsize $(Y,+)$};
     \node at ([xshift=0.9cm]300:0.4) [string-violet,below right] {\scriptsize $(n,-)$};
     \node at ([xshift=0.9cm]250:0.4) [string-violet,below left] {\scriptsize $(m,+)$};
     \end{scope}
     %
     %
     %
    \begin{scope}[shift={(8.2,0)}]
    \draw[thick, loosely dotted,line cap=round] (0.6,0) arc (0:180:0.6 and 0.3);
    \fill[white,opacity=0.7] (0,0) ellipse (0.6);
    \draw[thick] (0,0) ellipse (0.6);
    \draw[thick] (0.6,0) arc (0:-180:0.6 and 0.3);
    \fill[string-blue](0,-0.6) circle (0.04)
     (0,0.6) circle (0.04);
     \node at (0,-0.55) [string-blue,below] {\scriptsize $(X,-)$};
     \node at (0,0.55) [string-blue,above] {\scriptsize $(Y,+)$};
    \end{scope}
    \draw[thick,->,line cap=round] (6.1,0) -- (7.2,0);
    \fill[black] (2.7,-0.07) circle (0.025)
    (2.7,0.07) circle (0.025);
    \end{tikzpicture}
\end{equation}
Here and below we will always use blue ribbons for ``free labels'' $X,Y$, i.e.\ the components of the natural transformation $\mathrm{cor}_{\id_{n,m}}$. 
Applying $\widehat{\rmV}_\calc$ and using the isomorphism between TFT state spaces and Hom-spaces in $\calc$ we immediately get the linear map 
\begin{equation}
    \begin{aligned}
     \Hom_\calc(X,m^*\otimes n) \otimes_\kk  \Hom_\calc(m^* \otimes n, Y) &\to  \Hom_\calc(X,Y) \\
     f \otimes_\kk g &\mapsto g \circ f.
    \end{aligned}
\end{equation}
Thus $\mathrm{cor}_{\id_{n,m}}$ induces the same map as the composition map.
\subsubsection{Bulk}\label{subsec:bulk2point}
Let $S^1_k \in \WS_\calc$ be the defect circle from above. For $(X,\overline{X}),(Y,\overline{Y}) \in \calc\times \overline{\calc}$ the connecting bordism 
\begin{equation}
    M_{\id_{S^1_k}}^{(X,\overline{X},Y,\overline{Y})} \colon -M_{S^1_k}^{(X,\overline{X})} \sqcup M_{S^1_k}^{(Y,\overline{Y})} \to \widehat{\id_{S^1_k}}^{(X,\overline{X},Y,\overline{Y})}
\end{equation}
of $\id_{S^1_k}$ is given by a bordism which has as underlying $3$-manifold $S^3$ with two incoming and two outgoing $3$-balls cut out. Since this cannot be as nicely visualised, we instead draw it as a standard $3$-ball with three smaller $3$-balls cut out and read the middle interior boundary sphere as outgoing:
\begin{equation}
        \begin{tikzpicture}
    \draw[thick,loosely dotted,line cap=round](0:2.5 and 0.9) arc (0:180:2.5 and 0.9);
    \fill[white,opacity=0.7] (0,0) circle (2.5);
    \draw[thick,loosely dotted,line cap=round](0:0.5 and 0.2) arc (0:180:0.5 and 0.2)
    ([yshift=-1.5cm]0:0.5 and 0.2) arc (0:180:0.5 and 0.2)
    ([yshift=1.5cm]0:0.5 and 0.2) arc (0:180:0.5 and 0.2);
    \fill[white,opacity=0.7](0,0) circle (0.5)
    (0,1.5) circle (0.5)
    (0,-1.5) circle (0.5);
    \draw[thick] (0,0) circle (0.5)
    (0,1.5) circle (0.5)
    (0,-1.5) circle (0.5);
    \draw[thick,line cap=round](0:0.5 and 0.2) arc (0:-180:0.5 and 0.2)
    ([yshift=-1.5cm]0:0.5 and 0.2) arc (0:-180:0.5 and 0.2)
    ([yshift=1.5cm]0:0.5 and 0.2) arc (0:-180:0.5 and 0.2);
    \begin{scope}[thick,decoration={
                      markings,
                      mark=at position 0.57 with {\arrow{>}}}
                     ]  
        \draw[string-blue,thick,postaction={decorate}] (0,1) -- (0,0.5);
        \draw[string-blue,thick,postaction={decorate}] (0,2) -- (0,2.5);
        \draw[string-blue,thick,postaction={decorate}] (0,-0.5) -- (0,-1);
        \draw[string-blue,thick,postaction={decorate}] (0,-2.5) -- (0,-2);
    \end{scope}
    \begin{scope}[thick,decoration={
                      markings,
                      mark=at position 0.51 with {\arrow{>}}}
                     ]  
                     \draw[string-green,thick,postaction={decorate}] (0.5,-1.5) arc (-90:90:0.7 and 1.5);
    \end{scope}
    \draw[ultra thick,white](2.5,0) arc (0:-180:2.5 and 0.9);
    \draw[thick] (0,0) ellipse (2.5)
    (2.5,0) arc (0:-180:2.5 and 0.9);
    \fill[string-green](0.5,-1.5) circle (0.04)
     (0.5,1.5) circle (0.04);
     \fill[string-blue]
     (0,2.5) circle (0.04)
     (0,2) circle (0.04)
     (0,1) circle (0.04)
     (0,0.5) circle (0.04)
     (0,-0.5) circle (0.04)
     (0,-1) circle (0.04)
     (0,-2) circle (0.04)
     (0,-2.5) circle (0.04);
    \node at (1.45,0) [string-green] {\scriptsize $k$};
    \node at (0,2.25) [string-blue,right] {\scriptsize $\overline{Y}$};
    \node at (0,0.75) [string-blue,right] {\scriptsize $Y$};
    \node at (0,-0.7) [string-blue,right] {\scriptsize $X$};
    \node at (0,-2.25) [string-blue,right] {\scriptsize $\overline{X}$};
     %
     %
     %
     \begin{scope}[shift={(4.5,0)}]
    \draw[thick] (0.9,0) circle (0.5)
    (-0.9,0) circle (0.5);
    \draw[thick,loosely dotted,line cap=round](0.4,0) arc (180:0:0.5 and 0.2)
    (-0.4,0) arc (0:180:0.5 and 0.2); 
    \fill[white,opacity=0.7] (-0.4,0) arc (0:180:0.5)
    (1.4,0) arc (0:180:0.5);
    \draw[thick] (-0.4,0) arc (0:180:0.5)
    (1.4,0) arc (0:180:0.5)
    (0.4,0) arc (180:360:0.5 and 0.2)
    (-0.4,0) arc (0:-180:0.5 and 0.2);
    \draw[thick] (-0.4,0) arc (0:180:0.5);
    \fill[string-blue] ([xshift=0.9cm]60:0.5) circle (0.04)
    ([xshift=-0.9cm]240:0.5) circle (0.04)
    ([xshift=0.9cm]120:0.5) circle (0.04)
    ([xshift=-0.9cm]300:0.5) circle (0.04);
    \fill[string-green](0.9,-0.5) circle (0.04)
     (-0.9,0.5) circle (0.04);
     \node at (0.9,-0.45) [string-green,below] {\scriptsize $(k,-)$};
     \node at ([xshift=0.9cm]70:0.4) [string-blue,above right] {\scriptsize $(\overline{Y},+)$};
     \node at ([xshift=0.9cm]110:0.4) [string-blue,above left] {\scriptsize $(Y,+)$};
     \node at (-0.9,0.5) [string-green,above] {\scriptsize $(k,+)$};
     \node at ([xshift=-0.9cm]300:0.4) [string-blue,below right] {\scriptsize $(\overline{X},-)$};
     \node at ([xshift=-0.9cm]250:0.4) [string-blue,below left] {\scriptsize $(X,-)$};
     \end{scope}
     %
     %
     %
    \begin{scope}[shift={(8.8,0)}]
    \draw[thick] (0.5,0) circle (0.5)
    (-0.9,0) circle (0.5);
    \draw[thick,loosely dotted,line cap=round](0,0) arc (180:0:0.5 and 0.2)
    (-0.4,0) arc (0:180:0.5 and 0.2); 
    \fill[white,opacity=0.7] (-0.4,0) arc (0:180:0.5)
    (1,0) arc (0:180:0.5);
    \draw[thick] (-0.4,0) arc (0:180:0.5)
    (1,0) arc (0:180:0.5)
    (0,0) arc (180:360:0.5 and 0.2)
    (-0.4,0) arc (0:-180:0.5 and 0.2);
    \draw[thick] (-0.4,0) arc (0:180:0.5);
    \fill[string-blue](-0.9,-0.5) circle (0.04)
     (-0.9,0.5) circle (0.04)
     (0.5,-0.5) circle (0.04)
     (0.5,0.5) circle (0.04);
     \node at (-0.9,-0.45) [string-blue,below] {\scriptsize $(X,-)$};
     \node at (-0.9,0.5) [string-blue,above] {\scriptsize $(Y,+)$};
     \node at (0.5,-0.45) [string-blue,below] {\scriptsize $(\overline{X},-)$};
     \node at (0.5,0.5) [string-blue,above] {\scriptsize $(\overline{Y},+)$};
    \end{scope}
    \draw[thick,->,line cap=round] (6.1,0) -- (7.2,0);
    \fill[black] (2.7,-0.07) circle (0.025)
    (2.7,0.07) circle (0.025);
    \end{tikzpicture}
\end{equation}

Applying the TFT $\widehat{\rmV}_\calc$ we get a linear map
\begin{equation}\label{eq:circlepairingtft}
     \widehat{\rmV}_\calc\left(M_{\id_{S^1_k}}^{(X,\overline{X},Y,\overline{Y})}\right) \colon \widehat{\rmV}_\calc \left(- M_{S^1_k}^{(X,\overline{X})} \sqcup M_{S^1_k}^{(Y,\overline{Y})}\right) \to  \widehat{\rmV}_\calc\left(\widehat{\id_{S^1_k}}^{(X,\overline{X},Y,\overline{Y})}\right)
\end{equation}
which we want to compare to the composition map
\begin{equation}
    \Hom_{\calc\boxtimes\overline{\calc}}(X\boxtimes \overline{X},F(k)) \otimes_\kk \Hom_{\calc\boxtimes\overline{\calc}}(F(k),Y\boxtimes \overline{Y}) \to \Hom_{\calc\boxtimes\overline{\calc}}(X\boxtimes \overline{X},Y\boxtimes\overline{Y}).
\end{equation}

To this end it will again be useful to employ the equivalence $\calc\boxtimes \overline{\calc} \simeq \calc_\coend$. Under this equivalence the relevant adjunction isomorphisms are given by
\begin{equation}
\begin{aligned}
    \phi \colon \Hom_\calc(k,Y\otimes \overline{Y}) &\to \Hom_\coend(k\otimes\coend,Y\otimes \overline{Y}) \\
     g &\mapsto \rho_{Y\otimes\overline{Y}} \circ (g \otimes \id_\coend)
\end{aligned}  
\end{equation}
and
\begin{equation}
\begin{aligned}
  \psi \colon \Hom_\calc(X\otimes \overline{X},k) &\to \Hom_\coend(X\otimes \overline{X},k\otimes\coend) \\
    f &\mapsto \calD^{-1} (f \otimes \id_\coend) \circ \delta^\Lambda_{X\otimes \overline{X}}
\end{aligned}
\end{equation}
where $\rho_{Y\otimes\overline{Y}}$ is the canonical $\coend$ action from \Cref{lem:center-modules}, $\delta^\Lambda_{X\otimes \overline{X}}$ is the coaction on $X\otimes \overline{X}$ obtained by precomposing the canonical action $\rho_{X\otimes \overline{X}}$ with the Radford copairing $(S \otimes\id ) \circ\Delta \circ \Lambda$ from \Cref{prop:frobpairingcoend}, and $\calD$ is a choice of a square root of the modularity parameter $\zeta$ needed to define the $3$-manifold invariant, see e.g.\ \autocite[Sec.\,3.2]{DGGPR19}. Note that introducing the factor $\calD^{-1}$ is merely a convention and we choose this one because it will make the rest of our argument more transparent. Next recall the canonical isomorphisms
\begin{align}\label{eq:homspaceprod}
       \Hom_\calc(X,Y)\otimes_\kk \Hom_{\overline{\calc}}(\overline{X},\overline{Y})\cong\Hom_{\calc\boxtimes\overline{\calc}}(X\boxtimes \overline{X},Y\boxtimes\overline{Y}) \cong \Hom_\coend(X\otimes\overline{X},Y\otimes\overline{Y})
\end{align}
coming from the universal property of the Deligne product \autocite[Prop.\,1.11.2]{EGNO} and the equivalence $\calc\boxtimes \overline{\calc} \simeq \calc_\coend$. 
Combining these isomorphisms with the isomorphisms between TFT state spaces and morphism spaces in $\calc$ leads us to study commutativity of the following diagram:
\begin{equation}\label{diag:circlepairing}
    \begin{tikzcd}[ampersand replacement=\&]
	{\widehat{\rmV}_\calc \left( -M_{S^1_k}^{(X,\overline{X})} \sqcup M_{S^1_k}^{(Y,\overline{Y})}\right)} \& {\widehat{\rmV}_\calc \left(\widehat{\id_{S^1_k}}^{(X,\overline{X},Y,\overline{Y})}\right)} \\
	{\Hom_\calc(X\otimes \overline{X},k) \otimes_\kk  \Hom_\calc(k,Y\otimes \overline{Y})} \& {\Hom_\calc(X,Y)\otimes_\kk \Hom_{\overline{\calc}}(\overline{X},\overline{Y})} \\
	{\Hom_\coend(X\otimes \overline{X},k\otimes\coend) \otimes_\kk \Hom_\coend(k\otimes\coend,Y\otimes \overline{Y})} \& {\Hom_\coend(X\otimes\overline{X},Y\otimes\overline{Y})}
	\arrow["{\widehat{\rmV}_\calc\left(M_{\id_{S^1_k}}^{(X,\overline{X},Y,\overline{Y})}\right)}", from=1-1, to=1-2]
	\arrow["\cong"', from=1-1, to=2-1]
	\arrow["\cong", from=1-2, to=2-2]
	\arrow["{\psi\otimes_\kk\phi}"', from=2-1, to=3-1]
	\arrow["\cong", from=2-2, to=3-2]
	\arrow["\circ", from=3-1, to=3-2]
\end{tikzcd}
\end{equation}
Let us first analyse the purely algebraic part. To this end consider the linear map
\begin{equation}
\begin{aligned}
    \Hom_\calc(X\otimes \overline{X},k) \otimes_\kk \Hom_\calc(k,Y\otimes \overline{Y}) &\to \Hom_\coend(X\otimes\overline{X},Y\otimes \overline{Y}) \\
     f \otimes_\kk g &\mapsto \calD^{-1} \rho_{Y\otimes\overline{Y}} \circ (g \circ f \otimes \id_\coend) \circ \delta^\Lambda_{X\otimes \overline{X}}.
\end{aligned}
\end{equation}
In terms of bichrome graphs, as defined in \autocite[Sec.\,3.2]{DGGPR19}, we can rewrite this as follows
\begin{equation}\label{eq:cuttingforbordism}
\begin{aligned}
\rho_{Y\otimes\overline{Y}} \circ (g \circ f \otimes \id_\coend) \circ \delta^\Lambda_{X\otimes \overline{X}} &=
        \begin{tikzpicture}[baseline=-0.1cm]
            \draw[draw=string-blue,thick, line cap = round] (0.4,-1.6) -- (0.4,1.6)
            (-0.4,-1.6) -- (-0.4,1.6);
            \draw[draw=string-blue,thick, line cap = round] (0.9,-0.6) -- (0.9,0.6);
            \filldraw[fill=white,draw=string-blue,thick, line cap = round] (-0.6,-0.3) rectangle (0.6,0.3);
            \filldraw[fill=white,draw=string-blue,thick, line cap = round] (-0.6,-1.2) rectangle (1.2,-0.6);
            \filldraw[fill=white,draw=string-blue,thick, line cap = round] (-0.6,0.6) rectangle (1.2,1.2);
            \node at (0.9,0) [right,string-blue] {\scriptsize $\coend$};
            \node at (-0.4,-1.6) [below,string-blue] {\scriptsize $X$};
            \node at (0.4,-1.56) [below,string-blue] {\scriptsize $\overline{X}$};
            \node at (-0.4,1.6) [above,string-blue] {\scriptsize $Y$};
            \node at (0.4,1.6) [above,string-blue] {\scriptsize $\overline{Y}$};
            \node at (0,0) [string-blue] {\scriptsize $g\circ f$};
            \node at (0.45,0.9) [string-blue] {\scriptsize $\rho$\tiny${}_{Y\otimes\overline{Y}}$};
            \node at (0.45,-0.9) [string-blue] {\scriptsize$\delta$\tiny${}^\Lambda_{X\otimes \overline{X}}$};
        \end{tikzpicture} = \begin{tikzpicture}[baseline=-0.1cm]
            \draw[draw=string-blue,thick, line cap = round] (0.4,-1.3) -- (0.4,1.3)
            (-0.4,-1.3) -- (-0.4,1.3);
            \draw[white,line width=1mm] (1.1,0.6) arc (0:180:0.5 and 0.25)
            (0.1,-0.6) arc (180:360:0.5 and 0.25);
            \begin{scope}[decoration={
                      markings,
                      mark=at position 0.5 with {\arrow{>}}}
                     ] 
            \draw[string-red,thick,line cap= round] (1.1,-0.6) -- (1.1,0.6);
            \draw[string-red,thick,postaction={decorate},line cap= round] (1.1,0.6) arc (0:180:0.5 and 0.25);
            \draw[string-red,thick,line cap= round] (0.1,0.6) arc (180:270:0.5 and 0.25);
            \draw[string-red,thick,postaction={decorate},line cap= round] (0.1,-0.6) arc (180:360:0.5 and 0.25);
            \draw[string-red,thick,line cap= round] (0.1,-0.6) arc (180:90:0.5 and 0.25);
            \draw[string-red,thick,line cap= round] (0.6,-0.35) arc (-90:90:0.2 and 0.35);
            \end{scope}
            \draw[white,line width=1mm](0.4,-0.5) -- (0.4,0.5);
            \draw[draw=string-blue,thick, line cap = round] (0.4,-0.6) -- (0.4,0.6);
            \fill[fill=white] (-0.3,-0.35) rectangle (0.3,0.35);\filldraw[fill=white,draw=string-blue,thick, line cap = round] (-0.6,-0.3) rectangle (0.6,0.3);
            \node at (-0.4,-1.3) [below,string-blue] {\scriptsize $X$};
            \node at (0.4,-1.26) [below,string-blue] {\scriptsize $\overline{X}$};
            \node at (-0.4,1.3) [above,string-blue] {\scriptsize $Y$};
            \node at (0.4,1.3) [above,string-blue] {\scriptsize $\overline{Y}$};
            \node at (0,0) [string-blue] {\scriptsize $g\circ f$};
        \end{tikzpicture}\\
        &=\begin{tikzpicture}[baseline=-0.1cm]
            \draw[draw=string-blue,thick, line cap = round] (0.4,-0.4) -- (0.4,1.3)
            (-0.4,-1.3) -- (-0.4,1.3)
            (0.8,-0.4) -- (0.8,0.9)
            (0.4,-1.3) -- (0.4,-1.2)
            (1.2,0.9) -- (1.2,-0.4)
            (0.4,-1.2) .. controls (0.4,-1) and (1.2,-0.6) .. (1.2,-0.4);
            \begin{scope}[decoration={
                      markings,
                      mark=at position 0.5 with {\arrow{>}}}
                     ] 
            \draw[string-blue,thick,postaction={decorate},line cap= round] (0.8,-0.4) arc (0:-180:0.2);
            \draw[string-blue,thick,postaction={decorate},line cap= round] (1.2,0.9) arc (0:180:0.2);
            \end{scope}
            
            \draw[white,line width=1mm] (0.6,0.6) ellipse (0.4 and 0.2);
            \begin{scope}[decoration={
                      markings,
                      mark=at position 0.25 with {\arrow{>}}}
                     ] 
            \draw[string-red,thick,postaction={decorate}] (0.6,0.6) ellipse (0.4 and 0.2);
            \end{scope}
            \draw[white,line width=1mm] (0.4,0.3) -- (0.4,0.6)
            (0.8,0.3) -- (0.8,0.6);
            \draw[draw=string-blue,thick, line cap = round] (0.4,0.3) -- (0.4,0.6)
            (0.8,0.3) -- (0.8,0.6);
            \filldraw[fill=white,draw=string-blue,thick, line cap = round] (-0.6,-0.3) rectangle (0.6,0.3);
            \node at (-0.4,-1.3) [below,string-blue] {\scriptsize $X$};
            \node at (0.4,-1.26) [below,string-blue] {\scriptsize $\overline{X}$};
            \node at (-0.4,1.3) [above,string-blue] {\scriptsize $Y$};
            \node at (0.4,1.3) [above,string-blue] {\scriptsize $\overline{Y}$};
            \node at (0,0) [string-blue] {\scriptsize $g\circ f$};
        \end{tikzpicture}
        = \zeta \sum_{\alpha=1}^m \begin{tikzpicture}[baseline=-0.1cm]
            \draw[draw=string-blue,thick, line cap = round] (0.4,-0.4) -- (0.4,0.4)
            (-0.4,-1.3) -- (-0.4,1.6)
            (0.8,-0.4) -- (0.8,0.4)
            (0.4,-1.3) -- (0.4,-1.2)
            (1.2,1.2) -- (1.2,-0.4)
            (0.4,1.2) -- (0.4,1.6)
            (0.4,-1.2) .. controls (0.4,-1) and (1.2,-0.6) .. (1.2,-0.4);
            \begin{scope}[decoration={
                      markings,
                      mark=at position 0.5 with {\arrow{>}}}
                     ] 
            \draw[string-blue,thick,postaction={decorate},line cap= round] (0.8,-0.4) arc (0:-180:0.2);
            \draw[string-blue,thick,postaction={decorate},line cap= round] (1.2,1.2) arc (0:180:0.2);
            \end{scope}
            
            \filldraw[fill=white,draw=string-blue,thick, line cap = round] (-0.6,-0.3) rectangle (0.6,0.3);
            \filldraw[fill=white,draw=string-blue,thick, line cap = round] (0.3,0.4) rectangle (0.9,0.7);
            \filldraw[fill=white,draw=string-blue,thick, line cap = round] (0.3,0.9) rectangle (0.9,1.2);
            \node at (-0.4,-1.3) [below,string-blue] {\scriptsize $X$};
            \node at (0.4,-1.26) [below,string-blue] {\scriptsize $\overline{X}$};
            \node at (-0.4,1.6) [above,string-blue] {\scriptsize $Y$};
            \node at (0.4,1.6) [above,string-blue] {\scriptsize $\overline{Y}$};
            \node at (0,0) [string-blue] {\scriptsize $g\circ f$};
            \node at (0.6,0.55) [string-blue] {\scriptsize $a$\tiny${}_\alpha$};
            \node at (0.6,1.05) [string-blue] {\scriptsize $b$\tiny${}_\alpha$};
        \end{tikzpicture}
\end{aligned}
\end{equation}
Where in the last step we used \Cref{lem:cutting} to get $m\geqslant 0 $, $a_\alpha \colon \overline{Y} \otimes \overline{X}^* \to \unit$, and $b_\alpha \colon \unit \to \overline{Y} \otimes \overline{X}^*$.
Thus the map 
\begin{equation}
\begin{aligned}
    \Psi \colon \Hom_\calc(X\otimes \overline{X},k) \otimes_\kk \Hom_\calc(k,Y\otimes \overline{Y}) &\to \Hom_\calc(X,Y)\otimes_\kk \Hom_{\overline{\calc}}(\overline{X},\overline{Y}) \\
     f \otimes_\kk g &\mapsto = \calD \sum_{\alpha=1}^m 
     \begin{tikzpicture}[baseline=-0.1cm]
            \draw[draw=string-blue,thick, line cap = round] (0.4,-0.4) -- (0.4,0.4)
            (-0.4,-1.3) -- (-0.4,1.6)
            (0.8,-0.4) -- (0.8,0.4);
            \begin{scope}[decoration={
                      markings,
                      mark=at position 0.5 with {\arrow{>}}}
                     ] 
            \draw[string-blue,thick,postaction={decorate},line cap= round] (0.8,-0.4) arc (0:-180:0.2);
            \end{scope}
            \filldraw[fill=white,draw=string-blue,thick, line cap = round] (-0.6,-0.3) rectangle (0.6,0.3);
            \filldraw[fill=white,draw=string-blue,thick, line cap = round] (0.3,0.4) rectangle (0.9,0.7);
            \node at (-0.4,-1.3) [below,string-blue] {\scriptsize $X$};
            \node at (-0.4,1.6) [above,string-blue] {\scriptsize $Y$};
            \node at (0,0) [string-blue] {\scriptsize $g\circ f$};
            \node at (0.6,0.55) [string-blue] {\scriptsize $a$\tiny${}_\alpha$};
        \end{tikzpicture}
        \otimes_\kk \begin{tikzpicture}[baseline=-0.1cm]
            \draw[draw=string-blue,thick, line cap = round] 
            (0.4,-1.3) -- (0.4,-1.2)
            (1.2,0.6) -- (1.2,-0.4)
            (0.4,0.6) -- (0.4,1.6)
            (0.4,-1.2) .. controls (0.4,-1) and (1.2,-0.6) .. (1.2,-0.4);
            \begin{scope}[decoration={
                      markings,
                      mark=at position 0.5 with {\arrow{>}}}
                     ] 
            \draw[string-blue,thick,postaction={decorate},line cap= round] (1.2,0.6) arc (0:180:0.2);
            \end{scope}
            
            \filldraw[fill=white,draw=string-blue,thick, line cap = round] (0.3,0.3) rectangle (0.9,0.6);
            \node at (0.4,-1.26) [below,string-blue] {\scriptsize $\overline{X}$};
            \node at (0.4,1.6) [above,string-blue] {\scriptsize $\overline{Y}$};
            \node at (0.6,0.45) [string-blue] {\scriptsize $b$\tiny${}_\alpha$};
        \end{tikzpicture}
\end{aligned}
\end{equation}
makes the diagram 
\[\begin{tikzcd}[ampersand replacement=\&]
	{\Hom_\calc(X\otimes \overline{X},k) \otimes_\kk  \Hom_\calc(k,Y\otimes \overline{Y})} \& {\Hom_\calc(X,Y)\otimes_\kk \Hom_{\overline{\calc}}(\overline{X},\overline{Y})} \\
	{\Hom_\coend(X\otimes \overline{X},k\otimes\coend) \otimes_\kk \Hom_\coend(k\otimes\coend,Y\otimes \overline{Y})} \& {\Hom_\coend(X\otimes\overline{X},Y\otimes\overline{Y})}
	\arrow["\Psi", from=1-1, to=1-2]
	\arrow["{\psi\otimes_\kk\phi}"', from=1-1, to=2-1]
	\arrow["\cong", from=1-2, to=2-2]
	\arrow["\circ", from=2-1, to=2-2]
\end{tikzcd}\]
commute. Next for $a_{\alpha},b_{\alpha}$ as above we define the family of bordisms $M_\alpha^{(X,\overline{X},Y,\overline{Y})}$:
\begin{equation}
            \begin{tikzpicture}
    \draw[thick,loosely dotted,line cap=round](0:2.5 and 0.9) arc (0:180:2.5 and 0.9);
    \fill[white,opacity=0.7] (0,0) circle (2.5);
    \draw[thick,loosely dotted,line cap=round]
    ([yshift=-1.3cm]0:0.5 and 0.2) arc (0:180:0.5 and 0.2)
    ([yshift=1.3cm]0:0.5 and 0.2) arc (0:180:0.5 and 0.2);
    \fill[white,opacity=0.7](0,0) circle (0.5)
    (0,1.3) circle (0.5)
    (0,-1.3) circle (0.5);
    \draw[thick] (0,1.3) circle (0.5)
    (0,-1.3) circle (0.5);
    \draw[thick,line cap=round]
    ([yshift=-1.3cm]0:0.5 and 0.2) arc (0:-180:0.5 and 0.2)
    ([yshift=1.3cm]0:0.5 and 0.2) arc (0:-180:0.5 and 0.2);
    \begin{scope}[thick,decoration={
                      markings,
                      mark=at position 0.57 with {\arrow{>}}}
                     ]  
        \draw[string-blue,thick,postaction={decorate}] (0,1.8) -- (0,2.5);
        \draw[string-blue,thick,postaction={decorate}] (0,-2.5) -- (0,-1.8);
        \draw[string-green,thick,postaction={decorate}] (0,-0.8) -- (0,0.8);
    \end{scope}
    \begin{scope}[thick,decoration={
                      markings,
                      mark=at position 0.5 with {\arrow{>}}}
                     ]  
                     \draw[string-blue,thick,postaction={decorate}] (1.6,0.3) .. controls (1.6,0) and  ([xshift=0.3cm, yshift=-1.3cm]30:0.5).. ([yshift=-1.3cm]30:0.5);
                     \draw[string-blue,thick,postaction={decorate}] (0.5,0.3) arc (-180:0:0.3 and 0.2); 
    \end{scope}
    \draw[string-blue,thick] ([yshift=1.3cm]330:0.5) .. controls ([xshift=0.1cm,yshift=1.3cm]330:0.5) and (0.5,0.5) .. (0.5,0.3);
    \draw[ultra thick,white](2.5,0) arc (0:-180:2.5 and 0.9);
    \draw[thick] (0,0) ellipse (2.5)
    (2.5,0) arc (0:-180:2.5 and 0.9);
    \fill[string-green](0,-0.8) circle (0.04)
     (0,0.8) circle (0.04);
     \fill[string-blue]
     (0,2.5) circle (0.04)
     (0,1.8) circle (0.04)
     (0,-1.8) circle (0.04)
     (0,-2.5) circle (0.04)
     ([yshift=1.3cm]330:0.5) circle (0.04)
     ([yshift=-1.3cm]30:0.5) circle (0.04);
     \filldraw[fill=white,draw=string-blue,thick] (1,0.3) rectangle (1.7,0.7);
    \node at (0.05,0.05) [right,string-green] {\scriptsize $k$};
    \node at (0,2.25) [string-blue,right] {\scriptsize $Y$};
    \node at (0.5,-0.15) [string-blue,right] {\scriptsize $\overline{Y}$};
    \node at (0,-2.25) [string-blue,right] {\scriptsize $X$};
    \node at (1.2,-0.5) [string-blue,right] {\scriptsize $\overline{X}$};
    \node at (1.37,0.5) [string-blue] {\scriptsize $a$\tiny${}_\alpha$ };
    %
    %
    \draw[thick,loosely dotted,line cap=round]([xshift=4cm]0:1.3 and 0.5) arc (0:180:1.3 and 0.5);
    \fill[white,opacity=0.7] (4,0) circle (1.3);
    \filldraw[fill=white,draw=string-blue,thick] (3.65,-0.2) rectangle (4.35,0.2);
     \begin{scope}[thick,decoration={
                      markings,
                      mark=at position 0.5 with {\arrow{>}}}
                     ]  
        \draw[string-blue,thick,postaction={decorate}] (3.75,0.2) .. controls (3.75,0.6) and (4,1) .. (4,1.3);
        \draw[string-blue,thick,postaction={decorate}] (4.85,0.2) arc (0:180:0.3 and 0.4);
    \end{scope}
    \draw[string-blue,thick] (4,-1.3) .. controls (4,-1) and (4.85,-0.2) .. (4.85,0.2);
    \draw[white,ultra thick] ([xshift=4cm]0:1.3 and 0.5) arc (0:-180:1.3 and 0.5);
    \draw[thick] (4,0) circle (1.3)
    ([xshift=4cm]0:1.3 and 0.5) arc (0:-180:1.3 and 0.5);
    \fill[string-blue]
     (4,1.3) circle (0.04)
     (4,-1.3) circle (0.04);
     \node at (4.1,0.8) [string-blue] {\scriptsize $\overline{Y}$};
    \node at (4.5,-0.8) [string-blue] {\scriptsize $\overline{X}$};
    \node at (4,0) [string-blue] {\scriptsize $b$\tiny${}_\alpha$ };
     %
     %
     %
     \begin{scope}[shift={(7.5,0)}]
    \draw[thick] (0.9,0) circle (0.5)
    (-0.9,0) circle (0.5);
    \draw[thick,loosely dotted,line cap=round](0.4,0) arc (180:0:0.5 and 0.2)
    (-0.4,0) arc (0:180:0.5 and 0.2); 
    \fill[white,opacity=0.7] (-0.4,0) arc (0:180:0.5)
    (1.4,0) arc (0:180:0.5);
    \draw[thick] (-0.4,0) arc (0:180:0.5)
    (1.4,0) arc (0:180:0.5)
    (0.4,0) arc (180:360:0.5 and 0.2)
    (-0.4,0) arc (0:-180:0.5 and 0.2);
    \draw[thick] (-0.4,0) arc (0:180:0.5);
    \fill[string-blue] ([xshift=0.9cm]60:0.5) circle (0.04)
    ([xshift=-0.9cm]240:0.5) circle (0.04)
    ([xshift=0.9cm]120:0.5) circle (0.04)
    ([xshift=-0.9cm]300:0.5) circle (0.04);
    \fill[string-green](0.9,-0.5) circle (0.04)
     (-0.9,0.5) circle (0.04);
     \node at (0.9,-0.45) [string-green,below] {\scriptsize $(k,-)$};
     \node at ([xshift=0.9cm]70:0.4) [string-blue,above right] {\scriptsize $(\overline{Y},+)$};
     \node at ([xshift=0.9cm]110:0.4) [string-blue,above left] {\scriptsize $(Y,+)$};
     \node at (-0.9,0.5) [string-green,above] {\scriptsize $(k,+)$};
     \node at ([xshift=-0.9cm]300:0.4) [string-blue,below right] {\scriptsize $(\overline{X},-)$};
     \node at ([xshift=-0.9cm]250:0.4) [string-blue,below left] {\scriptsize $(X,-)$};
     \end{scope}
     %
     %
     %
    \begin{scope}[shift={(11.8,0)}]
    \draw[thick] (0.5,0) circle (0.5)
    (-0.9,0) circle (0.5);
    \draw[thick,loosely dotted,line cap=round](0,0) arc (180:0:0.5 and 0.2)
    (-0.4,0) arc (0:180:0.5 and 0.2); 
    \fill[white,opacity=0.7] (-0.4,0) arc (0:180:0.5)
    (1,0) arc (0:180:0.5);
    \draw[thick] (-0.4,0) arc (0:180:0.5)
    (1,0) arc (0:180:0.5)
    (0,0) arc (180:360:0.5 and 0.2)
    (-0.4,0) arc (0:-180:0.5 and 0.2);
    \draw[thick] (-0.4,0) arc (0:180:0.5);
    \fill[string-blue](-0.9,-0.5) circle (0.04)
     (-0.9,0.5) circle (0.04)
     (0.5,-0.5) circle (0.04)
     (0.5,0.5) circle (0.04);
     \node at (-0.9,-0.45) [string-blue,below] {\scriptsize $(X,-)$};
     \node at (-0.9,0.5) [string-blue,above] {\scriptsize $(Y,+)$};
     \node at (0.5,-0.45) [string-blue,below] {\scriptsize $(\overline{X},-)$};
     \node at (0.5,0.5) [string-blue,above] {\scriptsize $(\overline{Y},+)$};
    \end{scope}
    \draw[thick,->,line cap=round] (9.1,0) -- (10.2,0);
    \fill[black] (5.7,-0.07) circle (0.025)
    (5.7,0.07) circle (0.025);
    \end{tikzpicture}.
\end{equation}
By construction this family satisfies $\Psi = \calD\sum_{\alpha=1}^m \widehat{\rmV}_\calc\left(M_\alpha^{(X,\overline{X},Y,\overline{Y})}\right)$. With this we can reformulate commutativity of Diagram \ref{diag:circlepairing} to the equality
\begin{equation}\label{eq:tftequality}
   \widehat{\rmV}_\calc\left(M_{\id_{S^1_k}}^{(X,\overline{X},Y,\overline{Y})}\right) = \calD \sum_{\alpha=1}^m \widehat{\rmV}_\calc\left(M_\alpha^{(X,\overline{X},Y,\overline{Y})}\right)
\end{equation}
of linear maps. To show this equality we will compare the corresponding matrix elements. Recall from \autocite[Prop.\,4.17]{DGGPR19}, see also \autocite[Sec.\,3.4]{HR2024modfunc} for our conventions, that the state space $\widehat{\rmV}_\calc\left(- M_{S^1_k}^{(X,\overline{X})} \sqcup M_{S^1_k}^{(Y,\overline{Y})}\right)$ is spanned by bordisms $B_{f,g} \colon \varnothing \to- M_{S^1_k}^{(X,\overline{X})} \sqcup M_{S^1_k}^{(Y,\overline{Y})}$ with underlying $3$-manifold the disjoint union of two $3$-balls and embedded ribbon graphs given by $f\in \Hom_\calc(X\otimes \overline{X},k) $ and $g \in \Hom_\calc(k,Y\otimes \overline{Y})$. For the dual state space $ \widehat{\rmV}_\calc\left(\widehat{\id_{S^1_k}}^{(X,\overline{X},Y,\overline{Y})}\right)^*$
recall from \autocite[Prop.\,4.11]{DGGPR19} that it is generated by admissible bichrome graphs $T$, i.e.\ $T$ containing at least one blue edge labelled with a projective object of $\calc$, inside a fixed connected $3$-manifold viewed as a bordism $N_T \colon \widehat{\id_{S^1_k}}^{(X,\overline{X},Y,\overline{Y})} \to \varnothing$.\footnote{Note that in our conventions the TFT $\widehat{\rmV}_\calc$ corresponds to the dual TFT $\check{\rmV}_\calc'$ in \autocite{DGGPR19} and vice versa.} In particular, we can choose the manifold underlying $N_T$ to be given by $S^2\times I$. With this we are lead to study the invariants of the closed $3$-manifolds with admissible ribbon graphs given by:
\begin{equation}
   A = N_T \bigsqcup_{\widehat{\id_{S^1_k}}^{(X,\overline{X},Y,\overline{Y})}}M_{\id_{S^1_k}}^{(X,\overline{X},Y,\overline{Y})} \bigsqcup_{- M_{S^1_k}^{(X,\overline{X})} \sqcup M_{S^1_k}^{(Y,\overline{Y})}} B_{f,g}
\end{equation}
and
\begin{equation}
  B_\alpha = N_T \bigsqcup_{\widehat{\id_{S^1_k}}^{(X,\overline{X},Y,\overline{Y})}} M_\alpha^{(X,\overline{X},Y,\overline{Y})} \bigsqcup_{- M_{S^1_k}^{(X,\overline{X})} \sqcup M_{S^1_k}^{(Y,\overline{Y})}} B_{f,g}.
\end{equation}
First note that the manifolds underlying $A$ and $B_\alpha$ are given by $S^2\times S^1$ and $S^3$, respectively, and thus differ by index 1 surgery. In particular the two ribbons running along the core of the extra $1$-handle of $A$ can be chosen to be the ones labelled with $\overline{X}$ and $\overline{Y}$, so we find the following bichrome graph presentation of $A$ in $S^3$:
\begin{equation}
    \begin{tikzpicture}[baseline=-0.1cm]
            \draw[draw=string-blue,thick, line cap = round] (0.4,0) -- (0.4,1.2)
            (0.8,0) -- (0.8,1.2);
            
            \draw[white,line width=1mm] (0.6,0.6) ellipse (0.4 and 0.2);
            \begin{scope}[decoration={
                      markings,
                      mark=at position 0.25 with {\arrow{>}}}
                     ] 
            \draw[string-red,thick,postaction={decorate}] (0.6,0.6) ellipse (0.4 and 0.2);
            \end{scope}
            \begin{scope}[decoration={
                      markings,
                      mark=at position 0.5 with {\arrow{>}}}
                     ] 
            \draw[string-blue,thick,postaction={decorate}] (1.6,1.2) arc (0:180:0.4 and 0.3);
            \draw[string-blue,thick,postaction={decorate}] (0.4,1.2) arc (180:0:1 and 0.8);
            \draw[string-blue,thick,postaction={decorate}] (0.8,0) arc (-180:0:0.4 and 0.3);
            \draw[string-blue,thick,postaction={decorate}] (2.4,0) arc (0:-180:1 and 0.8);
            \end{scope}
            \draw[string-blue,thick,line cap=round] (1.6,0) -- (1.6,1.2)
            (2.4,0) --(2.4,1.2);
            \draw[white,line width=1mm] (0.4,0.3) -- (0.4,0.6)
            (0.8,0.3) -- (0.8,0.6);
            \draw[draw=string-blue,thick, line cap = round] (0.4,0.3) -- (0.4,0.6)
            (0.8,0.3) -- (0.8,0.6);
            \filldraw[fill=white,draw=string-blue,thick, line cap = round] (1.2,0.3) rectangle (2.8,0.9);
            \node at (2,0.6) [string-blue] {\scriptsize $T \cup (g \circ f)$};
            \node at (1.8,1.1) [string-blue] {\scriptsize $\overline{X}$};
            \node at (2.6,1.1) [string-blue] {\scriptsize $\overline{Y}$};
        \end{tikzpicture}
\end{equation}
where $T\cup (g\circ f)$ contains the admissible graph $T$ and the graph corresponding to $(g\circ f)$.
Analogously for $B_{\alpha}$ we have:
\begin{equation}
    \begin{tikzpicture}[baseline=-0.1cm]            
            \begin{scope}[decoration={
                      markings,
                      mark=at position 0.5 with {\arrow{>}}}
                     ] 
            \draw[string-blue,thick,postaction={decorate}] (1.6,1.2) arc (0:180:0.4 and 0.3);
            \draw[string-blue,thick,postaction={decorate}] (0.4,1.2) arc (180:0:1 and 0.8);
            \draw[string-blue,thick,postaction={decorate}] (0.8,0) arc (-180:0:0.4 and 0.3);
            \draw[string-blue,thick,postaction={decorate}] (2.4,0) arc (0:-180:1 and 0.8);
            \end{scope}
            \draw[string-blue,thick,line cap=round] (1.6,0) -- (1.6,1.2)
            (2.4,0) --(2.4,1.2);
            \filldraw[fill=white,draw=string-blue,thick, line cap = round] (1.2,0.3) rectangle (2.8,0.9);
            \filldraw[fill=white,draw=string-blue,thick, line cap = round] (0.3,0) rectangle (0.9,0.4);
            \filldraw[fill=white,draw=string-blue,thick, line cap = round] (0.3,0.8) rectangle (0.9,1.2);
            \node at (2,0.6) [string-blue] {\scriptsize $T \cup (g \circ f)$};
            \node at (1.8,1.1) [string-blue] {\scriptsize $\overline{X}$};
            \node at (2.6,1.1) [string-blue] {\scriptsize $\overline{Y}$};
            \node at (0.6,0.2) [string-blue] {\scriptsize $a$\tiny${}_\alpha$};
            \node at (0.6,1) [string-blue] {\scriptsize $b$\tiny${}_\alpha$};
        \end{tikzpicture}
\end{equation}
By \Cref{lem:cutting} and since the surgery link for $B$ has one less component then the one for $A$ an analogous argument as in the proof of \autocite[Prop.\,4.10]{DGGPR19} leads to
\begin{equation}
     \calD \widehat{\rmV}_\calc(A) = \zeta \sum_{\alpha=^1}^m \widehat{\rmV}_\calc(B_\alpha).
\end{equation}
Since $\zeta = \calD^2$ we get
\begin{equation}
   \widehat{\rmV}_\calc\left(M_{\id_{S^1_k}}^{(X,\overline{X},Y,\overline{Y})}\right) = \calD \sum_{\alpha=1}^m \widehat{\rmV}_\calc\left(M_\alpha^{(X,\overline{X},Y,\overline{Y})}\right).
\end{equation}
as desired.

\subsection{Some CFT quantities}\label{sec:cftstuff}
In this section we compute some more quantities of physical interest including boundary states and annulus amplitudes for different boundary conditions and compare our results to the ones proposed in \autocite{FGSS18cardy}. Moreover, we will compute the torus partition function as well as the action of line defects on bulk fields.
\subsubsection{Boundary states}\label{subsec:boundarystates}
Boundary states are bulk one point correlators on a disc with fixed boundary condition. In our formulation these correspond to the correlators for a world sheet $\frakD$ with underlying surface a cylinder such that one boundary is a gluing boundary and the other one a free boundary labelled with a boundary condition $n \in \calc$. If the gluing boundary is outgoing, i.e.\ $\frakD^{\mathrm{out}}_n \colon \varnothing \to S^1$ we will call it an \emph{outgoing} boundary state, and for an incoming gluing boundary $\frakD^{\mathrm{in}}_n \colon S^1\to \varnothing$ an \emph{incoming} boundary state.

Let $(X,\overline{X}) \in \calc\times\overline{\calc}$, the connecting bordism $M_{\frakD_n}^{(X,\overline{X})} \colon M_{S^1}^{(X,\overline{X})} \to \widehat{\frakD_n}^{(X,\overline{X})}$ of $\frakD^{\mathrm{out}}_n \colon \varnothing \to S^1$ is given by 
\begin{equation}
            \begin{tikzpicture}
    \draw[thick,loosely dotted,line cap=round]
    (0.6,0) arc (0:180:1.5 and 0.75); 
    \fill[white,opacity=0.7] (0.6,0) arc (0:180:1.5);
    \begin{scope}[thick,decoration={
                      markings,
                      mark=at position 0.57 with {\arrow{>}}}
                     ]  
        \draw[string-blue,thick,postaction={decorate}] (-0.9,-0.5) -- (-0.9,-1.5);
    \end{scope}
    \draw[white,ultra thick] (-0.9,0) ellipse (1.2 and 0.6);
    \begin{scope}[thick,decoration={
                      markings,
                      mark=at position 0.495 with {\arrow{>}}}
                     ]  
        \draw[string-violet,thick,postaction={decorate}] (-0.9,0) ellipse (1.2 and 0.6);
    \end{scope}
    \draw[white,ultra thick] (-0.9,0.5) -- (-0.9,1.5);
    \begin{scope}[thick,decoration={
                      markings,
                      mark=at position 0.57 with {\arrow{>}}}
                     ]  
        \draw[string-blue,thick,postaction={decorate}] (-0.9,0.5) -- (-0.9,1.5);
    \end{scope}
    \draw[white, ultra thick](0.6,0) arc (0:-180:1.5 and 0.75);
    \draw[thick](0.6,0) arc (0:-180:1.5 and 0.75);
    \draw[thick] (-0.9,0) circle (1.5);
    \draw[thick] (-0.4,0) arc (0:180:0.5);
    \draw[thick] (-0.9,0) circle (0.5);
    \draw[thick,loosely dotted,line cap=round]
    (-0.4,0) arc (0:180:0.5 and 0.2); 
    \fill[white,opacity=0.7] (-0.4,0) arc (0:180:0.5);
    \draw[thick] (-0.4,0) arc (0:180:0.5)
    (-0.4,0) arc (0:-180:0.5 and 0.2);
    \draw[thick] (-0.4,0) arc (0:180:0.5);
    \fill[string-blue](-0.9,-0.5) circle (0.04)
     (-0.9,0.5) circle (0.04)
     (-0.9,-1.5) circle (0.04)
     (-0.9,1.5) circle (0.04);
     \node at  (-0.85,-1) [string-blue,right] {\scriptsize $\overline{X}$}; 
     \node at  (-0.85,1) [string-blue,right] {\scriptsize $X$};
     \node at  (0.33,0) [string-violet,left] {\scriptsize $n$};
     %
     %
     %
     \begin{scope}[shift={(2.5,0)}]
    \draw[thick] (-0.9,0) circle (0.5);
    \draw[thick,loosely dotted,line cap=round]
    (-0.4,0) arc (0:180:0.5 and 0.2); 
    \fill[white,opacity=0.7] (-0.4,0) arc (0:180:0.5);
    \draw[thick] (-0.4,0) arc (0:180:0.5)
    (-0.4,0) arc (0:-180:0.5 and 0.2);
    \draw[thick] (-0.4,0) arc (0:180:0.5);
    \fill[string-blue](-0.9,-0.5) circle (0.04)
     (-0.9,0.5) circle (0.04);
     \node at (-0.9,-0.45) [string-blue,below] {\scriptsize $(\overline{X},+)$};
    \node at (-0.9,0.45) [string-blue,above] {\scriptsize $(X,+)$};
     \end{scope}
     %
     %
     %
    \begin{scope}[shift={(4.9,0)}]
    \draw[thick] (-0.9,0) circle (0.5);
    \draw[thick,loosely dotted,line cap=round]
    (-0.4,0) arc (0:180:0.5 and 0.2); 
    \fill[white,opacity=0.7] (-0.4,0) arc (0:180:0.5);
    \draw[thick] (-0.4,0) arc (0:180:0.5)
    (-0.4,0) arc (0:-180:0.5 and 0.2);
    \draw[thick] (-0.4,0) arc (0:180:0.5);
    \fill[string-blue](-0.9,-0.5) circle (0.04)
     (-0.9,0.5) circle (0.04);
     \node at (-0.9,-0.45) [string-blue,below] {\scriptsize $(\overline{X},+)$};
    \node at (-0.9,0.45) [string-blue,above] {\scriptsize $(X,+)$};
     \end{scope}
    \draw[thick,->,line cap=round] (2.3,0) -- (3.3,0);
    \fill[black] (0.8,-0.07) circle (0.025)
    (0.8,0.07) circle (0.025);
    \end{tikzpicture}
\end{equation}
The connecting bordism for $\frakD^{\mathrm{in}}_n \colon S^1\to \varnothing$ differs from the one illustrated above only in the orientation of the $(X,\overline{X})$-labelled ribbons. We will only discuss the outgoing case in detail since the ingoing case can be treated completely analogously.

Applying the TFT $\widehat{\rmV}_\calc$ to $M_{\frakD_n}^{(X,\overline{X})}$ and using the isomorphism between TFT state spaces and morphism spaces in $\calc$ we get the linear map
\begin{equation}
    \begin{aligned}
    \Hom_\calc(\unit,X\otimes\overline{X}) &\to \Hom_\calc(\unit,X\otimes\overline{X}) \\
        \begin{tikzpicture}[baseline=1cm]
            \filldraw[fill=white,draw=string-blue,thick, line cap = round] (-0.6,-0.1) rectangle (0.6,0.4);
            \draw[draw=string-blue,thick, line cap = round] (0.4,0.4) -- (0.4,1.6)
            (-0.4,0.4) -- (-0.4,1.6);
            \node at (-0.4,1.6) [above,string-blue] {\scriptsize $X$};
            \node at (0.4,1.6) [above,string-blue] {\scriptsize $\overline{X}$};
        \end{tikzpicture}
        &\mapsto 
        \begin{tikzpicture}[baseline=1cm]
            \filldraw[fill=white,draw=string-blue,thick, line cap = round] (-0.6,-0.1) rectangle (0.6,0.4);
            \draw[draw=string-blue,thick, line cap = round] (0.4,0.4) -- (0.4,1.6)
            (-0.4,0.4) -- (-0.4,1.6);
            \draw[draw=white,double=string-red,very thick] (-0.6,1) circle (0.4);
            \begin{scope}[decoration={
                      markings,
                      mark=at position 0.52 with {\arrow{>}}}
                     ] 
            \draw[string-violet,thick,postaction={decorate},line cap= round] (-0.2,1) arc (0:180:0.4);
            \draw[string-violet,thick,postaction={decorate},line cap= round] (-1,1) arc (180:360:0.4);
            \end{scope}
            \draw[draw=white,double=string-blue,very thick, line cap = round] (-0.4,1.2) -- (-0.4,1.6);
            \draw[string-blue,thick, line cap = round] (-0.4,1) -- (-0.4,1.6);
            \node at (-0.4,1.6) [above,string-blue] {\scriptsize $X$};
            \node at (0.4,1.6) [above,string-blue] {\scriptsize $\overline{X}$};
            \node at (-1,1) [left,string-violet] {\scriptsize $n$};
        \end{tikzpicture}
    \end{aligned}.
\end{equation}
Next to get $\Cor_{\frakD^{\mathrm{out}}_n}$ we use the adjunction $\Hom_\calc(\unit,X\otimes\overline{X}) \cong \Hom_\coend(\coend,X\otimes\overline{X})$, where we view $X\otimes\overline{X}$ as an $L$-module with it's canonical $L$-module structure constructed from the half braiding on $X\otimes\overline{X}$ as in \Cref{lem:center-modules} to obtain a linear map 
\begin{align}
\Hom_\coend(\coend,X\otimes\overline{X}) \to \Hom_\coend(\coend,X\otimes\overline{X})
\end{align}
which we will also call $\widehat{\rmV}_\calc(M_{D_n}^{(X,\overline{X})})$.\footnote{It is important here that $\overline{X}$ is an object in $\overline{\calc}$, i.e.\ we use the mirrored braiding on $\overline{X}$ to define the $\coend$-action.}
By definition $\Cor_{\frakD^{\mathrm{out}}_n}\in \mathrm{Nat}(\Cor_{S^1},\Bl_\calc(\frakD^{\mathrm{out}}_n))$ is induced by the natural transformation with components $\widehat{\rmV}_\calc(M_{\frakD_n}^{(X,\overline{X})})$. In order to compare this with the proposed boundary states of \autocite{FGSS18cardy} we need to make this abstractly defined $\Cor_{\frakD^{\mathrm{out}}_n}$ more concrete. To this end, consider the linear isomorphisms
\begin{equation}
    \begin{aligned}
        \mathrm{Nat}(\Cor_{S^1},\Bl_\calc(\frakD^{\mathrm{out}}_n)) &\cong \mathrm{Nat}(\Hom_\coend(\coend,-),\Hom_\coend(\coend,-)) \\
        &\cong \Hom_\coend(\coend,\coend) \\
        &\cong \Hom_\calc(\unit,\coend)
    \end{aligned}
\end{equation}
where in the first step we used $\Cor_{S^1} \cong \Hom_\coend(\coend,-) \cong \Bl_\calc(\frakD^{\mathrm{out}}_n)$, the Yoneda lemma in the second, and the free-forgetful adjunction $\Hom_\coend(\coend,\coend) \cong \Hom_\calc(\unit,\coend)$ in the last one. The space $\Hom_\calc(\unit,\coend)$ is precisely the one also considered in \autocite[Sec.\,3.2]{FGSS18cardy}. Unfortunately, we cannot simply compute the image of $\Cor_{\frakD^{\mathrm{out}}_n}$ under this chain of isomorphisms because we do not have direct access to the $L$-component of the natural transformation which is needed for the Yoneda lemma. Instead we will work our way backwards by showing that the outgoing boundary state proposed in \autocite[Sec.\,3.2]{FGSS18cardy} induces the natural transformation $\Cor_{\frakD^{\mathrm{out}}_n}$. According to \autocite[Sec.\,3.2]{FGSS18cardy} an outgoing boundary state should be described by the \emph{cocharacter} $\check{\chi}_n = \iota_n \circ \lcoev_n \in \Hom_\calc(\unit,\coend)$ 
of $n$ from \Cref{subsec:reptheory}. Following the isomorphism above the cocharcter $\check{\chi}_n$ gets sent to the natural transformation with components 
\begin{equation}
    \begin{aligned}
       \Phi_n^{Y}\colon \Hom_\coend(\coend,Y) &\to \Hom_\coend(\coend,Y) \\
        f &\mapsto f \circ \mu \circ (\check{\chi}_n \otimes \id_\coend)
    \end{aligned}
\end{equation}
for $Y \in \calc_\coend$.
It now follows from a straightforward calculation that $\Phi_n^{X\otimes\overline{X}} = \widehat{\rmV}_\calc(M_{D^{\mathrm{out}}_n}^{(X,\overline{X})})$, thus they both induce $\Cor_{\frakD^{\mathrm{out}}_n}$. As mentioned above, an analogous computation can be done for an incoming boundary state $\Cor_{\frakD^{\mathrm{in}}_n}$. In this case we obtain a \emph{character} $\chi_n \in \Hom_\calc(\coend,\unit)$ precomposed with the modular $S$-transformation $\calS \colon \coend \to \coend$ from \Cref{eq:def-S-endo-of-L}. With this we confirm the relation between boundary states and (co)characters postulated in \autocite[Sec.\,3.2]{FGSS18cardy}.

\subsubsection{Annulus amplitude}
Next we compute the annulus amplitude in the form of the correlator for an annulus $\frakA_{n,m}$ with boundary conditions $m,n\in \calc$.
The connecting bordism of $\frakA_{n,m}$ is given by 
\begin{equation}
    M_{\frakA_{m,n}} = 
    \begin{tikzpicture}[baseline=0cm]
    \draw[thick] (0,0) ellipse (1.5 and 1)
           (0.5,0.06) arc (-10:-170:0.5 and 0.3)
           (0.4,-0.06) arc (10:170:0.4 and 0.25);
         \begin{scope}[decoration={
                      markings,
                      mark=at position 0.5 with {\arrow{<}}}
                     ] 
            \draw[string-violet,thick,postaction={decorate}] (0,0) ellipse (0.9 and 0.6);
        \end{scope}
        \begin{scope}[decoration={
                      markings,
                      mark=at position 0.5 with {\arrow{>}}}
                     ] 
            \draw[string-violet,thick,postaction={decorate}] (0,0) ellipse (1.2 and 0.8);
        \end{scope}
        \node at (0.95,0) [string-violet,left] {\scriptsize $m$};
        \node at (1.1,0) [string-violet,right] {\scriptsize $n$};
    \end{tikzpicture}
    \colon \varnothing \longrightarrow 
    \begin{tikzpicture}[baseline=0cm]
    \draw[thick] (0,0) ellipse (1.5 and 1)
           (0.5,0.06) arc (-10:-170:0.5 and 0.3)
           (0.4,-0.06) arc (10:170:0.4 and 0.25);
    \end{tikzpicture}
\end{equation}
To compute $\widehat{\rmV}_\calc(M_{\frakA_{m,n}})$ we employ \autocite[Eq.\,4.26]{HR2024modfunc} again, this time for $X = m^* \otimes n$. 
We find that $\widehat{\rmV}_\calc(M_{\frakA_{m,n}}) \in \widehat{\rmV}_\calc(T^2)$ corresponds to $\check{\chi}_{m^*\otimes n} \in \Hom_\calc(\unit,\coend)$ under the isomorphisms $\widehat{\rmV}_\calc(T^2) \cong \Hom_\calc(\coend,\unit) \cong \Hom_\calc(\unit,\coend)$ where we used the Frobenius algebra structure on $\coend$ for the second one. This again reproduces the results of \autocite[Sec.\,4]{FGSS18cardy}. To be more precise we recover their result for the \emph{open-string channel}. The \emph{closed-string channel} can be obtained by performing a modular $S$-transformation.

\subsubsection{Torus partition function}
For the partition function, in the form of the torus correlator $\Cor_{T^2}$, we have to compute $\widehat{\rmV}_\calc(T^2\times I) \colon \widehat{\rmV}_\calc(\varnothing) \to \widehat{\rmV}_\calc(T^2\sqcup -T^2)$. In the semisimple setting one can compute the matrix coefficients of $\widehat{\rmV}_\calc(T^2\times I)$ by pairing with a dual basis of $\widehat{\rmV}_\calc(T^2)$ given by ribbons along the non-contractible cycle in a solid torus \autocite[Sec.\,5.3]{FRSI}. In the non-semisimple setting this approach cannot be used because $\widehat{\rmV}_\calc$ is defined on a non-rigid bordism category which results in not having a canonical basis for the dual state space. 

Instead we will compute $\Cor_{T^2}$ using factorisation of correlators in the form of Diagram \ref{diag:horinatural}. Let us denote with $\mathrm{ev}_{S^1} \colon S^1 \sqcup -S^1 \to \varnothing$ and $\mathrm{coev}_{S^1} \colon \varnothing \to S^1 \sqcup -S^1$ the cylinder $S^1\times I$ viewed as the evaluation and coevaluation morphism of $S^1$, respectively. With this we have $T^2 = \mathrm{ev}_{S^1} \sqcup_{ S^1 \sqcup -S^1} \mathrm{coev}_{S^1}$ and thus 
\begin{equation}
\Cor_{T^2} = (\Cor_{\mathrm{ev}_{S^1}} \diamond \Cor_{\mathrm{coev}_{S^1}}) \circ \eta_{\Cor_{ S^1 \sqcup -S^1}}.
\end{equation}

Using monoidality of $\widehat{\rmV}_\calc$ it is straightforward to check that $\eta_{\Cor_{S^1 \sqcup -S^1}} = \eta_{\Cor_{S^1}}\otimes_\kk \eta_{\Cor_{-S^1}}$. To get $\Cor_{\mathrm{ev}_{S^1}} \diamond \Cor_{\mathrm{coev}_{S^1}}$ we have to consider the corresponding version of Diagram \ref{diag:corgluing}. For this we need to first compute $\Cor_{\mathrm{ev}_{S^1}}^{(X,\overline{X},Y,\overline{Y})}$ and $\Cor_{\mathrm{coev}_{S^1}}^{(X,\overline{X},Y,\overline{Y})}$ for $(X,\overline{X}),(Y,\overline{Y})\in\calc\times\overline{\calc}$. Instead of computing $\Cor_{\mathrm{ev}_{S^1}}^{(X,\overline{X},Y,\overline{Y})}$ and $\Cor_{\mathrm{coev}_{S^1}}^{(X,\overline{X},Y,\overline{Y})}$ by evaluating with the TFT, as we did before, we can use that the only difference between the connecting manifolds of $\mathrm{ev}_{S^1}$, $\mathrm{coev}_{S^1}$, and $\id_{S^1}$ is the orientation of some of the ribbons. Algebraically changing the orientation of a ribbon corresponds to exchanging the labelling object with its dual, this implies $\Cor_{\mathrm{ev}_{S^1}}^{(X,\overline{X},Y,\overline{Y})} = \Cor_{\mathrm{id}_{S^1}}^{(X,\overline{X},Y^*,\overline{Y}^*)}$ as well as $\Cor_{\mathrm{coev}_{S^1}}^{(X,\overline{X},Y,\overline{Y})} = \Cor_{\mathrm{id}_{S^1}}^{(X^*,\overline{X}^*,Y,\overline{Y})}$. Since we already computed $\Cor_{\id_{S^1}}$ in \Cref{sec:twopointcor} we can use this to continue. 

The vector spaces in Diagram \ref{diag:corgluing} are isomorphic to the following $\Hom$-spaces in $\calc_\coend$
\begin{align*}
    \Cor_{S^1}(-) \cong \Hom_\coend(\coend,-)&, \quad \Cor_{-S^1}(-) \cong \Hom_\coend(\coend,(-)^*) \\
    \Bl_\calc(\mathrm{ev}_{S^1})(-,-) \cong \Hom_\coend(-,(-)^*)&, \quad \Bl_\calc(\mathrm{coev}_{S^1})(-,-) \cong \Hom_\coend((-)^*,-) \\
    \Bl_\calc(T^2) \cong \Hom_\coend(\unit,\mathbbm{L})&
\end{align*}
 where $\mathbbm{L} = \int^{Z \in \calc_\coend} Z^*\otimes Z$ is the canonical coend in $\calc_\coend \simeq \calc\boxtimes\overline{\calc}$. Under these isomorphism the relevant gluing maps are given by

\begin{equation}
    \begin{aligned}
        i^{(X,\overline{X})} \colon \Hom_\coend(X\otimes\overline{X},\coend) \otimes_\kk \Hom_\coend((X\otimes\overline{X})^*,\coend) &\to \Hom_\coend(\unit,\coend\otimes\coend) \\
        h \otimes_\kk k &\mapsto (h \otimes k)\circ \mathrm{ev}_{(X\otimes\overline{X})},
    \end{aligned}
\end{equation}
\begin{equation}
    \begin{aligned}
        \Tilde{i}^{(Y,\overline{Y})} \colon \Hom_\coend(\coend,Y\otimes\overline{Y}) \otimes_\kk \Hom_\coend(\coend,(Y\otimes\overline{Y})^*) &\to \Hom_\coend(\coend\otimes\coend,\unit) \\
        l \otimes_\kk m &\mapsto \mathrm{coev}_{(X\otimes\overline{X})}\circ(l \otimes m),
    \end{aligned}
\end{equation}
and 
\begin{equation}
    \begin{aligned}
        I^{(X,\overline{X},Y,\overline{Y})}\colon\Hom_\coend(X\otimes\overline{X},(Y\otimes\overline{Y})^*) &\otimes_\kk \Hom_\coend((X\otimes\overline{X})^*,Y\otimes\overline{Y}) \to \Hom_\coend(\unit,\mathbbm{L}) \\
        f &\otimes_\kk g \mapsto \iota_{(Y\otimes\overline{Y})}\circ(f \otimes g)\circ \mathrm{ev}_{(X\otimes\overline{X})},
    \end{aligned}
\end{equation}
where $\iota_{(Y\otimes\overline{Y})} \colon (Y\otimes\overline{Y})^*\otimes (Y\otimes\overline{Y}) \to \mathbbm{L}$ is the universal dinatural morphism in $\calc_\coend$.
Commutativity of Diagram \ref{diag:corgluing} thus corresponds to finding the unique linear map
\begin{equation}
       \Cor_{\mathrm{ev}_{S^1}} \diamond \Cor_{\mathrm{coev}_{S^1}} \colon \Hom_\coend(\coend\otimes\coend,\unit) \otimes_\kk \Hom_\coend(\unit,\coend\otimes\coend) \to \Hom_\coend(\unit,\mathbbm{L}) 
\end{equation}
such that 
\begin{equation}\label{eq:torusgluecond}
    \Cor_{\mathrm{ev}_{S^1}} \diamond \Cor_{\mathrm{coev}_{S^1}} \circ (\iota^{(X,\overline{X})} \otimes_\kk \Tilde{\iota}^{(Y,\overline{Y})}) = I^{(X,\overline{X},Y,\overline{Y})} \circ (\Cor_{\mathrm{id}_{S^1}}^{(X^*,\overline{X}^*,Y,\overline{Y})} \otimes_\kk \Cor_{\mathrm{id}_{S^1}}^{(X,\overline{X},Y^*,\overline{Y}^*)}).
\end{equation}
A straightforward computation now shows that the map
\begin{equation}
\begin{aligned}
       \Hom_\coend(\coend\otimes\coend,\unit) \otimes_\kk \Hom_\coend(\unit,\coend\otimes\coend) &\to \Hom_\coend(\unit,\mathbbm{L}) \\
        \phi \otimes_\kk \psi &\mapsto (\phi \otimes \iota_{\coend})\circ (\id_\coend \otimes \lcoev_\coend \otimes \id_\coend) \circ \psi
\end{aligned}
\end{equation}
satisfies \eqref{eq:torusgluecond}.

Putting everything together gives
$\Cor_{T^2} = \iota_{\coend}\circ \lcoev_\coend = \check{\chi}_\coend \in \Hom_{\coend}(\unit,\mathbbm{L})$. Combining \autocite[Cor.\,4.3]{Shimizu15character} and \autocite[Thm.\,4.11]{Shimizu15character} we can express this in terms of characters of simple objects in $\calc$ as follows
\begin{equation}
    \check{\chi}_\coend = \sum_{U,V\in \mathrm{Irr}(\calc)} C_{U,V} \,\check{\chi}_{U^*} \boxtimes \check{\chi}_{V}
\end{equation}
where $C_{U,V}$ is the Cartan matrix of $\calc$, i.e.\ $C_{U,V} = \mathrm{dim}_\kk \Hom_\calc(P_V,P_U)$ with 
$P_S$ the projective cover of the simple $S$. 
With this we reproduce \autocite[Thm.\,3]{FSS13CardyCartan} in the Hopf algebra setting, see also \autocite[Rem.\,7.8]{woike2025constructioncorrelatorsfiniterigid}.

\subsubsection{Line defect action on bulk fields}
Finally, let us discuss how a line defect $D\in \calc$ acts on the bulk fields $\mathbbm{F}_{S^1}$. The action is induced by the world sheet with underlying surface the cylinder $S^1\times I$ and $D$-labelled line defect at $S^1\times \{1/2\}$
\begin{equation}
    \mathfrak{O}_D =\begin{tikzpicture}[baseline=-0.1cm]
        \draw[thick,line cap=round,dashed] (0.5,-1) arc (0:180:0.5 and 0.3);
        \draw[string-violet,thick,line cap=round,dashed] (0.5,0) arc (0:180:0.5 and 0.3);
        \fill[white,opacity=0.7] (-0.5,-1) rectangle (0.5,1);
        \draw[thick,line cap=round] (0,1) ellipse (0.5 and 0.3)
        (0.5,-1) arc (0:-180:0.5 and 0.3);
        \draw[thick,line cap=round] (0.5,-1) --(0.5,1)
        (-0.5,-1) --(-0.5,1);
        \begin{scope}[decoration={
                      markings,
                      mark=at position 0.5 with {\arrow{>}}}
                     ] 
        \draw[string-violet,thick,line cap=round,postaction={decorate}] (-0.5,0) arc (180:360:0.5 and 0.3);
        \end{scope}
        \node at (-0.5,0) [left,string-violet] {\scriptsize $D$};
    \end{tikzpicture}
    \colon 
    \begin{tikzpicture}[baseline=-0.1cm]
        \draw[thick,line cap=round] (0,0) circle (0.4);
    \end{tikzpicture} 
    \longrightarrow
    \begin{tikzpicture}[baseline=-0.1cm]
        \draw[thick,line cap=round] (0,0) circle (0.4);
    \end{tikzpicture}.
\end{equation}
The connecting bordism $M_{\frakO_D}^{(X,\overline{X},Y,\overline{Y})}$ for this world sheet is given by 
\begin{equation}
            \begin{tikzpicture}
    \draw[thick,loosely dotted,line cap=round](0:2.5 and 0.9) arc (0:180:2.5 and 0.9);
    \fill[white,opacity=0.7] (0,0) circle (2.5);
    \draw[thick,loosely dotted,line cap=round](0:0.5 and 0.2) arc (0:180:0.5 and 0.2)
    ([yshift=-1.5cm]0:0.5 and 0.2) arc (0:180:0.5 and 0.2)
    ([yshift=1.5cm]0:0.5 and 0.2) arc (0:180:0.5 and 0.2);
    \draw[thick,string-violet,line cap=round,->](1.3,0) arc (0:180: 1.3 and 0.575);
    \fill[white,opacity=0.7](0,0) circle (0.5)
    (0,1.5) circle (0.5)
    (0,-1.5) circle (0.5);
    \draw[white,ultra thick] (0,1) -- (0,0.5);
    \draw[thick] (0,0) circle (0.5)
    (0,1.5) circle (0.5)
    (0,-1.5) circle (0.5);
    \draw[thick,line cap=round](0:0.5 and 0.2) arc (0:-180:0.5 and 0.2)
    ([yshift=-1.5cm]0:0.5 and 0.2) arc (0:-180:0.5 and 0.2)
    ([yshift=1.5cm]0:0.5 and 0.2) arc (0:-180:0.5 and 0.2);
    \begin{scope}[thick,decoration={
                      markings,
                      mark=at position 0.57 with {\arrow{>}}}
                     ]  
        \draw[string-blue,thick,postaction={decorate}] (0,1) -- (0,0.5);
        \draw[string-blue,thick,postaction={decorate}] (0,2) -- (0,2.5);
        \draw[string-blue,thick,postaction={decorate}] (0,-0.5) -- (0,-1);
        \draw[string-blue,thick,postaction={decorate}] (0,-2.5) -- (0,-2);
    \end{scope}
    \draw[ultra thick,white](1.3,0) arc (0:-180:1.3 and 0.575);
    \draw[thick,string-violet,line cap=round,<-](1.3,0) arc (0:-180: 1.3 and 0.575);
    \draw[ultra thick,white](2.5,0) arc (0:-180:2.5 and 0.9);
    \draw[thick] (0,0) ellipse (2.5)
    (2.5,0) arc (0:-180:2.5 and 0.9);
     \fill[string-blue]
     (0,2.5) circle (0.04)
     (0,2) circle (0.04)
     (0,1) circle (0.04)
     (0,0.5) circle (0.04)
     (0,-0.5) circle (0.04)
     (0,-1) circle (0.04)
     (0,-2) circle (0.04)
     (0,-2.5) circle (0.04);
    \node at (0,2.25) [string-blue,right] {\scriptsize $\overline{Y}$};
    \node at (0,0.75) [string-blue,right] {\scriptsize $Y$};
    \node at (0,-0.72) [string-blue,right] {\scriptsize $X$};
    \node at (0,-2.25) [string-blue,right] {\scriptsize $\overline{X}$};
    \node at (1.32,-0.1) [string-violet,right] {\scriptsize $D$};
     %
     %
     %
     \begin{scope}[shift={(4.5,0)}]
    \draw[thick] (0.9,0) circle (0.5)
    (-0.9,0) circle (0.5);
    \draw[thick,loosely dotted,line cap=round](0.4,0) arc (180:0:0.5 and 0.2)
    (-0.4,0) arc (0:180:0.5 and 0.2); 
    \fill[white,opacity=0.7] (-0.4,0) arc (0:180:0.5)
    (1.4,0) arc (0:180:0.5);
    \draw[thick] (-0.4,0) arc (0:180:0.5)
    (1.4,0) arc (0:180:0.5)
    (0.4,0) arc (180:360:0.5 and 0.2)
    (-0.4,0) arc (0:-180:0.5 and 0.2);
    \draw[thick] (-0.4,0) arc (0:180:0.5);
    \fill[string-blue] ([xshift=0.9cm]60:0.5) circle (0.04)
    ([xshift=-0.9cm]240:0.5) circle (0.04)
    ([xshift=0.9cm]120:0.5) circle (0.04)
    ([xshift=-0.9cm]300:0.5) circle (0.04);
     \node at ([xshift=0.9cm]70:0.4) [string-blue,above right] {\scriptsize $(\overline{Y},+)$};
     \node at ([xshift=0.9cm]110:0.4) [string-blue,above left] {\scriptsize $(Y,+)$};
     \node at ([xshift=-0.9cm]300:0.4) [string-blue,below right] {\scriptsize $(\overline{X},-)$};
     \node at ([xshift=-0.9cm]250:0.4) [string-blue,below left] {\scriptsize $(X,-)$};
     \end{scope}
     %
     %
     %
    \begin{scope}[shift={(8.8,0)}]
    \draw[thick] (0.5,0) circle (0.5)
    (-0.9,0) circle (0.5);
    \draw[thick,loosely dotted,line cap=round](0,0) arc (180:0:0.5 and 0.2)
    (-0.4,0) arc (0:180:0.5 and 0.2); 
    \fill[white,opacity=0.7] (-0.4,0) arc (0:180:0.5)
    (1,0) arc (0:180:0.5);
    \draw[thick] (-0.4,0) arc (0:180:0.5)
    (1,0) arc (0:180:0.5)
    (0,0) arc (180:360:0.5 and 0.2)
    (-0.4,0) arc (0:-180:0.5 and 0.2);
    \draw[thick] (-0.4,0) arc (0:180:0.5);
    \fill[string-blue](-0.9,-0.5) circle (0.04)
     (-0.9,0.5) circle (0.04)
     (0.5,-0.5) circle (0.04)
     (0.5,0.5) circle (0.04);
     \node at (-0.9,-0.45) [string-blue,below] {\scriptsize $(X,-)$};
     \node at (-0.9,0.5) [string-blue,above] {\scriptsize $(Y,+)$};
     \node at (0.5,-0.45) [string-blue,below] {\scriptsize $(\overline{X},-)$};
     \node at (0.5,0.5) [string-blue,above] {\scriptsize $(\overline{Y},+)$};
    \end{scope}
    \draw[thick,->,line cap=round] (6.1,0) -- (7.2,0);
    \fill[black] (2.7,-0.07) circle (0.025)
    (2.7,0.07) circle (0.025);
    \end{tikzpicture}.
\end{equation}
By functoriality of $\widehat{\rmV}_\calc$ the linear map $\widehat{\rmV}_\calc\left(M_{ _D}^{(X,\overline{X},Y,\overline{Y})}\right)$ can be obtained by composing the maps $\widehat{\rmV}_\calc\left(M_{\id_{S^1}}^{(X,\overline{X},Y,\overline{Y})}\right)$ and $\widehat{\rmV}_\calc\left(M_{\frakD^{\mathrm{out}}_D}^{(Y,\overline{Y})}\right)$ from Sections \ref{subsec:bulk2point} and \ref{subsec:boundarystates}. In particular, the natural transformation induced by $\widehat{\rmV}_\calc\left(M_{\frakO_N}^{(X,\overline{X},Y,\overline{Y})}\right)$ is the composition
\begin{equation}
    \Cor_{\frakO_D} = \epsilon_{\Cor_{S^1}} \circ (\id_{\Cor_{S^1}^\dagger}\otimes_\kk \Cor_{\frakD^{\mathrm{out}}_D}) \colon \Cor_{S^1}^\dagger \otimes_\kk \Cor_{S^1} \Rightarrow \Bl_\calc(\frakO_D).
\end{equation}
Next note that by the Yoneda lemma we have $\mathrm{Nat}(\Cor_{S^1}^\dagger \otimes_\kk \Cor_{S^1}, \Bl_\calc(\frakO_D)) \cong \Hom_\coend(\mathbbm{F}_{S^1},\mathbbm{F}_{S^1})$, so instead of computing the corresponding natural transformation we will focus on the endomorphism of $\mathbbm{F}_{S^1} \cong \coend$ instead. We will call this endomorphism the \emph{defect operator} associated to $D$ and denote it with $\mathcal{O}_D$. To compute $\mathcal{O}_D$ first recall from \Cref{rem:unitality} that $\epsilon_{\Cor_{S^1}}$ gets sent to $\id_{\coend} \in \Hom_\coend(\coend,\coend)$.
In the discussion in \Cref{subsec:boundarystates} we saw that $\Cor_{\frakD_D^{\mathrm{out}}}$ corresponds to $\mu \circ (\check{\chi}_D \otimes\id_\coend) \colon \coend \to \coend$. Combining this we find
\begin{equation}
    \mathcal{O}_D = \mu \circ (\check{\chi}_D \otimes\id_\coend) \in \Hom_\coend(\coend,\coend).
\end{equation}
As a self-consistency check note that we also could have decomposed $M_{\frakO_N}^{(X,\overline{X},Y,\overline{Y})}$ using $M_{\frakD_D^{\mathrm{in}}}^{(X,\overline{X})}$ instead of $M_{\frakD_D^{\mathrm{out}}}^{(Y,\overline{Y})}$. By doing this we obtain the endomorphism $((\chi_D\circ \calS) \otimes \id_\coend)\circ \Delta_{\Lambda}$ where $\Delta_{\Lambda}$ is the Frobenius coalgebra structure on $\coend$ from \Cref{prop:frobpairingcoend} and $\calS$ is the modular S-transformation from \Cref{eq:def-S-endo-of-L}. Using the relation between characters, cocharacters, and the modular S-transformation from \autocite[Sec.\,2.4]{FGSS18cardy} a direct calculation gives $\mu \circ (\check{\chi}_D \otimes\id_\coend) = ((\chi_D\circ \calS) \otimes \id_\coend)\circ \Delta_{\Lambda}$. Thus both decompositions of $M_{\frakO_N}^{(X,\overline{X},Y,\overline{Y})}$ lead to the same endomorphism of $\coend$. 

\begin{remark}
    In the semisimple setting we have $\coend \cong \bigoplus_{i\in \mathrm{Irr}(\calc)} i^*\boxtimes i$ as objects in $\calc_\coend \simeq \calc\boxtimes\overline{\calc}$ where $\mathrm{Irr}(\calc)$ is a set of representatives of simple objects in $\calc$. For $j$ a simple object the defect operator $\mathcal{O}_j$ acts on $L$ by multiplication with the (normalised) $S$-matrix element $S_{j,i^*}/S_{0,i^*}$, see e.g.\ \autocite[Sec.\,6.2]{FFRS2007defects}.
\end{remark}

We can compose two line defects by bringing them close to each other. To be more precise for a second line defect labelled by $E\in \calc$ the composed line defect is labelled by $E\otimes D$. This is compatible with composition of endomorphisms since $ \mathcal{O}_E \circ  \mathcal{O}_D =  \mathcal{O}_{E\otimes D}$ because $\mu \circ (\check{\chi}_E \otimes \check{\chi}_D) = \check{\chi}_{E\otimes D}$, see \autocite[Thm.\,3.10]{Shimizu15character} or \autocite[Sec.\,2.3]{FGSS18cardy}. Finally, by reformulating \autocite[Cor.\,4.3]{Shimizu15character} we find that the defect operators span the Grothendieck ring of $\calc$: 
\begin{proposition}\label{prop:algofdefs}
    The subalgebra of $\mathrm{End}_\coend(\coend)$ generated by the defect operators $(\mathcal{O}_D)_{D\in \calc}$ is isomorphic (as an algebra) to the linearised Grothendieck ring $\mathrm{Gr}_\kk(\calc)$.
\end{proposition}

Combining the above proposition with \autocite[Cor.\,8.7]{GR17projmod}, which states that $\mathrm{Gr}_\kk(\calc)$ is semisimple as a $\kk$-algebra if and only if $\calc$ is semisimple, gives:

\begin{corollary}
    The subalgebra of $\mathrm{End}_\coend(\coend)$ generated by defect operators is semisimple if and only if $\calc$ is semisimple.
\end{corollary}

As an example, consider the non-semisimple modular tensor category $\mathcal{SF}(N)$ of $N \in \mathbb{Z}_{>0}$ pairs of symplectic fermions \autocite{Run12freesupbos,DR12extHopf,FGR17sympfermi}. It has 4 simple objects $\unit$, $\Pi\unit$, $T$, $\Pi T$, and hence $\mathrm{Gr}_\mathbb{C}(\mathcal{SF}(N))$ is a 4-dimensional $\mathbb{C}$-algebra with basis $[\unit]$, $[\Pi\unit]$, $[T]$, $[\Pi T]$. The product is given by
\begin{align}
        &[\Pi\unit] \cdot [\Pi\unit] = [\unit] ~~,\quad
        [\Pi\unit] \cdot [T] = [\Pi T] ~,
        \nonumber \\
        &[T] \cdot [T] = [T] \cdot [\Pi T] = 2^{2N-1} ([\unit] + [\Pi\unit]) \ .
\end{align}
This algebra is indeed not semisimple, for example the element $n = [T]-[\Pi T]$ satisfies $n^2=0$. Thus in a faithful representation, $n$ is non-diagonalisable, and hence one of $[T]$ and $[\Pi T]$ (and hence both) are non-diagonalisable.

By Proposition \ref{prop:algofdefs}, the defect operators $\mathcal{O}_X$, $X \in \{\unit, \Pi\unit, T, \Pi T \}$ satisfy the same composition rules. In particular, the action of $\mathcal{O}_T$ on the space of bulk fields cannot be diagonalised.

\addcontentsline{toc}{section}{References}
\sloppy
\printbibliography
\end{document}